\documentclass[12pt]{amsart}
\usepackage{fullpage,url,amssymb}
\usepackage{enumitem, graphicx}
\usepackage[all]{xy} 

\usepackage[OT2,OT1]{fontenc}
\usepackage[latin1]{inputenc}
\usepackage{amsmath}
\usepackage{amsthm}
\usepackage{amssymb}
\usepackage{hyperref}

\font\elevensc=cmcsc10 scaled \magstephalf

\newcommand{\auau}{{$\hhb{-2}_{\hbox{\large\~{}}}\hhb{-1}$}}
\newcommand{\dwa}{\downarrow}
\newcommand{\dwn}[1]{\big\downarrow\rlap{$\vcenter{\hbox{$\scriptstyle{#1}$}}$}}

\newcommand{\hhb}[1]{{\hbox to#1pt{}}}

\newcommand{\hor}[1]{\smash
         {\mathop{{\lgrghtar}}\limits^{\lower2pt\hbox{$\scriptstyle{#1}$}}}}
\newcommand{\horr}[1]{\smash
         {\mathop{{\lglgrghtar}}\limits^{\lower2pt\hbox{$\scriptstyle{#1}$}}}}
\newcommand{\hr}[1]{\smash
         {\mathop{{\to}}\limits^{\lower2pt\hbox{$\scriptstyle{#1}$}}}}
\newcommand{\shr}[1]{\smash
         {\mathop{{\to}}\limits^{\lower2pt\hbox{$\scriptscriptstyle{#1}\hp$}}}}

\newcommand{\lgrghtar}{{\hhb2{\relbar\joinrel\rightarrow}\hhb2}}
\newcommand{\lglgrghtar}{{\hhb1{\relbar\joinrel\relbar\joinrel\rightarrow}\hhb1}}

\newcommand{\ilim}[1]{\hboxto14pt{%
lim\kern-14pt\lower4.5pt\hbox{$\scriptstyle\longrightarrow$}%
\kern-8pt\lower5pt\hbox{$\scriptstyle{#1}$}}\hhb{3}}

\newcommand{\ilm}[1]{\hboxto14pt{%
lim\kern-14pt\lower5pt\hbox{$\scriptstyle\longrightarrow$}%
\kern-8pt\lower4pt\hbox{$\scriptstyle{#1}$}}\hhb{3}}

\newcommand{\plim}[1]{\hbox to14pt{%
lim\kern-14pt\lower4.5pt\hbox{$\scriptstyle\longleftarrow$}%
\kern-8pt\lower8.5pt\hbox{$\scriptstyle{#1}$}}\hhb{3}}

\newcommand{\plm}[1]{\hbox to14pt{%
lim\kern-14pt\lower5pt\hbox{$\scriptstyle\longleftarrow$}%
\kern-8pt\lower4pt\hbox{$\scriptstyle{#1}$}}\hhb{3}}

\newcommand{\ratmap}{\dashrightarrow}

\newcommand{\rsdp}{{\,\times\kern-3pt\lower-1pt%
\hbox{$\scriptscriptstyle|$\hbp3}}}

\newcommand{\srjr}{\twoheadrightarrow}

\newcommand{\upa}{\uparrow}

\newcommand{\defi}[1]{\textsf{#1}}

\newcommand{\nix}{{\phantom{|}}}

\newcommand{\nmnm}[2]{{\scalebox{.9}[1]{\textsc{#1}}\scalebox{.8}[.9]{\textsc{#2}}}}

\newcommand{\clA}{{\mathcal A}}

\newcommand{\clD}{{\mathcal D}}

\newcommand{\clG}{{\mathcal G}}
\newcommand{\clH}{{\mathcal H}}
\newcommand{\clI}{{\mathcal I}}

\newcommand{\clL}{{\mathcal L}}

\newcommand{\clN}{{\mathcal N}}
\newcommand{\clO}{{\mathcal O}}
\newcommand{\clP}{{\mathcal P}}
\newcommand{\clQ}{{\mathcal Q}}
\newcommand{\clR}{{\mathcal R}}
\newcommand{\clS}{{\mathcal S}}
\newcommand{\clT}{{\mathcal T}}
\newcommand{\clU}{{\mathcal U}}
\newcommand{\clV}{{\mathcal V}}
\newcommand{\clW}{{\mathcal W}}
\newcommand{\clX}{{\mathcal X}}
\newcommand{\clY}{{\mathcal Y}}
\newcommand{\clZ}{{\mathcal Z}}

\let\lv=\mathbb

\newcommand{\lvA}{{\mathbb A}}

\newcommand{\lvF}{{\mathbb F}}
\newcommand{\lvG}{{\mathbb G}}

\newcommand{\lvN}{{\mathbb N}}

\newcommand{\lvP}{{\mathbb P}}
\newcommand{\lvQ}{{\mathbb Q}}

\newcommand{\lvT}{{\mathbb T}}

\newcommand{\lvZ}{{\mathbb Z}}
\newcommand{\euA}{{\mathfrak A}}
\newcommand{\eua}{{\mathfrak a}}

\newcommand{\euD}{{\mathfrak D}}

\newcommand{\euF}{{\mathfrak F}}

\newcommand{\eul}{{\mathfrak l}}

\newcommand{\eum}{{\mathfrak m}}

\newcommand{\euP}{{\mathfrak P}}
\newcommand{\eup}{{\mathfrak p}}

\newcommand{\euR}{{\mathfrak R}}

\newcommand{\euT}{{\mathfrak T}}

\newcommand{\euv}{{\mathfrak v}}

\newcommand{\euw}{{\mathfrak w}}

\DeclareMathOperator{\chr}{char}
\DeclareMathOperator{\dvv}{div}

\DeclareMathOperator{\End}{End}
\DeclareMathOperator{\Gal}{Gal}
\DeclareMathOperator{\GL}{GL}
\DeclareMathOperator{\id}{id} %
\DeclareMathOperator{\im}{im} %
\DeclareMathOperator{\Spec}{Spec}

\newcommand{\Aut}[2]{{\rm Aut}{#1}\hhb{-1}\big(#2\big)}

\newcommand{\bfv}{{\euv}}
\newcommand{\bfw}{{\euw}}
\newcommand{\bfvL}{{\euv_{\!L}}}

\newcommand{\bfva}{{\euv_\alpha}}
\newcommand{\bsl}[2]{{{#1}\backslash{#2}}}

\newcommand{\compreglike}{complete regular like}

\newcommand{\Div}{{\rm Div}}
\newcommand{\DD}[1]{{D\hhb{-.5}_{#1}}}
\newcommand{\dvdv}{{\rm div}}

\newcommand{\euCl}{{\mathfrak{Cl}}}
\newcommand{\eudim}{{\mathfrak{dim}}}
\newcommand{\eutd}{{\mathfrak{td}}}

\newcommand{\fda}[1]{{#1'}}
\newcommand{\fdc}[1]{#1^{{\scriptstyle c^{\!\!\phantom|}}}}

\newcommand{\Ggr}[1]{\Gal_{#1}}

\newcommand{\Gm}{{\lv G}_m}

\newcommand{\hm}[1]{\hhb{-#1}}
\newcommand{\hmm}{\hhb{-2.25}}
\newcommand{\hp}{\hhb1}
\newcommand{\hpp}{\hhb2}
\newcommand{\hla}{\hookleftarrow}
\newcommand{\hra}{\hookrightarrow}

\newcommand{\Inr}{{\mathfrak{In}}}
\newcommand{\Inrtm}{{\mathfrak{In.tm}}}
\newcommand{\Inrdiv}{{\mathfrak{In.div}}}

\newcommand{\Inrqdiv}{{\mathfrak{In.q.div}}}
\newcommand{\Inrtmqdiv}{{\mathfrak{In.tm.q.div}}}

\newcommand{\Isom}{{\rm Isom}}

\newcommand{\imm}{{\rm im}}
\newcommand{\itm}[2]{\begin{itemize}[leftmargin=#1pt]
                     #2\end{itemize}}

\newcommand{\jmthim}[1]{{\jmath_{#1}(#1\tms\hhb1)}}
\newcommand{\jmtha}[1]{{\jmath_{#1}:#1\tms\to\whgen{#1}}}

\newcommand{\Kv}{{K\hhb{-1.5}v}}  
\newcommand{\Kbfv}{{K\!\bfv}}  
  
\newcommand{\Kbfw}{{K\hhb{-1}\bfw}}  
\newcommand{\Ktlv}{{K\hhb{-1}\tilde v}}  
\newcommand{\Ktlbfv}{{K\hhb{-1}\tlbfv}}  
\newcommand{\Ktlbfvm}{{K\hhb{-1}\tlbfvm}}  
\newcommand{\Ktlw}{{K\hhb{-1.5}\tilde w}}  
\newcommand{\Ktlbfw}{{K\hhb{-1.5}\tlbfw}}

\newcommand{\kr}{{k\bfw}}
\newcommand{\Kr}{\Kbfw}
\newcommand{\lr}{{l\bfw}}
\newcommand{\kone}{{{\textsl{\textsf{k}}}_{\scriptscriptstyle1}}}

\newcommand{\kpu}{\kappa_u}
\newcommand{\kpx}{{\kappa_x}}

\newcommand{\kpy}{{\kappa_y}}

\newcommand{\Lbfw}{{L\hm{.125}\bfw}}

\newcommand{\lxzl}[1]{\clL_{#1}}

\newcommand{\md}{\hhb1|\hhb1}

\newcommand{\mT}{T^1}
\newcommand{\mZ}{Z^1}

\newcommand{\nx}[1]{{#1^{^{\rm n}}}}

\newcommand{\oli}{\overline}
\newcommand{\ot}{\leftarrow}

\newcommand{\Pic}{{\rm Pic}}

\newcommand{\pel}{{(\ell)}}

\newcommand{\pia}[1]{\Pi_{#1}}
\newcommand{\pic}[1]{\Pi^{{\scriptstyle c^{\phantom|}}}_{#1}}

\newcommand{\qtottrK}{\clQ^{\rm tot}_{K}}
\newcommand{\qtottrL}{\clQ^{\rm tot}_{L}}

\newcommand{\shra}[1]{\!\!\!\hookrightarrow\!\!\!}
\newcommand{\shor}[1]{\!\!\!\!\!\hor{#1}\!\!\!\!\!}
\newcommand{\shorr}[1]{\!\!\!\!\!\horr{#1}\!\!\!\!\!}

\newcommand{\spsp}{{\rm sp}}
\newcommand{\str}[2]{{{}^*\hhb{-#1}#2}}

\newcommand{\TT}{T}
\newcommand{\td}[2]{{{\rm td}\hhb{.5}(#1\hhb{.5}|\hhb{.5}#2)}}
\newcommand{\tl}[1]{{\tilde{#1}}}
\newcommand{\tlv}{{\tilde v}}
\newcommand{\tlw}{{\tilde w}}

\newcommand{\tms}{^\times\!}
\newcommand{\tlbfv}{{\hhb3\tilde{}\hhb{-3}\bfv}}
\newcommand{\tlbfw}{{\hhb{4.5}\tilde{}\hhb{-4.5}\bfw}}
\newcommand{\tlbfvm}{{\hhb2\tilde{}\hhb{-2}\bfv}}
\newcommand{\tlbfwm}{{\hhb{3}\tilde{}\hhb{-3}\bfw}}

\newcommand{\tottrK}{{\clD^{\rm tot}_{\hhb{-2}K}}}
\newcommand{\tottrL}{{\clD^{\rm tot}_L}}
\newcommand{\trK}{\clD_{\!K}}

\newcommand{\trL}{\clD_{\hm1 L}}

\newcommand{\uux}{{{\textsf u}}}

\newcommand{\Val}[1]{{\rm Val}_{#1}}

\newcommand{\vcr}[1]{v_{\hm1\scriptscriptstyle#1}}

\newcommand{\vid}{{\lower-1pt\hbox{/}\kern-7pt{\hbox{O}}}}

\newcommand{\whkpp}[1]{\widehat{\hhb1
  \kappa_{\lower2pt\hbox{$\scriptstyle#1$}}\hhb2}}
\newcommand{\whgen}[1]{\hhb{1}{\lower.5pt\hbox{$
    \widehat{\phantom{K}}$}\hhb{-11.75}#1}}
\newcommand{\whgne}[1]{\hhb5{\lower1pt\hbox{$
    \widehat{\phantom|\hhb{-8.5}}$}}}
\newcommand{\whgngn}[1]{\hhb{1}{\lower.5pt\hbox{$
    \widehat{\phantom{\Kv}}$}\hhb{-15}#1}}

\newcommand{\whCl}{\hhb{1}{\lower.5pt\hbox{$
    \widehat{\phantom{\euCl}}$}\hhb{-10}\euCl}}
\newcommand{\whDiv}{{\widehat{\rm Div}}}
\newcommand{\whK}{\hhb{1}{\lower.5pt\hbox{$
    \widehat{\phantom{K}}$}\hhb{-11.75}K}}
\newcommand{\whL}{\hhb{1}{\lower.5pt\hbox{$
    \widehat{\phantom{L}}$}\hhb{-9}L}}

\newcommand{\whU}{\hhb{1}{\lower.5pt\hbox{$
    \widehat{\phantom{U}}$}\hhb{-9.5}U}}
\newcommand{\whclL}{{\widehat\clL}}
\newcommand{\whKfin}{\hhb{1}{\lower.5pt\hbox{$
    \widehat{\phantom{K}}$}\hhb{-12}K_{\rm fin}}}
\newcommand{\whKbfvfin}{{\whgngn{\Kbfv_{\rm fin}}}}
\newcommand{\whKbfv}{{\whgngn{\Kbfv}}}
\newcommand{\whLfin}{{\whgen{L_{\rm fin}}}}

\newcommand{\whlalp}{{\widehat\clL_\alpha}}

\newcommand{\whph}{{\hat\phi}}

\newcommand{\wtl}[1]{{\widetilde{#1}}}

\newcommand{\XX}{{X}}

\newcommand{\xx}{{{\textsf x}}}

\newcommand{\yx}{{{\textsf y}}}

\newcommand{{\Yi}}{Y^{\rm i}}
\newcommand{{\Ys}}{Y^{\rm s}}
\newcommand{\YY}{Y}
\newcommand{\ZZ}{Z}
\newcommand{\Zell}{\lvZ_\ell}
\newcommand{\Zetell}{\lvZ_\pel}

\newcommand{\zX}{V}
\newcommand{\zY}{W}
\newcommand{\zclX}{\clV}
\newcommand{\zclY}{\clW}

\newtheorem{thm}{Theorem}[section]
\newtheorem{lem}[thm]{Lemma}
\newtheorem{prop}[thm]{Proposition}

\theoremstyle{definition}

\newtheorem{constr}[thm]{Construction}

\newtheorem{definition}[thm]{Definition}
\newtheorem{definition/remark}[thm]{Definition/Remark}

\newtheorem{example/fact}[thm]{Example/Fact}
\newtheorem{fact}[thm]{Fact}
\newtheorem{notation}[thm]{Notation}

\newtheorem{fact/definition}[thm]{Fact/Definition}
\newtheorem{fact/nota}[thm]{Fact/Notations}
\newtheorem{remark}[thm]{Remark}
\newtheorem{prepa/nota}[thm]{Preparation/Notations}
\newtheorem{remark/definition}[thm]{Remark/Definition}
\newtheorem{remarks/notations}[thm]{Remarks/Notations}
\newtheorem{remarks}[thm]{Remarks}
\title[On Bogomolov's birational 
         anabelian program]
         {\normalsize Reconstruction of Function Fields from their
         \vskip4pt
         pro-$\ell$ abelian divisorial inertia
         \vskip6pt
{\scriptsize\hhb2 --- with Applications to Bogomolov's Program and I/OM ---\hhb2}}

\author{Florian Pop}
\begin{document}

\begin{abstract}
Let $\pic K\to\pia K$ be the maximal pro-$\ell$ 
abelian-by-central, respectively abelian, Galois groups
of a function field $K|k$ with $k$ algebraically closed 
and ${\rm char}\neq\ell$. We show that $K|k$ can be 
functorially reconstructed by group theoretical recipes 
from $\pic K$ endowed with the set of divisorial inertia 
$\Inrdiv(K)^\nix\subset\pia K$. As applications, one has:
(i)~A group theoretical recipe to reconstruct $K|k$ from 
$\pic K$, provided either $\td K k\!>{\rm dim}(k)\!+\!1$ or
$\td K k >{\rm dim}(k)>1$, where ${\rm dim}(k)$ is 
the Kronecker dimension; 
(ii)~An application to the {\it pro-$\ell\!$ 
abelian-by-central\/} I/OM (Ihara's question / 
Oda-Matsumoto conjecture), which in the 
cases considered here implies the classical I/OM.
\end{abstract}



\thanks{Partially supported by the Simons Grant~342120.}
\date{Variant of July 4, 2018}

\maketitle

\section{Introduction}


\vskip3pt
\noindent
We introduce notations which will be used throughout
the manuscript as follows: For a fixed prime number 
$\ell$, and function fields $K|k$ over algebraically closed 
base fields $k$, we denote by 
\[
\pic K:={\rm Gal}(\fdc K|K)\to{\rm Gal}(\fda K|K)=:\pia K
\]
the Galois groups of the maximal \defi{pro-$\ell$ 
abelian-by-central}, respectively \defi{pro-$\ell$ abelian} 
extensions $\fdc K|K \hla \fda K|K$ of $K$ (in some 
fixed algebraic closure, which is not necessary to be 
explicitly specified). Notice that $\pic K\to\pia K$ 
viewed as a morphism of abstract profinite groups can 
be recovered by a {\it group theoretical recipe\/} from 
$\pic K$, because $\ker(\pic K\to\pia K)=[\pic K,\pic K]$.
\vskip5pt
\noindent
{\bf Hypothesis.} If not otherwise specified, we will also 
suppose that $\chr(k)\neq\ell$.
\vskip5pt
The present work concerns a program initiated by 
\nmnm{B}{ogomolov}~[Bo] at the beginning of the 
1990's, which aims at giving group theoretical recipes 
to reconstruct $K|k$ from $\pic K$ in a functorial 
way, provided $\td K k >1$. Notice that in contrast 
to Grothendieck's (birational) anabelian philosophy, 
where a rich arithmetical Galois action is part of the 
picture, in Bogomolov's program there is \underbar{no} 
arithmetical action (because the base field $k$ is 
algebraically closed). Note also that the hypothesis 
$\td K k >1$ is necessary, because $\td K k =1$ 
implies that the absolute Galois group $G_K$ is 
profinite free on $|k|$ generators (as shown by 
\nmnm{D}{ouady}, \nmnm{H}{arbater}, \nmnm{P}{op}), 
and therefore, $\pic K$ encodes only the cardinality 
$|k|$ of the base field $|k|$, hence that of $K$ as well.
\vskip3pt
A strategy to tackle Bogomolov's program was 
proposed by the author in 1999, and along the 
lines of that strategy, the Bogomolov program was 
completed for $k$ an algebraic closure of a finite 
field by \nmnm{B}{ogomolov}--\nmnm{T}{schinkel}~[B--T] 
and \nmnm{P}{op}~[P4].
\vskip3pt
In order to explain the results of this note, we recall a few
basic facts as follows.
\vskip3pt
For a further function field $L|l$ with $l$ algebraically 
closed, let $\Isom^{\rm F}(L,K)$ be the set of the 
isomorphisms of the pure inseparable closures 
$L^{\rm i}\to K^{\rm i}$ up to Frobenius twists. Further, let 
${\rm Isom}^{\rm c}(\pia K,\pia L)$ be the set of the 
abelianizations $\Phi:\pia K\to\pia L$ of the isomorphisms 
$\pic K\to\pic L$ modulo multiplication by $\ell$-adic units. 
Notice that given $\phi\in\Isom^{\rm F}(L,K)$, and any
prolongation ${\fda\phi}:{L'}^{\rm }\to{K'}^{\rm }$ 
of $\phi$ to ${L'}^{\rm }$, one gets $\Phi_\phi\in\Isom^{\rm c}
(\pia K,\pia L)$ defined by $\Phi_\phi(\sigma):=
{\fda\phi}^{-1}\circ\sigma\circ\fda\phi$ for $\sigma\in\pia K$.
Noticing that $\Phi_\phi$ depends on $\phi$ only, and not on 
the specific prolongation~$\fda\phi$, one finally gets a 
canonical embedding
\[
\Isom^{\rm F}(L,K)\to{\rm Isom}^{\rm c}(\pia K,\pia L), 
\quad \phi\mapsto\Phi_\phi.
\] 

Next recall that given a function field $K|k$ as above,  
the \defi{prime divisors} of $K|k$ are the valuation~$v$ 
of $K|k$ defined by the Weil prime divisors of the 
normal models $X$ of $K|k$. It turns out that a 
valuation $v$ of $K$ is a prime divisor if and only 
if $v$ is trivial on $k$ and its residue field $\Kv$ satisfies 
$\td K k =\td \Kv k +1$. For any valuation $v$ of $K$ 
we denote by $T_v\subseteq Z_v\subset\pia K$ the 
inertia/decomposition groups of some prolongation $v$ 
to $\fda K$ (and notice that $T_v\subset Z_v$ depend 
on $v$ only, because $\pia K$ is abelian). We denote by 
$\clD_{K|k}$ the set of all the prime divisors of $K|k$, and
consider the set of \defi{divisorial inertia} in $\pia K$
\[
\Inrdiv(K):=\cup_{v\in\clD_{K|k}} T_v\subset\pia K.
\]

Finally recall that a natural generalization of prime 
divisors are the \defi{quasi prime divisors} of $K|k$. 
These are the valuations $\bfv$ of $K|k$ (not necessarily 
trivial on~$k$), minimal among the valuations of $K$ 
satisfying: \ i)~$\td Kk=\td\Kbfv{k\bfv}+1$; \
ii)~$\bfv K/\bfv k\cong\lvZ$.  [Here, the minimality 
of $\bfv$ means (by definition) that if $w$ is a 
valuation of $K$ satisfying~i),~ii), and the valuation 
rings satisfy $\clO_{w}\supseteq\clO_\bfv$, then 
$w=\bfv$.] In particular, the prime divisors of $K|k$ 
are precisely the quasi prime divisors of $K|k$ that 
are trivial on $k$. For quasi prime divisors~$\bfv$, 
let $\mT_\bfv\subset\mZ_\bfv\subset\pia K$
be their \defi{minimized} inertia/decomposition
groups, see~Section~2,~A), and/or~\nmnm{P}{op}~[P5] 
and \nmnm{T}{opaz}~[To2], for definitions. We denote 
by $\clQ_{K|k}$ the set of quasi prime divisors of $K|k$,
and recall that $\mT_\bfv=T_\bfv$, $\mZ_\bfv=Z_\bfv$ 
are the usual inertia/decomposition groups if 
${\rm char}(\Kbfv)\neq\ell$. By \nmnm{P}{op}~[P1] and 
\nmnm{T}{opaz}~[T1], there are ({\it concrete uniform\/}) 
{\it group theoretical recipes\/} such that given $\pic K$,
one can recover the (minimized) quasi divisorial groups 
groups $\mT_\bfv\subset\mZ_\bfv\subset\pia K$, thus 
the set of \defi{(minimized) quasi-divisorial inertia} 
\[
\Inrqdiv^1(K):=\cup_{\bfv\in\clQ_{K|k}}\mT_\bfv\subset\pia K.
\] 
Moreover, the recipes to do so are invariant under 
isomorphisms, i.e., in the above notations,  every 
$\Phi\in {\rm Isom}^{\rm c}(\pia K,\pia L)$ has the property 
that  $\Phi\big(\Inrqdiv^1(K)\big)=\Inrqdiv^1(L)$.
\vskip3pt
Unfortunately, for the time being, in the case $k$ 
is an arbitrary algebraically closed base field, one does 
not know group theoretical recipes neither to distinguish 
the divisorial subgroups among the quasi divisorial 
ones, nor to describe/recover $\,\Inrdiv(K)=\cup_v T_v$ 
inside $\Inrqdiv^1(K)=\cup_\bfv\mT_\bfv$,  using 
solely the group theoretical information encoded in 
$\pic K$. (Actually, the latter apparently weaker 
question turns out to be equivalent to the former one.)
\vskip2pt
The main result of this note reduces the Bogomolov 
program about reconstructing function fields $K|k$ 
with $\td Kk>2$ from the group theoretical information 
encoded in $\pic K$ to giving group theoretical recipes 
such that given $\pic K$ one can answer the following:
\vskip7pt
\centerline{{\it \defi{Recover $\,\Inrdiv(K)=\cup_v T_v\,$ 
inside $\,\Inrqdiv^1(K)=\cup_\bfv \mT_\bfv\,$ 
using the pro-$\ell$ group $\,\pic K$.}\/}}
\vskip7pt
\noindent
The precise result reads as follows:
\begin{thm}
\label{mainthm}
In the above notations, there exists a group theoretical recipe
about pro-$\ell$ abelian-by-central groups such that the 
following hold:
%
\begin{itemize}[leftmargin=30pt]
\item[{\rm1)}] The recipe reconstructs $K|k$ from 
$\pic K$ endowed with~$\Inrdiv(K)\subset\pia K$ in a functorial 
way, provided $\td K k >2$.
\vskip4pt
\item[{\rm2)}]  The recipe is invariant under isomorphisms 
$\Phi\in {\rm Isom}^{\rm c}(\pia K,\pia L)$ in the sense that  
if $\Phi\big(\Inrdiv(K)\big)=\Inrdiv(L)$, then $\Phi=\Phi_\phi$ 
for a $($\hhb{-1}unique$)$ $\phi\in{\rm Isom}^{\rm F}(L,K)$. 
\end{itemize}
\end{thm}
The proof of the above Theorem~\ref{mainthm} 
relies heavily on ``specialization'' techniques, among other 
things, Appendix,~Theorem~\ref{Thm:Main} by~\nmnm{J}{ossen}
(generalizing \nmnm{P}{ink}~[Pk],~Theorem~2.8), and previous 
work by the author \nmnm{P}{op}~[P1],~[P2],~[P3],~[P4]. For
reader's sake, 
I added at the beginning of the next section, precisely
in~section~2),~A) explanations about the strategy the 
logical structure of quite involved proof.
\vskip5pt
Concerning applications of Theorem~\ref{mainthm} above,
the point is that under supplementary conditions on the base
field $k$, one can recover $\Inrdiv(K)$ inside $\Inrqdiv^1(K)$,
thus~Theorem~\ref{mainthm} above is applicable. Namely, 
recall that the \defi{Kronecker dimension} $\dim(k)$ 
of $k$ is defined as follows: 
First, we set $\dim(\,\lvF_p)=0$ and $\dim(\lvQ)=1$, 
and second, if $k_0\subset k$ is the prime field of $k$, we define
$\dim(k):=\td k {k_0}+\dim(k_0)$. Hence $\dim(k)=0$ if 
and only if $k$ is an algebraic closure of a finite field, 
and $\dim(k)=1$ if and only if $k$ is an algebraic 
closure of a global field, etc. Therefore, 
$\td K k >1$ is equivalent to $\td K k >\dim(k)+1$ in the 
case $k$ is an algebraic closure of a finite field, whereas 
$\td K k >\dim(k)+1$ is equivalent to $\td K k >2$ if $k$ 
is an algebraic closure of a global field, etc. 
\vskip3pt
In light of the above discussion and notations, 
one has the following generalization of the 
main results of 
\nmnm{B}{ogomolov}--\nmnm{T}{schinkel}~[B--T], 
\nmnm{P}{op}~[P4]: 
\begin{thm} 
\label{appl1}
Let $K|k$ be an function field with $\td K k >1$ and 
$k$ algebraically closed of Kronecker dimension
$\dim(k)$. The following hold:
\vskip3pt
\begin{itemize}[leftmargin=30pt]
\item[{\rm1)}] For every nonnegative integer $\delta$, 
there exists a group theoretical recipe $\eudim(\delta)$ 
depending on $\delta$, which holds for $\pic K$ if and 
only if $\,\dim(k)=\delta$ and $\td K k >\dim(k)$.
\vskip3pt
\item[{\rm2)}] There is a group theoretical recipe which 
recovers $\Inrdiv(K)\subset\pia K$ from $\pic K$, provided 
$\,\td K k >\dim(k)$. Thus by\/ {\bf Theorem~\ref{mainthm},~1)}, 
one can reconstruct $K|k$ functorially from $\pic K$, provided
either $\,\td K k >\dim(k)+1$, or $\,\td K k >\dim(k)>1$.
\vskip3pt
\item[{\rm3)}] Both recipes above are invariant under 
isomorphisms of profinite groups as follows: Suppose that
$\pic K$ and $\pic L$ are isomorphic. Then $\td K k = \td L l$, 
and one has: $\,\dim(k)=\delta$ and $\td K k >\dim(k)$ 
if and only if $\,\dim(l)=\delta$ and $\td L l >\dim(l)$. 
\vskip3pt
\item[{\rm4)}] In particular, if either 
$\,\td K k >\dim(k)+1$, or $\,\td K k >\dim(k)>1$ holds,
then by {\bf Theorem~\ref{mainthm},~2)}, one concludes 
that the canonical map below is a bijection:
\vskip7pt
\centerline{$\Isom^{\rm F}(L,K)\hor{}
      \Isom^{\rm c}(\pia K,\pia L)$, \ $\phi\mapsto\Phi_\phi$.}
\end{itemize}
\vskip3pt
\end{thm}
\indent
If $\dim(k)=1$, i.e., $k$ is an algebraic closure of a 
global field, our methods developed here work as well 
for the function fields $K=k(X)$ of projective smooth 
surfaces $X$ with finite (\'etale) fundamental group, 
thus for function fields of ``generic'' surfaces. 
\vskip5pt
An immediate consequence of Theorem~\ref{appl1} 
is a positive answer to a question by Ihara from the 
1980's, which in the 1990's became a conjecture by 
Oda--Matsumoto, for short I/OM, which is about 
giving a topological/combinatorial description of the 
absolute Galois group of the rational numbers; see 
\nmnm{P}{op}~[P6], Introduction, for explanations 
concerning I/OM. The situation we consider here is as follows: Let 
$k_0$ be an arbitrary perfect field, and $k:=\oli k_0$ an 
algebraic closure. Let $X$ be a geometrically integral 
$k_0$-variety, $\clU_X:=\{U_i\}_i$ be a basis of open 
neighborhoods of the generic point $\eta_X$, and 
$\clU_{\oli X}=\{\hhb1\oli U\hhb1\!_i\}_i$ its base change 
to~$k$. Set $\Pi^{\rm c}_{U\!_i}:=\pi^{\rm c}_1(\oli U\!_i)$ and 
$\Pi_{U\!_i}:=\pi^{\ell, \rm ab}_1(\oli U\!_i)$. Then letting
$K:=k(\oli X)$ be the function field of~$\oli X$, it follows 
that $\pic K\to\pia K$ is the projective limit of the 
system $\Pi^{\rm c}_{U\!_i}\to\Pi_{U\!_i}$, 
and there exists a canonical embedding 
$\Aut{^{\rm c}}{\Pi_{\clU_X}}\hra
\Aut{^{\rm c}}{\pia K}$. Finally let 
$\Aut{^{\rm F}_k}K\hra\Aut{^{\rm F}}K$ 
be the group of all the $k$-automorphisms of $K^{\rm i}$, 
respectively all the field automorphisms of $K^{\rm i}\!$,  
modulo Frobenius twists. Note that since $k\subset K^{\rm i}$ 
is the unique maximal algebraically closed subfield in 
$K^{\rm i}$, every $\phi\in\Aut{^{\rm F}}K$ maps $k$ 
isomorphically onto itself, hence $\Aut{^{\rm F}}K$ 
acts on~$k$. Let $k_K\subseteq k_0$ be the corresponding 
fixed field up to Frobenius twists.
\begin{thm} 
\label{appl2}
In the above notations, suppose that 
$\dim(X)>\dim(k_0)+1$. Then 
one has a canonical exact sequence of the form:
$1\to\Aut{^{\rm F}_k}{K}\to\Aut{^{\rm c}}{\hhb1\pia K}
\to\Aut{^{\rm F}_{k_K}}k\to1$.
\vskip2pt
Thus if $\Aut{^{\rm F}_k}{K}=1$ and $k_K=k_0$, then 
$\imath^{\rm c}_K:\Ggr{k_0}\to\Aut{^{\rm c}}{\hhb1\pia K}$
is an isomorphism, hence the pro-$\ell$ abelian-by-central\/ 
{\rm I/OM} holds for~$\clU_X$, and so does the 
classical\/ {\rm I/OM}.
\vskip0pt
$\hhb1$
\end{thm}
We note that Theorem~\ref{appl2} is an immediate 
consequence of Theorem~\ref{appl1}: Setting namely
$L|l=K|k$, 
Theorem~\ref{appl1} implies that the canonical 
map $\Aut{^{\rm F}}K\to\Aut{^{\rm c}}{\pia K}$ is 
an isomorphism of groups. Further, since $k\subset K^{\rm i}$ 
is the unique maximal algebraically closed subfield of 
$K^{\rm i}$, every $\phi\in\Aut{^{\rm F}}{K}$ 
maps $k$ isomorphically onto itself. Conclude by 
noticing that one has an obvious exact sequence of groups
\vskip7pt
\centerline{$1\to \Aut{^{\rm F}_k}K\to
\Aut{^{\rm F}}K\to\Aut{^{\rm F}_{k_K}}k\to1$.}
\vskip7pt
{\bf Thanks:} I would like to thank all who showed 
interested in this work, among whom: Ching-Li Chai, 
Pierre Deligne, Peter Jossen, Frans Oort, Richard Pink, Jakob 
Stix, Tam\'as Szamuely, Michael Temkin, and Adam 
Topaz for clarifying technical aspects, and Viktor Abrashkin, 
Mihnyong Kim, Hiroaki Nakamura, Mohamed Saidi 
and Akio Tamagawa for discussion on several occasions
concerning earlier versions and ideas of the manuscript. 
\section{Proof of Theorem~\ref{mainthm}}
\noindent
A) {\it On the strategy of the proof\/}
\vskip5pt
First recall the strategy to tackle Bogomolov's program
as explained in~[P3]. Namely, after choosing a throughout 
fixed identification $\imath:\lvT_{{\lvG}_m,\hhb1k}
=\lvZ_\ell(1)\to\Zell$,
one gets via Kummer theory an isomorphism,which is
canonical up to the choice of the identification $\imath$,
as follows:
\[
\whK:=\plim{e}\,\,K\tms/\ell^e\to
    \plim{e}\,\,{\rm Hom}(\pia K,\lvZ/\ell^e)={\rm Hom}(\pia K,\Zell)
\]

Let $\jmath_K:K\tms\to\whK$ be the completion
functor. Since $k\tms$ is divisible, and $K\tms/k\tms$ is
a free abelian group, it follows that $\ker(\jmath_K)=k\tms$,
and $\jmath_K(K\tms)=K\tms/k\tms$ is $\ell$-adically 
dense in $\whK$. We identify $\jmath_K(K)\subset\whK$ 
with $\clP(K):=K\tms/k\tms$ and interpret it as the 
projectivization of the (infinite dimensional) $k$-vector space 
$(K,+)$. The 1-dimensional projective subspaces of 
$\clP(K)$ are called \defi{collineations} in $\clP(K)$,
and notice that the collineations in $\clP(K)$ are of the 
form $\eul_{x,y}:=(k x+k y)\tms/k\tms$, where 
$x,y\in K\tms$ are linearly independent over $k$. 
\vskip2pt
Using the {\it Fundamental Theorem of Projective Geometries\/}
FTPG, see e.g.~\nmnm{A}{rtin}~[Ar], it follows that one can 
recover $(K,+)$ from $\clP(K)$ endowed with all the collineations 
$\eul_{x,y}$, $x,y\in K$. Moreover, the atomorphisms of 
$\clP(K)$ which respect all the collineations are semi-linear. Using
this fact one shows that the multiplication on $\clP(K)$ induced
by the group structure on $\whK={\rm Hom}(\pia K,\lvZ)$ 
is distributive w.r.t.\ the addition 
of the group $(K,+)$ recovered via the FTPG. Thus knowing 
$\clP(K)\subset\whK$ as a subgroup together with all the 
collineations in $\clP(K)$ allows one to finally to recover 
the function field $K|k$. Furthermore, for an automorphism 
$\hat\phi:\whK\to\whK$ the following are equivalent:
\vskip2pt
\begin{itemize}[leftmargin=30pt]
\item[i)] $\hat\phi$ is the $\ell$-adic completion of an 
automorphism $\phi\in{\rm Aut}^{\rm F}(K)$.
\vskip2pt
\item[ii)] $\hat\phi\big(\clP(K)\big)=\clP(K)$ and $\hat\phi$ maps
the set of collineations onto itself.
\end{itemize}
\vskip2pt
This leads to the following strategy to tackle Bogomolov's 
program, see~[P3], Introduction, for more details and how 
ideas evolved in this context. 
\vskip5pt
\defi{\textit{Give group theoretical recipes which are invariant 
under isomorphisms of profinite groups and recover$\hp/$reconstruct 
from $\pic K$, viewed as abstract profinite group, the following:}} 
\begin{itemize}[leftmargin=25pt]
\item[-] \defi{\textit{The subset $\clP(K)=\jmath_K(K\tms)\subset
   \whK={\rm Hom}(\pia K,\Zell)$.}}
\vskip3pt
\item[-] \defi{\textit{The collineations $\eul_{x,y}\subset\clP(K)$ 
for all $k$-linearly independent $x,y\in K$.}}
\end{itemize}
\vskip3pt
\noindent
We will give such group theoretical recipes which work
under the hypothesis of Theorem~\ref{mainthm}, that is,
recover $\clP(K)$ together with the collineations $\eul_{x,y}$
from $\pic K$ endowed with $\Inrdiv(K)$. 
\vskip5pt
\begin{itemize}[leftmargin=52pt]
\item[Step 1.] Given $\pic K$, thus $\pic K\to\pia K$, 
endowed with $\Inrdiv(K)\subset\pia K$, recover the 
\defi{total decomposition graph} $\clG_{\tottrK}$ for $K|k$.
\vskip2pt
\item[-] This will be accomplished in Subsection~B),
see Proposition~\ref{totaldecgraph}.
\vskip5pt
\item[Step 2.] Given the total decomposition graph 
$\clG_{\tottrK}$  for $K|k$, reconstruct a {\it first 
approximation\/} $\lxzl K\subset\whK$~of $\,\clP(K)$ 
inside $\whK$, called the \defi{(canonical) divisorial 
$\whU\!_K$-lattice}~of~$K|k$. 
\vskip2pt
\item[-] This will be accomplished in Subsection~C), see 
Proposition~\ref{propSubsecC}.
\vskip5pt
\item[Step 3.] Given $\pic K$ and $\lxzl K\subset\whK$, 
reconstruct a {\it better approximation\/} 
$\lxzl K^0\subset\lxzl K$~of~$\,\clP(K)$
inside~$\lxzl K$, called the~\defi{quasi arithmetical 
$\whU\!_K$-lattice} in $\whK$, satisfying the following: 
$\clP(K)\subseteq\lxzl K^0$~and 
$\lxzl K^0/\big(\whU\!_K\cdot\clP(K)\big)$ is torsion.
\vskip2pt
\item[-] This will be accomplished in Subsection~D), see 
Proposition~\ref{specres2}.
\vskip5pt
\item[Step 4.] Given $\clG_{\tottrK}$ endowed with $\lxzl K^0$, 
for all $x\in K$ with $k(x)\subset K$ relatively algebraically 
closed, recover the \defi{(rational) projection} 
$pr_{k(x)}:\pia K\to\pia{k(x)}$.
\vskip2pt
\item[-] This will be accomplished in Subsection~E), see
Proposition~\ref{charratquot}.
\end{itemize}
\vskip2pt
Finally, given $\clG_{\tottrK}$ endowed with $\lxzl K$
and all the rational projections $\pia K\to\pia{k(x)}$ as above,
conclude by applying the Main Theorem from 
\nmnm{P}{op}~[P3], Introduction.
\vskip5pt
\noindent

\vskip10pt
\noindent
B) {\it Generalities about decomposition graphs\/}
\vskip5pt
Let $k$ be an algebraically closed field with
${\rm char}(k)\neq\ell$, and $K|k$ be a function field
with $d:=\td K k >1$. We begin by recalling briefly basics 
about the (quasi) prime divisors of the function field
$K|k$, see \nmnm{P}{op}~[P1], Section 3, for more details. 

\vskip2pt
A \defi{flag of generalized prime divisors} of $K|k$ 
is a chain of $k$-valuations $\tlv_1\leq\dots\leq\tlv_r$ 
of $K$ such that $\tlv_1$ is a prime divisor of $K|k$
and inductively, $\tlv_{i+1}/\tlv_i$ is a prime divisor
of the function field $K\tlv_i|\hhb1k$. In particular, 
$r\leq\td K k$, and we also say that $\tlv_r$ is a 
prime $r$-divisor of $K|k$. By abuse of language, we
will say that the trivial valuation is the prime $0$-divisor
of $K|k$. A \defi{flag of generalized quasi prime divisors} 
$\tlbfv_1\leq\dots\leq\tlbfv_r$ of $K|k$ is defined in a
similar way, but replacing \defi{prime} by \defi{quasi
prime}. In particular, $\tlbfv_r$ will also be called a
quasi prime $r$-divisor, or a generalized quasi prime 
divisors of $K|k$ if $r$ is irrelevant for the context.
Note that the prime $r$-divisors of $K|k$ are precisely the 
quasi prime $r$-divisors of $K|k$ which are trivial~on~$k$. 
\vskip5pt
The \defi{total prime divisor graph} $\tottrK$
of $K|k$ is the half-oriented graph defined as follows:
\vskip2pt
\begin{itemize}[leftmargin=30pt]
\item[{a)}] Vertices: The vertices of $\tottrK$ are the 
residue fields $\Ktlv$ of all the generalized prime 
divisors $\tlv$ of $K|k$ viewed as distinct function 
fields.
\vskip2pt
\item[{b)}] Edges: If $\tlv=\tlw$, the trivial valuation 
$\tlv/\tlw=\tlw/\tlv$ is the only edge from $\Ktlv=\Ktlw$
 to itself, and it by definition a non-oriented edge. If 
 $\tlv\neq\tlw$, and $\tlw/\tlv$ is a prime divisor of 
 $\Ktlv$, then $w:=\tlw/\tlv$ is the only edge from $\Ktlv$ 
 to $\Ktlw$, and by definition,~oriented. Otherwise there
 are no edges from $\Ktlv$ to $\Ktlw$. 
\end{itemize}

The \defi{total quasi prime divisor graph} $\qtottrK$ of $K|k$ 
is defined in a totally parallel way, but considering as vertices 
all the generalized quasi prime divisors. 
\vskip2pt
Notice that $\tottrK\subset\qtottrK$ 
is a full subgraph, and that the following functoriality holds:
\vskip3pt
\begin{itemize}[leftmargin=25pt]
\item[1)] {\it Embeddings.\/} Let $L|l\hra K|k$ be an embedding 
of function fields which maps $l$ isomorphically onto $k$. 
Then the canonical restriction map of valuations
$\Val K\to\Val L$, $v\mapsto v|_L$, gives rise to a surjective 
morphism of the total (quasi) prime divisor graphs 
$\varphi_\imath:\tottrK\to\tottrL$ and $\varphi_\imath:\qtottrK\to\qtottrL$.
\vskip5pt
\item[2)] {\it Restrictions.\/} Given a generalized prime 
divisor $\tlv$ of $K|k$, let $\clD^{\rm tot}_{\tlv}$ 
be the set of all the generalized prime divisors 
$\tlbfw$ of $K|k$ with $\tlv\leq\tlw$. Then the map 
$\clD_\tlv^{\rm tot}\to\Val\Ktlv$, $\tlw\mapsto\tlw/\tlv$, 
is an isomorphism of $\clD_\tlv^{\rm tot}$ onto the total 
prime divisor graph of $\Ktlv\hhb1|\hhb1k$. Similarly, 
the corresponding assertion for generalized quasi prime 
divisors holds as well.
\end{itemize} 
\vskip5pt
\noindent 
$\bullet$ \defi{Decomposition graphs} 
[See \nmnm{P}{op}~[P3], Section 3, for more details.] 
\vskip2pt
Let $K|k$ be as above. For every valuation $v$ of $K$, 
let $1+\eum_v=:U^1_v\subset U_v:=\clO_v^\times$ be
the principal $v$-units, respectively the $v$-units in $K\tms$. 
By \nmnm{P}{op}~[P1] and \nmnm{T}{opaz}~[To1], it
follows that the decomposition field $K^\ZZ_v$ of $v$ is 
contained in $K^{\ZZ^1}:=K[\!\root{\ell^\infty}\of{U_v^1}\,]$,
and the inertia field $K^\TT_v$ of $v$ is contained in 
$K^{\TT^1}:=K[\!\root{\ell^\infty}\of{U_v}\,]$. We denote 
$\mT_v:={\rm Gal}(K'|K^{\TT^1})\subseteq T_v$ and 
$\mZ_v:={\rm Gal}(K'|K^{\ZZ^1})\subseteq Z_v$ and 
call $T^1\subseteq Z^1$ the \defi{minimized inertia/decomposition} 
groups of $v$. Recalling that $\Kv^\times=U_v/U_v^1$, by
Kummer theory one gets that
\[
\Pi_\Kv^1:=\mZ_v/\mT=
   {\rm Hom}^{\rm cont}\big(U_v/U_v^1,\lvZ_\ell(1)\big)=
     {\rm Hom}^{\rm cont}\big(\Kv^\times\!,\lvZ_\ell(1)\big).
\]
By abuse of language, we say that $\Pi_\Kv^1$ is the 
\defi{minimized residue Galois group} at $v$. We 
notice that $K^\ZZ=K^{\ZZ^1}$, $K^{\TT^1}=K^\TT$,
$\Pi_\Kv^1=\Pi_\Kv$, provided $\chr(\Kv)\neq\ell$. On
the other hand, if $\chr(\Kv)=\ell$, then one must have
$\chr(k)=0$, and in this case $T_v^1\subseteq Z_v^1\subseteq T_v$,
hence the residue field of $K^{\ZZ^1}\!$ contains $(\Kv)'$,
thus $\Pi_\Kv^1\subseteq T_v/T_v^1$ has trivial image 
in $\Pi_\Kv=Z_v/T_v$.
\vskip5pt
Anyway, for generalized prime divisors $\tlv$ of 
$K$, one has: $T_\tlv=T_\tlv^1$, $Z_\tlv=Z_\tlv^1$, 
$\Pi_\Ktlv^1=\Pi_\Ktlv$.
\vskip5pt
Recall that a generalized quasi prime divisors $\tlbfv$ 
is a quasi prime $r$-divisor iff $\mT_\tlbfvm\cong\Zell^r$. 
Moreover, any of the equalities 
$K_\tlbfvm^{\ZZ^1}\!=K_\tlbfvm^\ZZ$, 
$K_\tlbfvm^{\TT^1}\!=K_\tlbfvm^\TT$, 
$\Pi_{K\tlbfvm}^1=\Pi_{K\tlbfvm}$, is equivalent to 
$\chr(\Ktlbfv)\neq\ell$. Further, for generalized quasi 
prime divisors $\tlbfv_1$ and $\tlbfv_2$ one has:
$\mZ_{\tlbfvm_1}\cap \mZ_{\tlbfvm_2}\neq1$ iff 
$\,\mT_{\tlbfvm_1}\cap \mT_{\tlbfvm_2}\neq1$, and
if $\,\mT_{\tlbfvm_1}\cap \mT_{\tlbfvm_2}\neq1$, 
there exists a unique generalized quasi prime divisor 
$\tlbfv$ of $K|k$ with $\mT_\tlbfvm= \mT_{\tlbfvm_1}
\cap \mT_{\tlbfvm_2}$. And $\tlbfv$ is also the 
unique generalized quasi prime divisor of $K|k$ 
maximal with the property $\mZ_{\tlbfvm_1}, 
\mZ_{\tlbfvm_2}\subseteq\mZ_\tlbfvm$. Finally, $\tlbfv$ 
is trivial on $k$, provided $\min(\tlbfv_1,\tlbfv_2)$ 
is so.
\vskip5pt
In particular, for generalized (quasi) prime divisors 
$\tlbfv$ and $\tlbfw$ of $K|k$ one has: $\tlbfv=\tlbfw$ 
iff $\mT_\tlbfvm=\mT_\tlbfwm$ iff $\mZ_\tlbfvm=\mZ_\tlbfwm$.
Further, $\tlbfv<\tlbfw$ iff $\mT_\tlbfvm\subset\mT_\tlbfwm$ 
strictly iff $\mZ_\tlbfvm\supset\mZ_\tlbfwm$ strictly. And 
if $\tlbfv<\tlbfw$ is are a (quasi) prime $r$-divisor, 
respectively a (quasi) prime $s$-divisor,
then $\mT_\tlbfwm/\mT_\tlbfvm\cong\Zell^{s-r}$.
\vskip2pt
We conclude that the partial ordering on the set of all 
the generalized (quasi) prime divisors $\tlbfv$ of $K|k$ 
{\it is encoded in the set\/} of their minimized 
inertia/decomposition groups $\mT_\tlbfvm\subseteq\mZ_\tlbfvm$. 
In particular, the existence of the trivial, respectively 
nontrivial, edge from $\Ktlbfv$ to $\Ktlbfw$ in $\qtottrK$ 
(and/or $\tottrK$) is equivalent to $\mT_\tlbfvm=\mT_\tlbfwm$, 
respectively to $\mT_\tlbfvm\subset\mT_\tlbfwm$ and 
$\mT_\tlbfwm/\mT_\tlbfvm\cong\Zell$.
\vskip3pt
Via the Galois correspondence and the functorial 
properties of the Hilbert decomposition theory for 
valuations, we attach to the total prime divisor graph 
$\tottrK$ of $K|k$ a graph $\clG_{\tottrK}$ whose 
vertices and edges are in bijection with those of $\tottrK$, 
as follows:
\vskip2pt
\begin{itemize}[leftmargin=30pt]
\item[{a)}] The vertices of $\clG_{\tottrK}$ are the 
pro-$\ell$ groups $\pia{\Ktlv}$, viewed as distinct 
pro-$\ell$ groups, where $\tlv$ are all the generaalized 
prime divisors of $K|k$.
\vskip2pt
\item[{b)}] If an edge from $\Ktlv$ to $\Ktlw$ exists, 
the corresponding edge from $\pia{\Ktlv}$ to 
$\pia{\Ktlw}$ is endowed with the pair of groups 
$T_{\tlw\!/\!\tlv}=T_\tlw/T_\tlv\subseteq 
Z_\tlw/T_\tlv=Z_{\tlw\!/\!\tlv}$, viewed as subgroups 
of the residue Galois group $\pia{\Ktlv}=Z_\tlw/T_\tlv$, 
and notice that in this case 
$\pia{\Ktlw}=Z_{\tlw\!/\!\tlv}/T_{\tlw\!/\!\tlv}$.
\end{itemize}
The graph $\clG_{\tottrK}$ will be called the \defi{total
decomposition graph} of $K|k$, or of $\pia K$.
\vskip5pt
In a similar way, we attach to $\qtottrK$ the \defi{total
quasi decomposition graph} $\clG^1_{\qtottrK}$ of 
$K|k$, but using the minimized 
inertia/decomposition/residue groups $\mT_\tlbfvm$,
$\mZ_\tlbfvm$, $\Pi_{K\tlbfvm}^1$ instead of the 
inertia/decomposition/residue Galois groups (which are 
the same for generalized prime divisors, because 
$\chr(k)\neq\ell\,$). Clearly, $\clG_{\tottrK}$ is a 
full subgraph of  $\clG^1_{\qtottrK}$.
\vskip5pt
The functorial properties of the total graphs of (quasi) prime 
divisors translate in the following functorial properties of 
the total (quasi) decomposition graphs:
\vskip3pt
\begin{itemize}[leftmargin=25pt]
\item[1)] {\it Embeddings.\/} Let $\imath: L|l\hra K|k$ be an
embedding of function fields which maps $l$ isomorphically
onto $k$. Then the canonical projection homomorphism 
$\Phi_\imath:\pia K\to\pia L$ is an open homomorphism,
and moreover, for every generalized (quasi) prime divisor 
$\bfv$ of $K|k$ and its restriction $\bfvL$ to $L$ one has:
$\Phi_\imath(\mZ_\bfv)\subseteq\mZ_{\bfvL}$ is an open 
subgroup, and $\Phi_\imath(\mT_\bfv)\subseteq\mT_{\bfvL}$
satisfies: $\Phi_\imath(\mT_\bfv)=1$ iff $\bfvL$ has 
divisible value group, e.g., $\bfvL$ is the trivial valuation. 
Therefore, $\Phi_\imath$ gives rise to \defi{morphisms} 
of total (quasi) decomposition graphs
\vskip5pt
\centerline{$\Phi_\imath:\clG_{\tottrK}\to\clG_{\tottrL}$,
\quad $\Phi_\imath:\clG^1_{\qtottrK}\to\clG^1_{\qtottrL}$.}
\vskip5pt
\item[2)] {\it Restrictions.\/} Given a generalized 
(quasi) prime divisor $\bfv$ of $K|k$, let 
$pr_\bfv:\mZ_\bfv\to\Pi_{\Kbfv}^1$ be the canonical 
projection. Then for every $\bfw\geq\bfv$ we have:
$\mT_\bfw\subseteq\mZ_\bfw$ are mapped onto 
$\mT_{\bfw\!/\!\bfv}:=\mT_\bfw/\mT_\bfv\subseteq
\mZ_\bfw/\mT_\bfv=:\mZ_{\bfw\!/\!\bfv}$. 
Therefore, the total (quasi) decomposition graph for 
$\Kbfv\hhb1|\hhb1k\bfv$ can be recovered from 
the one for $K|k$ in a canonical way via 
$pr_\bfv:\mZ_\bfv\to\Pi_{\Kbfv}^1$.
\end{itemize}
\begin{remark}
\label{remSubsecA}
By the discussion above, the following hold:
\vskip2pt
\itm{25}{
\item[1)] Reconstructing $\clG_{\tottrK}$, respectively
$\clG^1_{\qtottrK}$, is equivalent to describing 
the set of all the generalized divisorial groups 
$T_{\tlv}\subset Z_{\tlv}$ in $\Pi_K$, 
respectively describing the set of all the 
generalized (minimized) quasi divisorial groups 
$\mT_{\tlbfvm}\subset\mZ_{\tlbfvm}$ 
in $\Pi_K$.
\vskip2pt
\item[2)] Given a generalized prime divisor $\tlv$ of 
$K|k$, one can recover $\clG_{\clD_\Kv^{\rm tot}}$ 
from $\clG_{\tottrK}$ together with $Z_\tlv\to\Pi_{\Ktlv}$.
Correspondingly, given a generalized quasi prime 
divisor $\tlbfv$ of $K|k$, one can recover 
$\clG^1_{\clQ_{\Ktlbfvm}^{\rm tot}}$ from 
$\clG^1_{\qtottrK}$ together with 
$\mZ_\tlbfvm\to\Pi^1_{K\tlbfvm}$.
}
\end{remark}

For later use we notice the following: Let $\tlbfv$ 
be a generalized quasi prime divisor.  Recall 
that for every generalized quasi prime divisor 
$\tlbfw$ one has: $\mT_{\tlbfwm}\supset\mT_{\tlbfvm}$
strictly iff $\mZ_{\tlbfwm}\subset\mZ_{\tlbfvm}$ 
strictly iff $\tlbfw>\tlbfv$. If so, then 
$\mT_{\tlbfvm}\subset\mT_{\tlbfwm}\subset\mZ_\tlbfvm$, 
and we consider the closed subgroup:
\[
T^1_{\Ktlbfvm}\subset\Pi^1_{K\tlbfvm} \ \hbox{\ generated
by } \ \mT_{\tlbfwm}/\mT_{\tlbfvm}\subset
\mZ_\tlbfvm/\mT_\tlbfvm\subset\Pi^1_{\Ktlbfvm}
\ \hbox{ for all } \ \tlbfw>\tlbfv.
\]
This being said, one has the following group
theoretical criterion to check that
$\chr(k\bfv)\neq\ell$.
\begin{prop}
\label{rescharneqell}
Let $\bfv$ be a $($generalized$\,)$ quasi prime 
divisor such that $k\bfv$ is an algebraic closure 
of a finite field and $\td{\Kbfv}{k\bfv}>1$. Then 
the following are equivalent:
\vskip2pt
\itm{25}{
\item[{\rm i)}] $\chr(k\bfv)\neq\ell$.
\vskip2pt
\item[{\rm ii)}] $\Pi^1_{\Ktlbfvm}/T^1_{\Ktlbfvm}$ 
is a finite $\lvZ_\ell$-module of even rank for all 
quasi prime $(d-1)$-divisors $\tlbfv>\bfv$.
}
\end{prop}
\begin{proof} To simplify notations, set
$\kappa:=k\bfv$. Since $\kappa$ is an algebraic 
closure of a finite field, it has only the trivial 
valuation, hence $k\tlbfv=\kappa$ for all quasi
prime divisors $\tlbfv\geqslant\bfv$. Further, 
$\tlbfv\geqslant\bfv$ is a quasi prime 
$(d-1)$-divisor iff $\td{\Ktlbfv}\kappa=1$ iff
$\Ktlbfv\hhb1|\hhb1\kappa$ is the function field
of a projective smooth $\kappa$-curve $X_\tlbfvm$.
If so, let $\dvv:\Ktlbfv^\times\to\Div^0(X_\tlbfvm)
\subset\Div(X_\tlbfvm)$ is the divisor map. Then 
by mere definitions, it follows that
$T^1_{\Ktlbfvm}\to\Pi^1_{\Ktlbfvm}$ is the $\ell$-adic 
dual of $\dvv(\Ktlbfv^\times)\to\Div^0(X_\tlbfvm)$, 
and in particular, $\Pi^1_{\Ktlbfvm}/T^1_{\Ktlbfvm}$
is the $\ell$-adic dual of 
$\Div^0(X_\tlbfvm)/\dvv(\Ktlbfv^\times)=\Pic^0(X_\tlbfvm)$.
Therefore, $\Pi^1_{\Ktlbfvm}/T^1_{\Ktlbfvm}$ is 
isomorphic to the Tate $\ell$-module of $\Pic^0(X_\tlbfvm)$. 
\vskip5pt
i) $\Rightarrow$ ii): Let $\tlbfv$ be an arbitrary
quasi prime $(d-1)$-divisor. In the above notations,
$X_\tlbfvm$ is a projective smooth curve over 
the algebraically closed field $\kappa$. Let
$g_\tlbfvm$ be the genus of $X_\tlbfvm$. Since
$\chr(\kappa)\neq\ell$, it follows that the 
$\lvZ_\ell$-rank of $\Pic^0(X_\tlbfvm)$ is 
$2g_\tlbfvm$, thus even.
\vskip3pt
ii) $\Rightarrow$ i): Equivalently, we have to
prove that if $\chr(\kappa)=\ell$, then there exist
quasi prime $(d-1)$-divisors $\tlbfv>\bfv$ such
that $\lvZ_\ell$-rank of the Tate $\ell$-module of 
$\Pic^0(X_\tlbfvm)$ has odd rank. To proceed, 
let $\bfv$ have rank $r$, hence by hypothesis, 
$e:=d-r-1=\td{\Kbfv}\kappa-1>1$. Let $X$ 
be a smooth (not necessarily proper) $\kappa$-model 
of the function field $\Kbfv\hhb1|\hhb1\kappa$, 
and recall the the following fact ---a proof of 
which was communicated to me by 
\nmnm{Ch}{ai}--\nmnm{O}{ort}~\cite{Ch-O}:
\begin{fact} In the above notation, let
$X^1\subset X$ be the points with $\dim(x^1)=1$,
and $C_{x^1}$ be the unique projective smooth 
curve with $\kappa(C_{x^1})=\kappa(x^1)$. 
Then $\chr(\kappa)=\ell$ iff there exist  
$x^1\in X^1$ such that the $\lvZ_\ell$-rank 
of the Tate $\ell$-module of $\Pic^0(C_{x^1})$ 
is odd.
\end{fact}  
\noindent
Now since $\chr(\kappa)=\ell$, by the Fact above, 
there exists a point $x^1\in X$ such that the 
$\lvZ_\ell$-rank of the Tate $\ell$-module of 
$\Pic^0(C_{x^1})$ is odd. By mere definitions
one has $\kappa(C_{x^1})=\kappa(x^1)$, and since 
$\kappa$ is algebraically closed, one has: First, 
$\kappa(x^1)\hhb1|\hhb1\kappa$ is separably 
generated, and second, since $x^1$ is smooth, 
the local ring $\clO_{x^1}$ is regular. Further, 
recalling that $\td{\Kbfv}\kappa-1=e>1$, 
one has that 
$\dim(\clO_{x^1})=\td{\Kbfv_0}\kappa-1=e>0$. 
Hence the completion of the local 
ring $\clO_{x^1}$ is the power series ring 
$\widehat\clO_{x^1}=\kappa(x^1)[[t_1,\dots,t_e]]$, 
where $(t_1,\dots,t_e)$ is any regular system of 
parameters at $x^1$. In particular, there exist
``many'' prime $e$-divisors $\tlbfw$ of 
$\Kbfv$ such that $(\Kbfv)\tlbfw=\kappa(x^1)$.
Hence setting $\tlbfv:=\tlbfw\circ\bfv$, it
follows that $\Ktlbfv=\kappa(x^1)$, thus 
$\tlbfv>\bfv$ is a quasi prime $(d-1)$-divisor
of $K|k$, having $k\tlbfv=\kappa$. Further, 
$X_\tlbfvm:=C_{x^1}$ is the projective smooth 
model of $\Ktlbfv\hhb1|\hhb1\kappa$. 
Since the $\lvZ_\ell$-rank of 
$\Pic^0(X_\tlbfvm)=\Pic^0(X_{x^1})$
is odd by the choice of $x^1$, the quasi prime 
$(d-1)$-divisor $\tlbfv$ does the job.
\end{proof}  
We conclude this subsection by the following 
Proposition, which relates the generalized prime
divisors to the space of divisorial inertia.
\begin{prop}
\label{totaldecgraph}
Let $\Inrtm_k(K)\subset\pia K$ be the topological 
closure of $\Inrdiv(K)$ in $\pia K$. Further, for
closed subgroups $\Delta\subset\pia K$, let 
$\Delta''\subset\pic K$ be their preimages under the
canonical projection $\pic K\to\pia K$. Then $\td Kk$
is the maximal integer $d$ such that there exist 
subgroups $\Delta\cong\lvZ_\ell^d$ of $\,\pia K$ 
with $\Delta''\subset\pic K$ abelian. Further, the
following hold:
\vskip2pt
\begin{itemize}[leftmargin=25pt]
\item[{\rm1)}] For
sequences of subgroups $Z_1\supset\dots\supset Z_r$,
$T_1\subset\dots\subset T_r$ with $T_m\subseteq Z_m$,
$m=1,\dots,r$, the following are equivalent:
\vskip2pt
\begin{itemize}[leftmargin=25pt]
\item[{\rm i)}] There exist prime $m$-divisors $\tlv_m$
such that $T_m=T_{\tlv_m}\subset Z_{\tlv_m}=Z_m$ for 
$m=1,\dots,r$.
\vskip2pt
\item[{\rm ii)}] $Z_1\supset\dots\supset Z_r$,
$T_1\subset\dots\subset T_r$ are maximal w.r.t.\ 
inclusion satisfying:
\vskip2pt
\begin{itemize}[leftmargin=10pt]
\item[-] $Z_r$ contains subgroups $\Delta\cong\lvZ_\ell^{\td Kk}$
with $\Delta''\subset\pic K$ abelian.
\vskip2pt
\item[-] $T_r\subset\Inrtm_k(K)$, $T_m\cong \lvZ_\ell^m$ 
and $T''_m\subset Z''_m$ is the center of $Z''_m$ for
$m=1,\dots,r$.
\end{itemize}
\end{itemize}
\vskip4pt
\item[{\rm2)}] In particular, the total decomposition graph
$\clG_{\tottrK}$ for $K|k$ can be reconstructed from
$\pic K\to\pia K$ endowed with divisorial inertia 
$\Inrdiv(K)\subset\pia K$.
\vskip4pt
\item[{\rm3)}] Moreover, the group theoretical recipe 
to recover $\clG_{\tottrK}$ is invariant under isomorphisms 
as follows: Let $\Phi\in{\rm Isom}(\pia K,\pia L)$ satisfy 
$\,\Phi\big(\Inrdiv(K)\big)=\Inrdiv(L)$. Then $\Phi$ maps
the generalized divisorial groups $T_{\tlv}\subset Z_{\tlv}$ 
of $\,\pia K$ onto the generalized divisorial groups  
$T_{\tlw}\subset Z_{\tlw}$ of 
$\,\pia L$, thus defines an isomorphism $\Phi:\clG_{\tottrK}
\to\clG_{\tottrL}$.
\end{itemize}
\end{prop}
\begin{proof}
The proof of the assertions~1),~2) follows instantly from
the characterization of the generalized divisorial groups
$T_{\tlv}\subset Z_{\tlv}$ given at~ii). Therefore, it is
sufficient to prove the equivalence of~assertions~i),~ii).
First, the implication i)~$\Rightarrow$~ii) holds even for 
a more general class of valuations, see~[P1],~Proposition~4.2.
The proof of the converse implication~ii)~$\Rightarrow$~i), 
is identical with the proof of the corresponding assertion 
from [P4],~Proposition~3.5, the only (formal) change 
necessary being to replace $\Inrtm(K)$ from~[P4] by
$\Inrtm_k(K)$ in the present situation.
\end{proof}
\noindent
C) {\it Fundamental groups and divisorial lattices}
\vskip5pt
We begin by recalling here a few basic facts
from~\nmnm{P}{op}~[P3] and proving a little bit more
precise/stronger results about the (pro-$\ell$ abelian)
fundamental group of quasi projective normal 
$k$-varieties, see the discussion from~[P3], Appendix,
section~7.3 for some of the details. These stronger and
more precise results will be needed later on, e.g., in
Subsection~D) when discussing specializations techniques. 
The facts were not known to the author and those whom 
he asked at the time~[P3] was written. 
\newpage
\vskip5pt
\noindent
$\bullet$ {\it On the fundamental group of sets of prime divisors\/}
\vskip2pt
Let $D$ be a set of prime divisors of $K|k$. We denote 
by $T_D\subseteq\pia K$ the closed subgroup 
generated by all the $T_v$, $v\in D$, and say that 
$\pia{1,D}:=\pia K/T_D$ is the \defi{fundamental 
group} of the set $D$. In the case $D$ equals the 
set of all the prime divisors $D=\DD{K|k}$ of  $K|k$, 
we say that $\pia{1,\DD{K|k}}=:\pia{1,K}$ is the 
\defi{(birational) fundamental group for} $K|k$.
Recall that a set $D$ of prime divisors of $K|k$ is called
\defi{geometric}, if there exists a normal model 
$X\to k$ of $K|k$ such that $D=\DD X$ is the set 
of Weil prime divisors of $X$. (If so, there are  
always quasi-projective normal models $X$ with 
$D=\DD X$.) In particular, 
if $X$ is a normal model of $K|k$ and $\Pi_1(X)$ 
denotes the maximal pro-$\ell$ abelian fundamental 
group of $X$, one has canonical surjective projections
$\pia{1,\DD X}\to\Pi_1(X)$ and $\pia{1,\DD X}\to\pia{1,K}$.
Moreover, if $U\subset X$ is an open {\it smooth\/} 
$k$-subvariety, then $\pia{1,\DD U}\to\Pi_1(U)$ is 
an isomorphism (by the purity of the branch locus).
In particular, $\pia{1,\DD U}=\Pi_1(U)$ is a finite 
$\Zell$-module, and since one has the canonical 
surjective homomorphisms 
$\pia{1,\DD U}\to\pia{1,\DD X}\to\pia{1,K}$, 
it follows that $\pia{1,\DD X}=\pia{1,D}$ and 
$\pia{1,K}$ are finite $\Zell$-modules. Further, 
it was shown in~[P3], Appendix,~7.3, see 
especially Fact~57, that there exist (quasi projective) 
normal models $X$ such that $\pia{1,\DD X}\to\pia{1,K}$ 
is an isomorphism. Nevertheless, it is not 
clear that for every geometric set $D$ there exists 
quasi projective normal models $X$ such that 
$\pia{1,\DD X}\to\Pi_1(X)$ is an isomorphism. 
\vskip2pt
We begin by recalling two fundamental facts
concerning {\it alterations\/} as introduced 
by~\nmnm{}{de}~\nmnm{J}{ong} and developed by
\nmnm{G}{abber}, \nmnm{T}{emkin}, and many 
others, see e.g.\ [ILO], Expose~X.
\vskip3pt
Let $D$ be a fixed geometric set of prime divisors
for $K|k$, $X_0$ be any projective normal model 
of $K|k$ with $D\subseteq\DD{X_0}$, and 
$S_0\subset X_0$ a fixed closed proper subset.
Usually $S_0$ will be chosen to define $D$ in 
the sense that $D=\DD{\bsl{X_0}{S_0}}$.
By~[ILO], Expose~X, Theorem~2.1, there exist 
\defi{prime to $\ell$ alterations above $S_0$,\/} i.e., 
projective generically finite separable morphisms 
$\YY \to X_0$ satisfying the following:
\vskip2pt
\begin{itemize}[leftmargin=20pt]
\item[{-}] $\YY$ is a projective smooth $k$-variety.
\vskip2pt
\item[{-}] $\TT:=f_0^{-1}(S_0)$ is a NCD (normal 
crossings divisor) in $\YY$.
\vskip2pt
\item[{-}] $[k(\YY):K]$ is prime to $\ell$. 
\end{itemize}
We denote by $D_0$ the restriction of~$\DD{\YY}$
to~$K$, and notice that $\DD{X_0}\subseteq D_0$, 
because $\YY\to X_0$ is surjective. Hence finally
$D\subseteq\DD{X_0}\subseteq D_0$.
\vskip2pt
Second, let $X_1\to X_0$ be a dominant morphism 
with $X_1$ a projective normal model of $K|k$ 
such that $D_0\subseteq\DD{X_1}$, thus we have 
$D\subseteq\DD{X_0}\subseteq D_0 \subseteq\DD{X_1}$. 
Then by \nmnm{}{de}~\nmnm{J}{ong}'s theory of alterations,
see e.g., [ILO], Expose~X, Lemma~2.2,  there 
exists a \defi{generically normal finite alteration} of 
$X_1$, i.e., a projective dominant $k$-morphism
$Z \to X_1$ satisfying the following:
\vskip2pt
\begin{itemize}[leftmargin=20pt]
\item[{-}] $Z$ is a projective smooth $k$-variety.
\vskip2pt
\item[{-}] The field extension $K=k(X_1)\hra k(Z)=:M$ 
is a finite and normal.
\vskip2pt
\item[{-}] ${\rm Aut}(M|K)$ acts on $Z$ and
$Z \to X_1$ is ${\rm Aut}(M|K)$-invariant. 
\end{itemize}
By a standard scheme theoretical construction 
(recalled below), there exists a projective normal 
model $\XX$ for $K|k$ and a dominant $k$-morphism 
$\XX \to X_1$ such that $Z \to X_1$ factors through 
$\XX \to X_1$, and the resulting $k$-morphism 
$Z \to \XX$ is finite. In particular, since $Z$ is 
smooth, thus normal, $Z \to \XX$ 
is the normalization of $\XX$ in the field extension 
$K\hra M$. 
\vskip2pt
We briefly recall the standard scheme 
theoretical construction, which is a follows: 
Let $Z \to\bsl{{\rm Aut}(M|K)}Z=:Z_{\rm i}$ be 
the quotient of $Z$ by ${\rm Aut}(M|K)$. Then
$Z\to Z_{\rm i}$ is a finite generically Galois 
morphism, and its function field $M_{\rm i}:=k(Z_{\rm i})$
satisfies: $M|M_{\rm i}$ is Galois, and $M_{\rm i}|K$
purely inseparable. Hence there exists $e>0$ such that 
$M_{\rm i}^{(e)}:={\rm Frob}^e(M_{\rm i})\subset K$,
and $M_{\rm i}^{(e)}$ is the function field of the 
$e^{\rm th}$ Frobenius twist $Z_{\rm i}^{(e)}$ 
of $Z_{\rm i}$. And notice that $Z_{\rm i}^{(e)}$
is a projective normal model for the function field 
$M_{\rm i}^{(e)}|k$. Further, the normalization 
of $Z_{\rm i}^{(e)}$ in the finite field extension 
$M_{\rm i}^{(e)}\hra M_{\rm i}$ is nothing but 
$Z_{\rm i}$. Finally, let $\XX$ be the normalization 
of $Z_{\rm i}^{(e)}$ in the function field extension
$M_{\rm i}^{(e)}\hra K$. Then $\XX$ is a projective 
normal $k$-variety because $Z_{\rm i}^{(e)}$
was so, and $k(\XX)=K$, thus $\XX$ is a projective 
normal model of $K|k$. Further, by the transitivity 
of normalization, it follows that the normalization 
of $\XX$ in $K\hra M_{\rm i}$ equals the normalization 
of $Z_{\rm i}^{(e)}$ in the field extension 
$M_{\rm i}^{(e)}\hra M_{\rm i}$, and that normalization
is $Z_{\rm i}$. Finally, using the transitivity of
normalization again, it follows that the normalization 
of $\XX$ in $K\hra M_1$ is $Z$ itself. Finally, to
prove that $Z\to X_1$ factors through $Z\to\XX$,
we proceed as follows:  First, since $X_1$ and $\XX$
are both projective normal models of $K|k$, there is
a canonical rational map $\XX\ratmap X_1$. We claim
that $\XX\to X_1$ is actually a morphism. Indeed, let
$x\in\XX$ be a fixed point, and $Z_x\subset Z$ 
be the preimage of $x$ under $Z\to\XX$. Then
${\rm Aut}(M|K)$ acts transitively on $Z_x$, and 
since $Z\to X_1$ is ${\rm Aut}(M|K)$-invariant,
it follows that the image of $Z_x$ under $Z\to X_1$
consists of a single point, say $x_1\in X_1$. Now let 
$V_1:={\rm Spec}\,R_1\subset X_1$ be an affine open 
subset containing $x_1$, and $W:={\rm Spec}\, S\subset Z$ 
be an ${\rm Aut}(M|K)$ invariant open subset of
$Z$ containing $Z_x$, and $V:={\rm Spec}\,R\subset X$
be the image of $W={\rm Spec}\,S$ under $Z\to\XX$.
Then identifying $R_1,R$ and $S_1$ with the corresponding
$k$-subalgebras of finite type of $K$, respectively of $M=k(Z)$,
it follows that $W\to V_1$ and $W\to V$ are defined
by the $k$-embeddings $R_1\hra S$, respectively 
$R\hra S$, defined via the inclusion $K\hra M$. Now
since $S$ is the normalization of $R$, it follows that
$R$ is mapped isomorphically onto $K\cap S$. Thus
since $K\hra M$ maps $R_1$ into $K\cap S$, we 
conclude that $R_1\subset R$. That in turns shows 
that $V\ratmap V_1$ is defined by the $k$-inclusion 
$R_1\hra R$, thus it is a morphism, and therefore, 
defined at $x\in V$. Conclude that $X\ratmap X_1$ 
is actually a $k$-morphism.
\vskip2pt 
\begin{prepa/nota}
\label{prepfct}
Summarizing the discussion above, for a geometric set 
$D$ of prime divisors for $K|k$, and $X_0$ a projective 
normal model of $K|k$ with $D\subseteq\DD{X_0}$, 
we let $S_0\subset X_0$ be a closed subset with 
$D=\DD{\bsl{X_0}{S_0}}$, and can/will consider 
the following:
\vskip2pt
\begin{itemize}[leftmargin=30pt]
\item[{a)}] A prime to $\ell$ alteration $\YY\to X_0$ 
above $S_0$. We denote by $T\subset\YY$ the 
preimage of $S_0\subset X_0$ under $\YY\to X_0$,
and by $D_0$ the restriction of $\DD\YY$ to $K$.
Hence $D\subseteq\DD{X_0}\subseteq D_0$.
\vskip2pt
\item[{b)}] A morphism of projective normal 
models $\XX\to X_0$ with $D_0\subseteq\DD\XX$.
For closed subsets $S'_0\subset X_0$, let $T'\subset\YY$
and $S'\subset\XX$ be the corresponding preimages
of $S'_0\subset X_0$.
\vskip2pt
\item[{c)}] A smooth projective $k$-variety $\ZZ$ together 
with a dominant finite $k$-morphism $\ZZ\to\XX$ such that
$K=k(X)\hra k(\ZZ)=:M$ is normal and ${\rm Aut}(M|K)$
acts on $\ZZ$.
\end{itemize}
\end{prepa/nota}
\begin{prop}
\label{Pivspi}
In the above notations, the following hold:
\vskip2pt
\begin{itemize}[leftmargin=30pt]
\item[{\rm 1)}] The canonical projection
$\pia{1,\DD{\bsl \XX{S'}}}\to\Pi_1(\bsl \XX{S'})$ is an
isomorphism. Hence one has canonical surjective projections
$\pia{1,\DD{\bsl{X_0}{S'_0}}}\to\pia{1,\DD{\bsl \XX{S'}}}
\to\Pi_1(\bsl \XX{S'})\to\Pi_1(\bsl{X_0}{S'_0})$. 
\vskip2pt
\item[{\rm 2)}] Suppose that $D_0\subseteq\DD{\bsl\XX{S'}}$.
Then $\pia{1,\DD{\bsl\XX{S'}}}\to\Pi_1(\bsl\XX{S'})\to\pia{1,K}$ 
are isomorphisms.
\end{itemize}
\end{prop} 
\begin{proof} To 1): The existence and the subjectivity
of the projections is clear. Thus it is left to show that
$\pia{1,\DD{\bsl \XX{S'}}}\to\Pi_1(\bsl \XX{S'})$ is 
injective. Equivalently,
one has to prove the following: Let $\tl K|K$ be  
an abelian $\ell$-power degree extension, and 
$\tl \XX \to \XX$ be the normalization of $\XX$ in 
$K\hra\tl K$. Then $\tl \XX \to \XX$ is etale above
$S'$ if and only if none of the prime divisors 
$v\in\DD{\bsl\XX{S'}}$ has ramification in $\tl K|K$.
Clearly, this is equivalent to the corresponding 
assertion for all cyclic sub extensions of $\tl K|K$,
thus without loss of generality, we can suppose 
that $\tl K|K$ is cyclic, thus ${\rm Gal}(\tl K|K)$ is cyclic.
For points $\tl x\mapsto x$ under $\tl\XX\to\XX$ 
and valuations $\tl v|v$ of $\tl K|K$, we denote 
by $T_{\tl x}$, respectively $T_{\tl v}$ the 
corresponding inertia groups.
\vskip5pt
\noindent
{\bf Claim 1.} {\it Suppose that 
$\tl G:={\rm Gal}(\tl K|K)=\langle\tl g\hhb{.75}\rangle$ 
is cyclic. Then for every $\tl x\in \tl \XX$ there 
exists a prime divisor $\tl w$ of $\tl K|k$ with 
$T_{\tl x}\subseteq T_{\tl w}$ and $\tl x$ in the closure  
of the center $x_{\tl w}$ of $\tl w$ on $\tl X$.\/}
\vskip5pt
{\it Proof of Claim 1.\/} Compare with 
\nmnm{P}{op}~[P2],~Proof of~Theorem~B. Recalling
the cover $\ZZ\to X$ with $M:=k(\ZZ)$, let 
$\tl K':=M\tl K$ be the compositum of $\tl K$ and 
$M=k(Z)$, and $\tl X'\to X$ be the normalization of 
$X$ in the function field extension $K\hra\tl K'$ 
and the resulting canonical factorizations 
$\tl X' \to\tl X \to X$ and $\tl X' \to Z\to X$. Further, 
choosing a preimage $\tl x'$ of $\tl x$ under
$\tl X'\to\tl X$, consider $\tl x'\mapsto\tl x\mapsto x$ 
and $\tl x'\mapsto z\mapsto x$ under the above 
factorizations of $\tl X\to X$. Then by the 
functoriality of the Hilbert decomposition/ramification 
theory, there are surjective canonical projections 
$T_{\tl x'}\to T_{\tl x}$ and $T_{\tl x'}\to T_z$. Thus
given a generator $\tl g$ of $T_{\tl x}$, there exists 
$\tl g' \in T_{\tl x'}$ which maps to $\tl g$ under 
$T_{\tl x'}\to T_{\tl x}$. And if $\tl w'$ is a prime 
divisor of $\tl K'|k$ such that $g'\in T_{\tl w'}$,
then setting $\tl w:=\tl w'|_{\tl w}$, it follows that
$g\in T_{\tl w}$. Hence letting $K':=M^{\tl g}$ be the 
fixed field of $\tl g$ in $M$, and replacing $\tl K|K$ by 
$\tl K'|K'$, we can suppose that from the beginning 
we have $M\subset \tl K$, and $\tl K|K$ is cyclic with 
Galois group $\tl G=\langle \tl g\rangle$, and $M|K$
is a cyclic subextension, say with Galois group
$G=\langle g\rangle$, where $g=\tl g|_M$. And 
notice that $\tl G=T_{\tl x}$ and $G=T_z$. Thus 
$G$ acts on the local ring $\clO_z$ of $z$, and 
$\tl G$ acts on the local ring $\clO_{\tl x}$ of $\tl x$.
\vskip3pt
\vskip3pt
\underbar{Case 1}. $K=M$. Then $\XX=\ZZ$ 
is a projective smooth model for $K|k$, thus 
$\clO\!_x=\clO_z$ is a regular ring. Let $\nx{\clO_{\!x}}$
be the normalization of $\clO\!_x$ in $K\hra\tl K$,
and ${\rm Spec}\,\nx{\clO_{\!x}}\to{\rm Spec}\,\clO\!_x$
the restriction of $\tl \XX\to \XX$ to ${\rm Spec}\,\nx{\clO_{\!x}}$.
Since ${\rm Spec}\,\clO\!_x$ is regular, by the 
purity of the branch locus it follows that $T_{\tl x}$
is generated by inertia groups of the form $T_{\tl w}$,
with $\tl w$ prime divisors of $\tl K|k$ having the
center on ${\rm Spec}\,\nx{\clO_{\!x}}\subset\tl\XX$.
Since $T_{\tl x}$ is cyclic, it follows that there exists
$\tl w$ with $T_{\tl x}\subseteq T_{\tl w}$, etc.
\vskip2pt
\underbar{Case 2}. $K\subset M$ strictly. 
Then letting $(\clO,\eum)$ be the local ring 
$(\clO_z,\eum_z)$ at $z$, one has $T_z=G$. Then
proceeding as in the in the proof of~Theorem~B, 
explanations after~Fact~2.2, (the proof of)~Lemma~2.4
of loc.cit.\ is applicable, and by~Step~3 of that proof, 
it follows by Lemma~2.6 of loc.cit.\ that there exists 
a local ring $(\clO'\!,\eum')$ which has the properties:
\vskip3pt
\noindent \ \ \ \
- $G$ acts on $(\clO'\!,\eum')$, and $(\clO'\!,\eum')$ 
dominates $(\clO_z,\eum_z)$.
\vskip2pt
\noindent \ \ \ \
- There exist local parameters $(t'_1,\dots,t'_d)$
of $\clO'$, and a primitive character $\chi$ of $G$
such 
\vskip0pt \ \ \ 
that $\sigma(t'_i)=t'_i$ for $1\leq i < d$ and 
$\sigma(t'_d)=\chi(\sigma)t'_d$.
%
\vskip3pt
\noindent
Since $\clO'$ dominates $\clO_z$, it follows that
$\clO^T:={\clO'}^G$ dominates $\clO_{\!z}^G=\clO\!_x$, 
and $\clO^T$ is a regular ring by~Step~4 of loc.cit. 
Further $G=T_z$ is a quotient of the ramification 
group $T_{\clO^T}$ of $\clO^T$ in $K\hra\tl K$, and
since $K\hra\tl K$ is cyclic of $\ell$-power order, it
follows that $T_{\clO^T}={\rm Gal}(\tl K|K)=T_{\tl x}$.
Thus $\clO^T$ has a unique prolongation $\tl\clO^T$
to $\tl K$. And since $\clO^T$ dominates $\clO\!_x$,
it follows that $\tl\clO^T$ dominates $\clO_{\tl x}$
(which is the unique prolongation of $\clO\!_x$ to $\tl K$).
On concludes in the same way as at~Case~1 above. 
\vskip5pt
\noindent
{\bf Claim 2.} 
{\it Let $\tl K|K$ be an abelian $\ell$-power Galois 
extension as above. Then a prime divisor~$w$ 
of $K|k$ ramifies in $\tl K|K$ iff $w$ has a prolongation 
$w_L$ to $L$ which ramifies in $\tl L|L$.\/}
\vskip5pt
{\it Proof of Claim 2.\/}
Let $\hat K|K$ be a (minimal) finite normal 
extension with $L,\tl K\subseteq\hat K$, and 
${\rm Gal}(\hat K|K)=:\hat G\to\tl G:={\rm Gal}(\tl K|K)$, 
$\hat g\mapsto g$, 
be the corresponding projection of Galois groups. 
Setting $\tl L=L\tl K$ inside $\hat K$, it follows that 
${\rm Gal}(\tl L|L)=:H \to \tl G$ is an isomorphism, and 
$\hat H:={\rm Gal}(\hat K|L) \to H$ is surjective,
because $\tl K$ and $L$ are linearly disjoint over 
$K$. Thus there exists a Sylow $\ell$-group 
$\hat H_\ell$ of $\hat H$ with $\hat H_\ell \to H$ 
surjective. Since $(\hat G:\hat H)=[L:K]$ 
is prime to $\ell$, it follows that $\hat H_\ell$ is a 
Sylow $\ell$-group of $\hat G$ as well. Let 
$\hat w|\tl w|w$ denote the prolongations of $w$ 
to $\tl K$, respectively $\hat K$, and $T_{\hat w}\to T_{\tl w}$ 
be the corresponding surjective projections of the 
inertia groups. For every $\sigma\in T_{\tl w}$, let 
$\hat\sigma\in T_{\hat w}$ be a preimage of 
order a power of $\ell$. Then $\hat\sigma$ is 
contained in a Sylow $\ell$-group of $\hat G$, 
hence there exists a conjugate $\hat\sigma^{\hat\tau}$
which lies in $\hat H_\ell$. Thus replacing $\hat w|\tl w|w$
by $\hat w^{\hat\tau}|\tl w^\tau|w$, and $\sigma$, $\hat\sigma$ 
by $\sigma^\tau\!$, respectively $\hat\sigma^{\hat\tau}\!$, 
without loss of generality, we can suppose that 
$\hat\sigma\in T_{\hat w}\subseteq\hat H_\ell$ is a 
preimage of $\sigma\in T_{\tl w}$. Thus setting
$\tl w_L:=\hat w|_{\tl L}$ and $w_L:=\hat w|_L$, 
it follows that $\tl w_L|w_L$ are prolongations 
of $\tl w|w$ to $\tl L$, respectively $L$. And 
taking into account that $\hat H \to \tl G$ factors 
as $\hat H \to H \to \tl G$, it follows that the image 
$\sigma_L\in H$ of $\hat\sigma\in\hat H$ under 
$\hat H\to H$ lies in $T_{\tl w_L}\subset H$. 
This concludes the proof of~Claim~2.
\vskip3pt
Coming back to assertion~1) of Proposition~\ref{Pivspi},
we prove the following more precise result: 
\begin{lem}
\label{XvsY}
Let $\tl K|K$ be a finite abelian pro-$\ell$ field extension, 
and set $\tl L:=L\tl K$. Then for the corresponding 
normalization $\tl\XX\to\XX$ and $\tl\YY\to\YY$, the 
following assertion are equivalent:
\vskip2pt
\begin{itemize}[leftmargin=35pt]
\item[{\rm i)}] $\tl\XX\to\XX$ is etale above $S'\!$.
\vskip2pt
\item[{\rm ii)}] All $v\in\DD{\bsl\XX{S'}}\cap D_0$
 are unramified in $\tl K|K$.
 \vskip2pt
 \item[{\rm iii)}] All $w\in\DD{\bsl\YY{T'}}$ are unramified 
 in $\tl L|L$.
 \vskip2pt
\item[{\rm iv)}] $\tl\YY\to\YY$ is etale above $T'\!$.
\end{itemize}
\end{lem}
{\it Proof of Lemma~\ref{XvsY}.\/} First,~i) implies~ii) in an 
obvious way. For the implication~ii)~$\Rightarrow$~iii), let 
$v_L\in\DD{\bsl\YY{T'}}$ have restriction $v:=(v_L)|_K$ to $K$.
Then $v\in D_0\cap\DD{\bsl\XX{S'}}$, as observed above.
On the other hand, by~Claim~2) above one has that $v$
does not ramify in $\tl K|K$ iff all its prolongations to $L$ 
do not ramify in $\tl L|L$. Hence $v_L$ does not ramify in
$\tl L|L$. Next, since $\YY$ is smooth, thus regular, 
assertions~iii),~iv) are equivalent by the purity of the branch 
locus. Thus it is left to prove that~iv) implies~i). By
contradiction, suppose that $\tl \XX\to\XX$ is branched at
some point~$x\in\bsl\XX{S'}$. Then by the discussion 
before~Claim~1 combined with~Claim~1, it follows that
there exists a prime divisor $w$ of $K|k$ which is ramified
in $\tl K|K$ and $x$ lies in the closure of the center~$x_w$
of $w$. Thus since $x\in\bsl\XX{S'}$ and $S'\subset\XX$ 
is closed, it follows that $x_w\in\bsl\XX{S'}$. By~Claim~2
there exists a prolongation $w_L$ of $w$ to $L$ which 
is ramified in $\tl L|L$, and let $y_w$ be the center
of $w_L$ on $\YY$. Then by the compatibility of
center of valuations with (separated) morphism,
it follows that $w_L$ and its restriction $w$ to $K$
have the same center $x_0$ on $X_0$, and that 
center satisfies $y_w\mapsto x_0$, $x_w\mapsto x_0$.
Now let $S'_0\subset X_0$ be the closed subset 
such that $S'\subset\XX$ is the preimage of $S'_0$, 
and let $T'\subset\YY$ be the preimage of $S'_0$. Then 
$x_0\in\bsl{X_0}{S'_0}$ (because $x\in\bsl\XX{S'}$),
and $y_w\in\bsl\YY{T'}$. Since $w_L$ is branched 
in $\tl\YY\to\YY$ and has center $y_L\in\bsl\YY{T'}$,
it follows that $\tl\YY\to\YY$ is not etale above $T'\!$,
contradiction! This concludes the proof of~Lemma~\ref{XvsY},
thus of~assertion~1) of Proposition~\ref{Pivspi}.
\vskip2pt
To 2): Let $S'_0\subset X_0$ be such that $S'\subset\XX$
is the preimage of $S_0$, and $T'\subset\YY$ be the
preimage of $S'_0$ in $\YY$. In the notations from
the proof of~assertion~2) above, suppose that all 
$v\in\DD{\bsl\XX{S'}}$ are unramified in $\tl K|K$. We
claim that all prime divisors $w$ of $K|k$ are unramified
in $\tl K|K$. Indeed, by~Claim~2 above and the discussion
there after, it follows that a prime divisor $w$ of $K|k$ 
is ramified in $\tl K|K$ iff $w$ has a prolongation $w_L$ 
to $L$ which is ramified in $\tl L|L$ iff there exists 
$v_L\in\DD\YY$ which is ramified in $\tl L|L$ iff the
restriction $v:=(v_L)|_K$ is ramified in $\tl K|K$. In 
other words, $\pia{1,\DD\YY}=\pia{1,L}$ if and only 
if $\pia{1,\DD{\bsl\XX{S'}}}=\pia{1,K}$, etc.
\end{proof} 
\vskip2pt
\noindent
$\bullet$ {\it Reviewing divisorial lattices\/}
\vskip3pt
Let $D=\DD X$ be a geometric set for $K|k$, where 
$X$ is a quasi projective normal model of  $K|k$. Then 
$\Gamma(X,\clO_X)^\times=:\Gm(D)$, the group of 
principal divisors $\clH_K(X):=K\tms/\hp\Gm(D)$, 
and $\Div(X)$ depend on $D$ only, and not on the 
specific $X$ with $D_X=D$. Hence the 
exact sequence $0\to\clH_K(X)\hor{{\rm div}\,}\Div(X)
             \hor{{\rm pr}\,}\euCl\,(X)\to 0$
depends only on $D$. To stress that, write:
\vskip7pt
\centerline{$
0\to\clH_K(D)\hor{{\rm div}\,}\Div(D)
             \hor{{\rm pr}\,}\euCl\,(D)\to 0$.}
\vskip7pt
\noindent
Tensoring the above exact sequence with $\lvZ/n$ for all \
$n=\ell^e$, $e>0$, one gets exact sequence of $\lvZ/n$-modules 
$
0\to{}_n\euCl(D)\to\clH_K(D)/n\to\Div(D)/n\to\euCl\,(D)/n\to 0,
$
where ${}_n\euCl(D)$ is the $n$-torsion of $\euCl(D)$.
Since $1\to\Gm(D)/k\tms\to K\tms/k\tms\to\clH_K(D)\to1$
is an exact sequence of free abelian groups, 
$1\to\Gm(D)/n\to K\tms/n\to\clH_K(D)/n\to1$ 
is exact as well. Hence if $U(D)/n\subset K\tms/n$ 
is the preimage of ${}_n\euCl(D)$ under 
$K\tms/n\to\clH_K(D)/n$, it follows that the resulting 
sequence $1\to\Gm(D)/n\to U(D)/n\to{}_n\euCl(D)\to0$
is exact, thus the above long exact sequence above give rise 
canonically to a long exact sequence:
\[
1\to U(D)/n\to K\tms/n\to \Div(D)/n\to\euCl\,(D)/n\to0\,.
\]
Taking projective limits over $n=\ell^e\to\infty$, we 
get the long exact sequence of $\ell$-adic modules:
\[
1\to\whU(D)\hra\whK\hhb2
   \hor{{\rm div}\,}\hhb2\whDiv(D)\hhb2
     \hor{{\rm pr}\,}\whCl\,(D)\to0\,.
 \leqno{\hhb{20}(\dag)}
\] 
A geometric set $D$ is called 
\defi{complete regular like}, if for every geometric 
set $\tl D\supseteq D$ one has that $\whU(D)=\whU(\tl D)$ 
and $\whCl(\tl D)\cong\whCl(D)\oplus\Zell^r$, where 
$r:=|\bsl{\tl D}D|$. The following hold:
\vskip2pt
\begin{itemize}[leftmargin=25pt]
\item[a)] The complete regular like sets of prime divisors 
are quite abundant, namely: First, every geometric set is 
contained in a complete regular like set $D$. Second, if $D$ 
complete regular like, then every geometric set containing 
$D$ is complete regular like as well.
\vskip2pt
\item[b)] For a complete regular like geometric 
set $D$, let $\euCl^0(D)\subseteq\euCl'(D)$ be the 
maximal divisible, respectively $\ell$-divisible subgroups 
of $\euCl(D)$. Then by structure of $\euCl(D)$,
see e.g.~[P3],~Appendix,~7.3, it follow that 
$\euCl'(D)/\euCl^0(D)$ is a (finite) torsion group 
of order prime to $\ell$. Thus if $\Div^0(D)\subseteq\Div'(D)$ 
are the preimages of $\euCl^0(D)\subseteq\euCl'(D)$
in $\Div(D)$, it follows that
$\Div^0(D)_\pel=\Div'(D)_\pel$ inside
$\Div(X)_\pel$.\footnote{Recall that for an abelian
group $A$, we denote $A_\pel:=A\otimes\lvZ_\pel$.} 
\vskip2pt
\item[c)] Both $\whU\!_K:=\whU(D)$ and 
$\Div'(D)_\pel$ are birational invariants of $K|k$. 
\end{itemize}
\vskip2pt
We let $\lxzl K\subset\whK$ be the preimage of 
$\Div'(D)_\pel\subset\whDiv(D)$ under 
${\rm div}:\whK\to\whDiv(D)$, and we call $\lxzl K$ 
the \defi{canonical divisorial $\whU\!_K$-lattice} for 
$K|k$. Notice that ${\rm div}(\lxzl K)=\Div'(D)_\pel$.
\vskip5pt
Next we recall the Galois theoretical counterpart
of the above exact sequence~$(\dag)$. 
\vskip5pt
Let $D$ be a geometric (complete regular like) set of 
prime divisors of $K|k$. Recall the closed subgroup \
$T_D\subseteq \Pi_K$ generated by $T_v$, $v\in D$,
and the resulting 
$1\to T_D\to\Pi_K\to\pia{1,D}\to1$. Interpreting
the elements of $\clP(K)=K\tms/k\tms\subset
\whK={\rm Hom}(\pia K,\Zell)$ as function on
$\pia K$, one gets: Each $T_v$ has a \defi{\it unique\/} 
generator $\tau_v\in T_v$, called \defi{canonical},
such that for uniformizing parameters $t_v\in\clO_v$ 
one has: $t_v(\tau_v)=1$. Then 
$\jmath^v:\whK={\rm Hom}(\Pi_K,\Zell)
    \to{\rm Hom}(T_v,\Zell)=\Zell$,
$\varphi\mapsto\varphi|_{T_v}$, is identified with the 
$\ell$-adic completion of the valuation $v:K\tms\to\lvZ$. 
\vskip2pt
Finally, the system of the maps $\jmath^v:\whK\to\Zell$,
$v\in D$, gives rise by restriction canonically~to
\[
\jmath^D=\oplus_v\,\jmath^v:\lxzl K\to\Div_D:=\oplus_v\,\Zell\,.
\]
Let $\euF_D$ be the free abelian pro-$\ell$ 
group on $(\tau_v)_{v\in D}$. The canonical projection
$\euF_D \to T_D\subseteq\Pi_K$ gives rise to a long exact 
sequence $1\to\euR_D\to\euF_D\to\Pi_K\to\pia{1,K}\to1$ 
of pro-$\ell$ abelian groups, where $\euR_D\subset\euF_D$ 
is the relation module for the system of inertia generators 
$(\tau_v)_v$, and the first three terms of the exact 
sequence have no torsion. Since the $\ell$-adic 
dual of $\euF_D$ is canonically isomorphic to 
$\whDiv_D$, by taking $\ell$-adic duals, on gets an 
exact sequence:
\[
1\to\whU\!_D\hra\whK\hor{{\widehat\jmath^{\hhb2D}}}
   \whDiv\hm1_D\hor{\rm can\hhb4}\whCl\hp_D\to0,
\leqno{\hhb{20}(\dag)_{\rm Gal}}
\]     
in which each of the terms is the $\ell$-adic dual of the 
corresponding pro-$\ell$ group, $\widehat\jmath^{\hhb2D}$ 
is the $\ell$-adic completion of the map $\jmath^{\hhb1D}$
given above, and the other morphisms are canonical.
\vskip2pt
Now recalling that $D=D_X$ for some quasi projective 
normal model of $K|k$, it is obvious that the above exact
sequences~$(\dag)$ and~$(\dag)_{\rm Gal}$ are canonically
isomorphic, and in particular one has $\whU\!_D=\whU(D)$
inside $\whK$, and $\whCl_D\cong\whCl(D)$ canonically. 
And notice that the exact sequence~$(\dag)_{\rm Gal}$ 
was constructed from $\Pi_K$ endowed with the canonical 
system $(\tau_v)_{v\in D}$ of the  inertia groups $(T_v)_{v\in D}$ 
only. {\it Unfortunately,\/} we do no have a group theoretical 
recipe (yet?) to identify the canonical system of inertia 
generators $(\tau_v)_v$ ---\hhb1say, up to simultaneous  
raising to some power $\epsilon\in\Zell\tms\,$, from 
$\pic K$ endowed with $\clG_{\tottrK}$. Nevertheless, 
using the theory of \defi{geometric decomposition 
graphs} as developed in~[P3], one proves the following:
\begin{prop}
\label{propSubsecC}
In the above notations, the following hold:
\vskip3pt
\begin{itemize}[leftmargin=25pt]
\item[{\rm1)}] There exist group theoretical recipes to
recover{\rm/}reconstruct from $\clG_{\tottrK}$ the following:
\vskip2pt
\begin{itemize}[leftmargin=20pt]
\item[{\rm a)}] The geometric sets of prime divisors 
$D$ of $K|k$. Moreover, the group theoretical recipes single out 
the complete regular like sets of prime divisors. 
\vskip2pt
\item[{\rm b)}] Given a complete regular like 
set of prime divisors $D$, the recipes reconstruct 
the divisorial $\whU\!_K$-lattices of the form 
$\epsilon\cdot\lxzl K\subset\whK$,  or equivalently, 
the subgroups of the form 
$\epsilon\cdot\Div'(D)_\pel\subset\whDiv(D)$, and 
systems of inertia generators 
$(\tau_v^\epsilon)^{\phantom{'}}_{v\in D}$
for all $\epsilon\in\lvZ_\ell\tms/\lvZ_\pel\tms\,$. 
Thus finally the recipes reconstruct the exact 
sequences of the form 
\[
1\to\whU\hmm_K\to\epsilon\cdot\lxzl{K}
   \hor{\widehat\jmath^{\hhb2D}}\epsilon\cdot
     \Div(D)_\pel\hor{{\rm can}\hhb3}\whCl\hhb1(D),
       \quad\epsilon\in\Zell\tms\,.
\]
\item[{\rm c)}] Replacing any complete regular 
like set $D$ by any geometric subset $D'\subset D$, 
one gets by restriction the corresponding
$1\to\whU\!_{D'}\to\epsilon\cdot\lxzl{D'}
   \hor{\widehat\jmath_{\!D'}}\epsilon\cdot\Div(D')_\pel
      \horr{{\rm can}_{D'}\,}\whCl\hhb1(D')$.
\end{itemize}
\vskip2pt
\item[{\rm2)}] The recipes to recover the above 
information from $\clG_{\tottrK}$ are invariant 
under isomorphisms of total decomposition graphs 
$\Phi:\clG_{\tottrK}\to\clG_{\tottrL}$ as follows: 
$\Phi$ defines bijections $D_K\to D_L$, say
$v\mapsto w$, between the $($complete regular 
like$\hhb1)$ geometric sets $D_K$ for $K|k$, 
and $D_L$ for $L|l$. Further, there exist 
$\epsilon\in\Zell\tms$ and 
$\epsilon_v\in\epsilon\cdot\lvZ_\pel\tms\,$, $v\in D_K$,
satisfying: 
\vskip2pt
\begin{itemize}[leftmargin=30pt]
\item[{\rm a)}]  
$\Phi(\tau_v)_{v\in D_K}=(\tau_w^{\epsilon_v})^{\phantom{'}}_{w\in D_L}$,
hence $\Phi$ gives rise to an isomorphism of $\lvZ_\pel$-modules
$\Div(D_L)_\pel\to\epsilon\cdot\Div(D_K)_\pel$ defined by 
via $\jmath^w\mapsto\epsilon_v\cdot\jmath^v$ for $v\mapsto w$.
\vskip2pt
\item[{\rm b)}] The Kummer isomorphism $\whph:\whL\to\whK$ 
of $\,\Phi$, i.e, the $\ell$-adic dual of $\,\Phi:\pia K\to\pia L$, satisfies:
$\whph(\whU\!_L)=\whph(\whU\!_K)$, $\whph(\lxzl L)=\epsilon\cdot\lxzl K$.
\end{itemize}
\end{itemize}
\end{prop}
\begin{proof} First, the assertion~1) follows from~[P3], 
Propositions~22 and~23 (where the above facts were 
formulated in terms of geometric decomposition groups, 
rather then geometric sets of prime divisors). Second, for 
assertion~2), see~[P3], Proposition~30, etc.
\end{proof}
\vskip2pt
\noindent
%
%
%
%
D) {\it Specialization techniques\/}
\vskip5pt
We begin by recalling briefly the specialization/reduction 
results concerning projective integral/normal/smooth 
$k$-varieties and (finite) $k$-morphisms between such 
varieties, see \nmnm{M}{umford}~[Mu],~II,~\S\S\hhb{2}7--8, 
\nmnm{R}{oquette}~[Ro], and 
\nmnm{G}{rothendieck}--\nmnm{M}{urre}~[G--M]. 
We first recall the general facts and then come back with 
specifics in our situation.
We consider the following context: $k$ is an algebraically 
closed field, and $\Val k$ is the space of all the valuations 
$v$ of $k$. For $v\in\Val k$ we denote by $\clO_v\subset k$ 
its valuation ring, and let $\clO_v\to kv$, $a\mapsto\oli a$, 
be its residue field. One has the \defi{reduction map}
$\clO_v[T_1,\dots,T_N]\to kv[T_1,\dots,T_N]=:\clR_v$,
$f\mapsto\oli f$, defined by mapping each coefficient of 
$f$ to its residue in $kv$. In particular, the reduction map 
above gives rise to a reduction of the ideals $\eua\subset
k[T_1,\dots,T_n]$ to ideals $\eua_v\subset\clR_v$ defined 
by $\eua\mapsto\euA:=\eua\cap\clO_v[T_1,\dots,T_n]$
followed by $\euA\to\eua_v$, $f\mapsto\oli f$. Finally, 
we recall that $\Val k$ carries the Zariski topology 
$\tau^{_{\rm Zar}}$ which has as a basis the subsets 
of the form 
\[
\clU_A=\{v\in\Val k\mid v(a)=0\ \forall a\in A\}, \quad 
                 A\subset k\tms \ \hbox{finite subsets}.
\]
We notice the trivial valuation $v_0$ on $k$ belongs to every 
Zariski open non-empty set, thus $\tau^{_{\rm Zar}}$ is a 
prefilter on $\Val k$. We will denote by $\euD$ ultrafilters 
on $\Val k$ with $\tau^{_{\rm Zar}}\subset\euD$.  
%
%
\begin{remarks}
\label{Zarlocal}
In the above notations, we consider/remark the following:
\vskip2pt
\begin{itemize}[leftmargin=25pt]
\item[{1)}] Let $\str1k:=\prod_v kv/\euD$ be the 
\defi{ultraproduct} of $(kv)_v$, and consider the 
\defi{canonical embedding} of $k$ into $\str1k$, defined 
by $a\mapsto(a_v)/\euD$, where $a_v=\oli a$ if $a$ is 
a $v$-unit, and $a_v=0$ else. 
\vskip2pt
\item[{2)}] Let $\eua:=(f_1,\dots,f_r)\subset R:=k[T_1,\dots,T_N]$ 
be an ideal. The the fact that $\eua$ is a prime ideal, 
respectively that $V(\eua)={\rm Spec}\,R/\eua\subseteq\lvA^N_k$ 
is normal/smooth is an open condition involving the coefficients 
if $f_1,\dots, f_r$ as parameters. Therefore, there exists a 
Zariski open subset $\clU_\eua\subset\Val k$ such that for 
all $v\in\clU_\eua$ the following hold:
\vskip2pt
\begin{itemize}[leftmargin=25pt]
\item[{a)}] $f_1,\dots,f_r\in\euA$, and $\eua_v=(\oli f_1,\dots,\oli f_r)$.
\vskip2pt
\item[{b)}] If $\eua$ is prime, then so is $\eua_v=(\oli f_1,\dots,\oli f_r)$.
\vskip2pt
\item[{c)}] If $V(\eua)\hra\lvA^N_k$ is a normal/smooth $k$-subvariety,
then so is $V(\eua_v)\hra\lvA^N_{kv}$. 
\end{itemize}
\end{itemize}
\end{remarks}
\begin{definition/remark}
\label{nonstandard} 
In the context of Remark~\ref{Zarlocal}, we introduce 
notations as follows:
\vskip2pt
\begin{itemize}[leftmargin=25pt]
\item[{1)}] $\eua_*:=\prod_v\eua_v/\euD\subset
    \prod_v\clR_v=:\clR_*$
is the $\euD$-\defi{ultraproduct} of the ideals $(\eua_v)_v$. 
One has a \defi{canonical embedding} 
$\eua\hra\eua_*$ defined by $f\mapsto f\!_*$, where 
$f\!_*$ is the image of $f$ under $R\hra\clR_*$.
In particular, $\eua=\eua_*\cap R$ inside $\clR_*$. 
\vskip2pt
\item[{2)}] $\eua k_*=(f_{1*},\dots,f_{r*})\subset 
  Rk_*=k_*[T_1,\dots,T_N]$ is called the \defi{finite part} 
of $\eua_*$. Notice that $V(\eua k_*)\hra\lvA^N_{k_*}$ 
is nothing but the base change of $V(\eua)\hra\lvA^N_k$ 
under $k\hra k_*$.
\vskip2pt
\item[{3)}] Finally, if $\eua$ is a prime ideal, then so
$\eua_*$, and the canonical inclusion of $k$-algebras 
$R/\eua\hra \clR_*/\eua_*$ gives rise to inclusions of 
their quotient fields $\kappa(\eua)\hra\kappa(\eua_*)$,
and one has: The algebraic closure of $\oli{\kappa(\eua)}$ 
of $\kappa(\eua)$ and $\kappa(\eua_*)$ are linearly 
disjoint over $\kappa(\eua)$. 
\end{itemize}
\end{definition/remark}

\vskip2pt
Let $\zX$ be a projective reduced $k$-scheme, and 
$\zX={\rm Proj}\,k[T_0,\dots,T_N]/\eup\hra\lvP^N_k$
be a fixed projective embedding, where 
$\eup=(f_1,\dots,f_r)\subset k[T_0,\dots,T_N]$ is a homogeneous prime 
ideal. Then $\euP:=\eup\cap\clO_v[T_0,\dots,T_N]$ is a 
homogeneous prime ideal of $\clO_v[T_0,\dots,T_N]$ and therefore, 
$\zclX:={\rm Proj}\,\clO_v[T_0,\dots,T_N]/\euP\hra\lvP^N_{\clO_v}$
is a projective $\clO_v$-scheme, etc.
\begin{fact} 
\label{fctunu}
In the above notations,  the following hold:
\vskip2pt
\begin{itemize}[leftmargin=30pt]
\item[{1)}] $\zclX\hra\lvP^N_{\clO_v}$ is the 
scheme theoretical closure of $\zX$ under 
$\zX\hra\lvP^N_k\hra\lvP^N_{\clO_v}$, and
$\zX\hra\lvP^N_k$ is the generic fiber of
$\zclX\hra\lvP^N_{\clO_v}$. Further, the special 
fiber $\zclX_v\hra\lvP^N_{kv}$ is reduced.
\vskip2pt
\item[{2)}] If $\zX=\cup_i\,\zX_i$ with $\zX_i$ 
closed subsets, then $\zclX=\cup_i\,\zclX_i$, and 
$\zclX_v=\cup_i\zclX_{i,v}$. Further, if $\zX$ is 
connected, so is $\zclX_v$, and if $\zX$ is integral, 
then $\zclX_v$ is of pure dimension equal to~$\dim(\zX)$.
\vskip2pt
\item[{3)}]  If $\zX$ is integral and normal, then the 
local ring at any generic point $\eta_i$ of $\zclX_v$ 
is a valuation ring $\clO_{v_{\eta_i}}$ of $k(\zX)$ 
dominating $\clO_v$. 
\end{itemize}
\end{fact}
%
%
\begin{definition}
$\hhb1$
\vskip2pt
\begin{itemize}[leftmargin=30pt]
\item[{\rm 1)}] The valuation $v_{\eta_i}$ with valuation
ring $\clO_{v_{\eta_i}}$ introduced at Fact~\ref{fctunu},~c),
above is called a \defi{Deuring constant reduction} of
$k(\zX)$ at $v$.
\vskip2pt
\item[{\rm 2)}] 
The $v$-\defi{reduction/specialization map} for closed 
subsets of $\zX$ is defined by
\[
\spsp_v:\{S\subseteq\zX\mid S \ \hbox{closed}\}\to
\{\clS'_v\subseteq\zclX_v\mid \clS'_v \ \hbox{closed}\}, \ \
S\mapsto\clS_v.
\] 
\item[{\rm 3)}] In particular, if $P\subset \zX$ is a prime 
Weil divisor, then $\clP\subset\zclX$ is a relative 
Weil divisor of $\zclX$. Finally, $\spsp_v$ gives rise to a 
\defi{Weil divisor reduction/specialization homomorphism}
\[
\spsp_v:\Div(\zX)\to\Div(\zclX_v),\quad 
   P\mapsto{\textstyle\sum}_i\,\clP_{v,i}
\] 
where $\clP_{v,i}$ are the irreducible components 
of the special fiber $\clP_v$ of $\clP$.
\end{itemize}
\end{definition}
\begin{fact}
\label{fctdoi}
There exists a Zariski open nonempty set $\clU\subset\Val k$ 
such that each $v\in\clU$ satisfies:
\vskip2pt
\begin{itemize}[leftmargin=30pt]
\item[{a)}] If $\zX$ is integral/normal/smooth, then so are
$\zclX$ and $\zclX_v$. In particular, if 
$\zX$ is integral and normal, then $\zclX_v$ is integral 
and normal, and $v_{\eta_v}$ is called the \defi{canonical 
(Deuring) constant reduction} of $k(\zX)$ at $v$, an we 
denote it by $\vcr{\hp k(\zX)}$.
\vskip2pt
\item[{b)}] If $S_1,\dots,S_r\subseteq\zX$ are distinct  
closed subsets of $\zX$, then $\spsp_v(S_1),\dots,\spsp_v(S_r)$ 
are distinct closed subsets of $\zclX_v$. And if $\clP=\cup_i\, P_i$ 
is a NCD (normal crossings divisor) in~$\zX$, then so is 
$\clP_v=\cup_i\, \spsp_v(P_i)$ in $\zclX_v$.
\vskip2pt
\item[{c)}] $\spsp_v:\Div(\zX)\to\Div(\zclX_v)$ is compatible 
with principal divisors, and therefore gives rise to a 
\defi{reduction/specialization homomorphism} 
$\spsp_v:\euCl(\zX)\to\euCl(\zclX_v)$.
\vskip2pt
\item[{d)}] If $H\subset\lvP^N_k$ is a general hyperplane,
then $\clH\subset\lvP^N_{\clO_v}$ is a general relative 
hyperplane and $\clH_v\subset\lvP^N_{kv}$ is a general 
hyperplane, and $(\clH\cap\zclX)_v=\clH_v\cap\zclX_v$
inside $\lvP^N_{kv}$.
\vskip2pt
\item[{e)}] Let $\lvG_m(\zX)\subset k(\zX)$ be the group of
invertible global section on $\zX$, and $\lvG_m(\zclX_v)$ be 
correspondingly defined. Then $\lvG_m(\zX)/k\tms\to\lvG_m(\zclX_v)/kv\tms$
is an isomorphism.
\end{itemize}
\end{fact}
In the above context and notations, let $\zY\to\zX$
be a dominant morphism of projective integral 
normal $k$-varieties, which is generically finite, and
$K:=k(\zX)\hra k(\zY)=:L$ be the corresponding
finite field extension. Let $\wtl L|L$ be a finite
Galois extension with $[\wtl L:L]$ relatively prime to
$[L:K]$, and $\wtl K|K$ be the maximal Galois
subextension of $K\hra \wtl L$ having degree
$[\wtl K:K]$ relatively prime to $[L:K]$. Then
the canonical projection of Galois gropups 
$H:={\rm Gal}(\wtl L|L)\to
{\rm gal}(\wtl K|K)=:G$ is surjective. Further, let
$\wtl\zX\to\zX$ and $\wtl\zY\to\zY$ be the normalizations
of $\zX$ in $K\hra\wtl K$, respectively of $\zY$ 
in $L\hra\wtl L$. Then by the universal property
of normalization, $\wtl\zY\to\zY\to\zX$ factors as
$\wtl\zY\to\wtl\zX\to\zX$. 
\vskip2pt
Finally, let $\zX\hra\lvP^N_k$, etc., be projective 
$k$-embeddings, and consider the corresponding 
projective $\clO_v$-schemes $\zX\hra\zclX$, etc.,
as defined above. Then by general scheme theoretical 
non-sense, $\zY\to\zX$, etc., give rise to dominant
rational $\clO_v$-maps $\zclY\ratmap\zclX$, etc.
%
%
%
%
\begin{fact}
\label{fcttrei} 
There exists a Zariski open nonempty set $\clU\subset\Val k$ 
such that each $v\in\clU$ satisfies:
\vskip2pt
\begin{itemize}[leftmargin=30pt]
\item[{a)}] The rational maps $\clO_v$-maps 
$\zclY\ratmap\zclX$ above are actually morphisms 
of $\clO_v$-schmes, and one has commutative 
diagrams of morphisms of projective $\clO_v$-schemes
\[
\begin{matrix}
\wtl\zY\hhb{7}\to\hhb{7}\zY 
       &\hhb{20}&\wtl\zclY\hhb{7}\to\hhb{7}\zclY  
                &\hhb{20}&\wtl\zclY_v\hhb{7}\to\hhb{7}\zclY_v  \cr
\downarrow\hhb{33}\downarrow&&\downarrow
      \hhb{35}\downarrow&&\downarrow\hhb{37}\downarrow \cr 
\wtl\zX\hhb{9}\to\hhb{9}\zX 
       &\hhb{20}&\wtl\zclX\hhb{11}\to\hhb{10}\zclX  
                &\hhb{20}&\hhb3\wtl\zclX_v\hhb{9}\to\hhb{10}\zclX_v  \cr
\end{matrix}
\]
where the LHS diagrams is the generic fiber of the middle one, 
respectively RHS diagrams is the special fiber of the middle one.
\vskip2pt
\item[{b)}] Moreover, the degrees of the morphism which
correspond to each other in the above diagrams are 
equal. Thus the canonical constant reductions $\vcr{L}$, 
$\vcr{\tl K}$, $\vcr{\tl L}$ are the unique prolongations 
of $\vcr{K}$ to $L$, $\tl K$, and $\tl L$, respectively. 
\vskip2pt
\item[{c)}] The residue field extension 
$\tl K\!\vcr{\tl K}|K\!\vcr{K}$ is the maximal Galois subextension 
of $\tl L\vcr{\tl L}|L\vcr{L}$ which has degree relatively prime 
to $[L:K]=[L\vcr L|K\!\vcr K]$.
\end{itemize}
\end{fact}
\begin{fact}
\label{fctpatru}
Let $\zY$ be a projective smooth $k$-variety, and 
$T' \subset \zY$ be either empty, or the support 
of a normal crossings divisor in $\zY\!$, and set
$\zclY':=\bsl{\zY\hhb1}{T'}$. For $v\in\Val k$ define 
correspondingly $\clT'\hra\zclY$, $\clT'_v\hra\zclY_v$,
thus $\zclY'=\bsl{\zclY\hhb1}{\clT'}$, 
$\zclY'_v=\bsl{\zclY_v}{\clT'_v}$,
and let $\DD{\zY'}$ and $\DD{\zclY'_v}$ be the 
corresponding geometric sets of Weil prime divisors. 
Then there exists a Zariski open nonempty set 
$\clU\subset\Val k$ such that for $v\in\clU$ the
special fiber $\zclY_v$ is smooth, $\clT'_v$ either
empty of a NCD, and the following hold: 
First, by the purity of the branch locus, and second, by 
\nmnm{G}{rothendick}--\nmnm{M}{urre}~[G--M] 
applied to $\zclY'_v$, the canonical maps of 
pro-$\ell$ abelian fundamental groups below 
are isomorphisms:
\[
\pia{1,D_{\zY'}} \to\Pi_1(\zY')\to\Pi_1(\zclY')\ot
  \Pi_1(\zclY'_v)\ot\pia{1,D_{\zclY'_v}}.
\] 
\end{fact}
We now come back to the context of 
Theorem~\ref{mainthm}. Recall the notations and the 
context from~Preparation/Notations~\ref{prepfct}, and
the $k$-morphisms $\XX\to X_0\ot\YY$, $\ZZ\to\XX$,
further the closed subset $S_0\subset X_0$ and the 
algebraic set of prime divisors $D=\DD{\bsl{X_0}{S_0}}$,
and finally, the preimages $S\to S_0\ot T$ of $S_0$ 
under $\XX\to X_0\ot\YY$. Finally, by Proposition~\ref{Pivspi}, 
one has canonical identifications, respectively a surjective 
canonical projection as below:
\[
\pia{1,D_{\bsl X{S}}}=\Pi_1(\bsl X{S})
\to \Pi_1(X)=\pia{1,\DD X}=\pia{1,K}.
\] 
After choosing projective embeddings for each of the
$k$-varieties $X_0$, $X$, $Y\!$, $Z$, we get for every 
$v\in\Val k$ the corresponding $\clO_v$-schemes 
$X_0\hra \clX_0 \hla \clX_{0v}$, etc., and for every of 
the $k$-morphisms above, we get corresponding 
dominant rational $\clO_v$-maps, etc. Then applying~Facts
\ref{fctunu}--\ref{fctpatru}, one has the following:
\begin{fact/nota}
\label{factnotations}
There exists a Zariski open subset $\clU\subset\Val k$ 
of valuations $v$ satisfying the following: The special
fiber of each of the above $k$-varieties is irreducible
and normal/smooth if the generic fiber was so. Further,  
all the dominant rational $\clO_v$-maps under discussion above 
are actually $\clO_v$-morphisms such that the degrees 
of the generic fiber $k$-morphisms and the corresponding
special fiber $kv$-morphisms are equal. In particular,
the canonical constant reductions $\vcr L$ and $\vcr M$
are the unique prolongations of $\vcr K$ to $L$, 
respectively~$M$. Finally, let $S'_0\subset X_0$ be either empty 
or equal to $S_0$, hence its preimages $S'\subset X$ 
and $T'\subset Y$ are either empty or equal to $S$, 
respectively $T$. We set $X'_0:=\bsl{X_0}{S'_0}$,
$X':=\bsl X{S'}$, $Y':=\bsl Y{T'}$, and for $v \in \clU$ 
consider the corresponding $X_0\to\clX_0\ot\clX_{0v}$, 
$S'_0\to\clS'_0\ot\clS'_{0v}$ and the resulting 
$X'_0\to\clX'_0\ot\clX'_{0v}$, etc. Then the resulting 
$kv$-morphisms satisfy the following: 
\vskip2pt
\begin{itemize}[leftmargin=30pt]
\item[{a)}] $\clS_v\subset\clX_v$, $\clT_v\subset\clY_v$ are the 
preimages of $\clS_{0v}$ under $\clX_v\to\clX_{0v}\ot\clY_v$,  
and the restriction $\clD_0$ of $\DD{\clY_v}$ to $K\hm1\vcr K$ 
satisfies $\DD{\clX_{0v}}\subseteq\clD_0\subseteq\DD{\clX_v}$.
\vskip2pt
\item[{b)}] Further, setting $\clD:=\DD{\bsl{\clX_{0v}}{\clS_{0v}}}$,
one has that $\clD\subseteq\spsp_v(D)$, and further:
\vskip2pt
\begin{itemize}[leftmargin=10pt] 
\item[{-}] $\clY_v\to\clX_{0v}$ is a prime to $\ell$ 
alteration above $\clS_{0v}$, hence $\clT_v$
is NCD in $\clY_v$. 
\vskip2pt
\item[{-}] $\clX_v\to \clX_{0v}$ are projective normal 
models of $K\hm1\vcr K|kv$, and  
$\clD\subseteq\clD_0\subseteq\DD{\clX_v}$. 
\vskip2pt
\item[{-}] $\clZ_v\to\clX_v$ is a finite generically normal
alteration with ${\rm Aut}(M\hm1\vcr M|K\hm1\vcr K)={\rm Aut}(M|K)$.
\end{itemize}
\vskip2pt
\item[{c)}] Applying Proposition~\ref{Pivspi} to the fibers, one has:
\vskip2pt
\begin{itemize}[leftmargin=10pt]
\item[{-}] The canonical maps $\pia{1,\DD{X'}}\to\Pi_1(X')$
and $\pia{1,\DD{\clX'_v}}\to\Pi_1(\clX'_v)$ 
are isomorphisms. 
\vskip0pt
\item[{-}] $\pia{1,\DD{X}}\to\Pi_1(X)\to\pia{1,K}$
and $\pia{1,\DD{\clX_v}}\to\Pi_1(\clX_v)\to\pia{1,K\hm1\vcr K}$ 
are isomorphisms. 
\end{itemize}
\vskip2pt
\item[{d)}] By the functoriality of fundamental group
and Fact~\ref{fctpatru} one has:
\vskip2pt
\begin{itemize}[leftmargin=10pt]
\item[{-}] $X'\hra\clX'\hla\clX'_v$ give rise 
to surjective projections 
$\Pi_1(X')\to\Pi_1(\clX')\ot\Pi_1(\clX'_v)$.
\vskip2pt
\item[{-}] $Y'\hra\clY'\hla\clY'_v$ give 
rise to canonical isomorphisms 
$\Pi_1(Y')\to\Pi_1(\clY')\ot\Pi_1(\clY'_v)$.
\end{itemize}
\end{itemize}
\end{fact/nota}
\begin{prop}
\label{lemaunu}
The canonical maps
$\Pi_1(X')\to\Pi_1(\clX')\ot\Pi_1(\clX'_v)$ 
are isomorphisms.
\end{prop}
\begin{proof} We consider the
canonical commutative diagram of $\clO_v$-morphism, 
and the corresponding maps of fundamental groups:
\vskip-5pt
\[
\begin{matrix}
&\hhb{10}Y'&\hra\hhb7\clY'\hhb7\hla&\clY'_v 
&\hhb{30}&&
\Pi_1(Y')&\to\hhb7\Pi_1(\clY')\hhb7\ot&\Pi_1(\clY'_v)
\cr
(*)\hhb{-10}&\hhb{10}\downarrow & \downarrow\hhb{-2} &\downarrow\hhb{-1} 
&&(\dag)&      
\downarrow &       \downarrow\hhb{-2} &\downarrow  \hhb{-1} 
\cr 
&\hhb{10}X'&\hra\hhb7\clX'\hhb7\hla&\clX'_v  
&&&
\Pi_1(X')&\to\hhb7\Pi_1(\clX')\hhb7\ot&\Pi_1(\clX'_v)
\cr
\end{matrix}
\]
\vskip2pt
\noindent
The homomorphisms in these diagrams satisfy:
First, since the vertical morphism in the diagram~$(*)$
are finite, having degree equal to $[L:K]$, the 
vertical homomorphisms in the diagram~$(\dag)$ have 
open images of index dividing $[L:K]$. Thus since 
$[L:K]$ is prime to~$\ell$, it follows that the vertical 
maps in the diagram~$(\dag)$ are actually surjective.
Further, the morphisms in the first row of the  
diagram~$(\dag)$ are isomorphisms, and those of 
the second row are surjective 
by~Fact/Notations~\ref{factnotations},~d). We have to
prove that the homomorphisms in the second 
row of the diagram~$(\dag)$ are isomorphisms too. 
\vskip5pt
{\bf Claim 1.} {\it $\Pi_1(X')\to\Pi_1(\clX')$ is injective.\/}
\vskip3pt
\noindent
Indeed, let $\tl X\to X'$ any finite abelian $\ell$-power
degree etale cover of $X'$, $\tl K=k(\tl X)$ its function
field, and $\tl\clX\to\clX'$ the normalization of $\clX'$
in $K\hra\tl K$. We claim that $\tl\clX\to\clX'$ is etale.
Since being etale is an open condition, it is sufficient
to prove that the cover of the special fiber $\tl\clX_v\to\clX'$
is etale. (Notice that we do not know yet whether $\tl\clX_v$
is reduced.) Let $\tl L:=L\tl K$, and $\tl Y\to Y'$ be 
the normalization of $Y'$ in $L\hra\tl L$. Since 
$\Pi_1(Y')\to\Pi_1(X')$ is surjective, and $\tl X\to X'$ 
is etale, it follows that ${\rm Gal}(\tl L|L)$ is a quotient 
of $\Pi_1(Y')$. Thus the normalization $\tl\clY\to\clY'$ 
of $\clY'$ in $L\hra\tl L$ is etale. Further, since the 
morphisms in the first row of the~diagram~$(\dag)$ 
are isomorphisms, it follows that the corresponding 
$\tl\clY\to\clY'$, $\tl\clY_v\to\clY'_v$ are etale covers 
of integral $\clO_v$-schemes satisfying
$\deg(\tl Y\to Y)=\deg(\tl\clY_v\to\clY_v)$. In particular,
the canonical constant reduction $\vcr L$ has a unique
prolongation to $\vcr{\tl L}$ to $\tl L$, and $\vcr{\tl L}|\vcr L$
is totally inert. Since $\vcr L|\vcr K$ is totally inert by
the fact that $v\in\clU$, it follows that $\vcr{\tl L}|\vcr K$ 
is totally inert as well. Thus $\vcr K$ has a unique 
prolongation $\vcr{\tl K}$ to $\tl K$ as well, and 
$\vcr{\tl K}|\vcr K$ is totally inert, i.e., ${\rm Gal}(\tl K|K)$
equals the decomposition group above $\vcr K$, and
the inertia group above $\vcr K$ is trivial. Equivalently,
$\tl\clX_v$ is integral, and $K\!\vcr K=\kappa(\clX'_v)\hra
\kappa(\tl\clX_v)=\tl K\vcr{\tl K}$ is Galois extension
with ${\rm Gal}(\tl K\vcr{\tl K}|K\!\vcr K)={\rm Gal}(\tl K|K)$.
Hence we have the following: $\tl K\!\vcr{\tl K}|K\vcr K$
is an abelian $\ell$-power degree extension such 
that $\tl L\vcr{\tl L}$ is the compositum 
$\tl L\vcr{\tl L}=L\vcr L \tl K\!\vcr{\tl K}$, and 
$\tl\clY_v\to\clY'_v$ is etale. By~Proposition~\ref{Pivspi}, 
it follows that that $\tl\clX_v\to\clX'_v$ is etale as well.
Thus $\tl\clX\to\clX'$ is etale, and~Claim~1 is proved.
Conclude that $\Pi_1(X')\to\Pi_1(\clX')$ is an isomorphism. 
\vskip4pt
{\bf Claim 2.} {\it $\Pi_1(\clX'_v)\to\Pi_1(\clX')$ is injective.\/} 
\vskip4pt
\noindent
In order to prove the claim, via the canonical
isomorphisms $\Pi_1(Y')\to \Pi_1(\clY')\ot\Pi_1(\clY'_v)$, 
we can identify these groups with a finite $\Zell$-module
$\Pi$, thus the canonical surjective projections
$\Pi\to\Pi_1(\clX'_v)\to\Pi_1(\clX')$ are defined as
quotients of $\Pi$ by  
$\Delta:=\ker\big(\Pi\to\Pi_1(\clX')\big)$ and 
$\Delta_v:=\ker\big(\Pi\to\Pi_1(\clX'_v)\big)$, which are 
finite $\Zell$-submodules of $\Pi$. The following 
conditions are obviously equivalent:
\vskip2pt
\begin{itemize}
\item[{i)}] $\Pi_1(\clX'_v)\to\Pi_1(\clX')$ is an isomorphism
\vskip2pt
\item[{ii)}] $\Delta_v=\Delta$
\vskip2pt
\item[{iii)}] $\ell\Delta+\Delta_v=\Delta$
\end{itemize}
where iii) $\Rightarrow$ ii) follows by Nakayama's Lemma,
because $\Delta\subseteq\Pi_1(Y)$ is a finite $\Zell$-module.
On the other hand, since $\Delta$ is a finite $\Zell$-module, 
there exist only finitely many subgroups $\Sigma$ such that 
$\ell\Delta\subseteq\Sigma\subseteq\Delta$. Thus for all 
sufficiently large $\ell^e\!$ [precisely, for $\Pi\to\oli\Pi:=\Pi/\ell^e$, 
one must have that $\Delta/\ell\to\oli\Delta/\ell$ is injective], the 
above conditions are equivalent as well to:
\vskip2pt
\begin{itemize}
\item[{iv)}] $\oli\Delta_v=\oli\Delta$.
\vskip2pt
\item[{v)}] $\Pi_1(\clX'_v)/\ell^e\to\Pi_1(\clX')/\ell^e$ is an 
isomorphism.
\end{itemize}
Thus we conclude that it is sufficient to show that 
given $e > 0$, there exists a Zariski open non-empty 
subset $\clU_e\subset\clU$ of valuations $v$ such that 
condition~v) above is satisfied. By contradiction, 
suppose that this is not the case, hence for every
Zariski open subset $\clU'\subset\clU$, there exists
some $v\in\clU'$ such that condition~v) does not
hold at $v$. We notice that $\Pi/\ell^e$ is finite, thus 
it has only finitely many quotients. Therefore, there 
exist a quotient $\Pi/\ell^e\to G$ such that the set
$\Sigma_G:=\{v\in\clU\mid G=\Pi_1(\clX'_v)/\ell^e\to
\Pi_1(\clX')/\ell^e \ \hbox{is not an isomorphism}\}$ is 
in~$\euD$. For $v\in\Sigma$, 
and the etale cover $\tl\clX_v\to\clX'_v$, it follows 
that the base change $\tl\clX_v\times_{\clX'_v}\clY'_v\to \clY'_v$ 
of $\tl\clX_v\to\clX'_v$ to $\clY'_v$ is etale and 
has Galois group $G$. Since the morphisms in
the first row of the diagram~$(\dag)$ are isomorphisms,
there exists a unique etale cover $\tl\clY\to\clY'$ with 
${\rm Gal}(\tl\clY|\clY')=G$ and special fiber 
$\tl\clY_v=\tl\clX_v\times_{\clX'_v}\clY'_v$. 
Further, if $\tl Y\to Y$ is the generic fiber of 
$\tl\clY\to\clY$, recalling the discussion/notations 
from~Definition/Remark~\ref{nonstandard}, one has the
following: Let $\kappa(\clX'_*):=\prod_v\kappa(\clX'_v)/\euD$,
and $\kappa(\tl\clX_*)$, $\kappa(\clY'_*)$ and 
$\kappa(\tl\clY_*)$ be correspondingly defined. 
Then by general principles of ultraproducts of fields  
one has that $\kappa(\tl\clY'_*)=\kappa(\tl\clX'_*)L$, 
$\kappa(\tl\clY_*)=\kappa(\tl\clX'_*)\tl L$, and
$\kappa(\clX'_*)\to\kappa(\tl\clX_*)$, 
$\kappa(\clY'_*)\to\kappa(\tl\clY_*)$ are Galois field 
extensions with Galois group $G$. On the other hand,
by~Definition/Remark~\ref{nonstandard},~c), the 
fields $\oli K$ and $\kappa(\clX'_*)$ are linearly disjoint
over $K$. Therefore, there exists a unique finite 
abelian $\ell$-power degree extension $\tl K|K$ such
that $\kappa(\tl\clX_*)=\kappa(\clX'_*)\tl K$. But then
the liner disjointness of $\oli K$ and $\kappa(\clX'_*)$ over 
$K$ implies that $\tl L=L\tl K$. And since $\tl Y\to Y'$
is etale with Galois group $G$, by~Lemma~\ref{XvsY}, 
it follows that the normalization $\tl X\to X'$ of $X'$ in 
$K\hra\tl K$ is etale (with Galois group $G$). Hence the 
corresponding $\tl\clX\to\clX'$ is an etale cover with
Galois group $G$, etc., contradiction! This concludes
the proof of~Proposition~\ref{lemaunu}.
\end{proof}

\begin{fact}
\label{fact4.5}
We finally notice the following:
\vskip5pt
\begin{itemize}[leftmargin=30pt] 
\item[{a)}] For $v\in\clU$, let $H\subset\lvP^{N}_{kv}$ be 
a general hyperplane. Then $\clX_H:=\clX_v\cap H$ is a 
projective normal $kv$-variety, and a Weil prime divisor of 
$\clX_v$, whose valuation we denote~by~$v_H$. And
$\clS_H:=\clS_v\cap H$ is a closed subset of $\clX_H$,
and we let $\DD{\bsl{\clX_H}{\clS_H}}$ be the corresponding
geometric set of Weil prime divisors of the function 
field $\kappa(\clX_H)|kv$ of $\clX_H$. 
\vskip2pt
\item[{b)}] For a valuation $\bfv$ of $K$, 
we let $U\!_\bfv:=\clO_\bfv^\times\subset K\tms$ be the 
$\bfv$-units, and $\hat\jmath_\bfv:\whU\!_\bfv\to\whKbfv$ 
be the $\ell$-adic completion of the $\bfv$-reduction 
homomorphism $\jmath_\bfv:U\!_\bfv\to\Kbfv\tms$. 
\end{itemize}
\end{fact}
\begin{prop}
\label{propunu}
Let $\,\clU$ be the Zariski open non-empty set 
from~Fact$\,/$Notations~\ref{factnotations}. In the above 
notations, for $v\in\clU$ we set $\bfv:=v_H\circ\vcr K$,
thus $\Kbfv=(K\!\vcr K)v_H$ is a function field over $kv$,
and consider the geometric sets of prime divisors: 
$\DD{\bsl X S}$ of $K|k$, $\DD{\bsl{\clX_v}{\clS_v}}$ 
of the function field $K\!\vcr{K}|\hp kv$, respectively 
$\DD{\bsl{\clX_H}{\clS_H}}$ of the function field 
$\Kbfv\hp|\hp kv$.  Then one has:
\vskip2pt
\begin{itemize}[leftmargin=30pt]
\item[{\rm 1)}] $\whU\!_K\subseteq\whU\!_{\DD{\bsl XS}}
\subseteq\whU\!_{\vcr K}$, and $\hat\jmath_{\vcr K}$ maps 
$\whU\!_K\subseteq\whU\!_{\DD{\bsl XS}}$ isomorphicaly onto 
$\whU\!_{K\hm1\vcr K}\subseteq\whU\!_{\DD{\bsl{\clX_v}{\clS_v}}}$.
\vskip2pt
\item[{\rm 2)}] $\bfv:=v_H\circ\vcr{K}$ is a quasi prime divisor 
of $K|k$ satisfying $\bfv|_k=v$ and $\Kbfv=\kappa(\clX_v)$.
\vskip2pt
\item[{\rm 3)}] $\whU\!_K\subseteq\whU\!_{D_{\bsl X S}}
\subseteq\whU\!_\bfv$, and $\hat\jmath_\bfv$ maps 
$\whU\!_K\subseteq\whU\!_{\DD{\bsl X S}}$ isomorphically onto 
$\whU\!_{\Kbfv}\subseteq\whU\!_{\DD{\bsl{\clX_H}{\clS_H}}}$. 
\end{itemize}
\end{prop}
\begin{proof}  To 1): The $\ell$-adic duals of the isomorphisms
$\Pi_1(X)\ot\Pi_1(\clX_v)$,
$\Pi_1(\bsl XS)\ot\Pi_1(\bsl{\clX_v}{\clS_v})$ are the
isomorphism $\whU_K\to\whU_{K\!\vcr K}$, respectively
$\whU(\bsl XS)\to\whU(\bsl{\clX_v}{\clS_v})$. 
On the other hand, by Kummer theory it follows that
the last two isomorphisms are defined by the 
$\ell$-adic completion of the residue homomorphism
$\jmath_{\vcr K}:U\!_{\vcr K}\to K\!\vcr{K}^{\,\times}$.
This completes the proof of assertion~1). 
\vskip4pt
The proof of assertion~2) is clear, because $w_H$ is trivial
on $kv$, thus the restriction of $\bfv$ to~$k$ equals the one
of $\vcr K$, which is $v$ by the definition of $\vcr K$.
\vskip4pt
Finally, assertion~3) is proved actually in~\nmnm{P}{op}~[P3], 
at the end of the proof of assertion~1) of loc.cit.,~Proposition~23, 
and we omit that argument here.
\end{proof}

We next make  a short preparation for the second
application of the specialization techniques. Let
$\zY\to \zX$ be a finite morphism of projective normal 
$k$-varieties with function fields $K:=k(\zX)$ and 
$M:=k(\zY)$. Consider the \defi{inclusion} 
$\clI:\Div(\zX)\to\Div(\zY)$ and the \defi{norm} 
$\clN:\Div(\zY)\to\Div(\zX)$ maps, which 
map principal divisors to principal divisors 
and $\clN\circ\clI=[M:K]\cdot{\rm id}_{\hp\Div(\zX)}$ 
on $\Div(\zX)$. Thus one has a commutative diagram 
of the form:
\vskip-2pt
\[
\begin{matrix}
1\to\hhb{-6}&\clH(M)&\hra&\Div(\zY)&\shor{}&\euCl(\zY)\,\to0   \cr
 &\uparrow\ \downarrow&&\uparrow\ \downarrow
           &&\uparrow\ \downarrow\hhb{30} \cr 
1\to\hhb{-6}&\clH(K)&\hra&\Div(\zX)&\shor{}&\euCl(\zX)\,\to0    \cr
\end{matrix}
\]
\vskip2pt
\noindent
where the upper arrows are defined by $\clI$, the 
down arrows are defined by $\clN$, and the composition
$\downarrow\circ\uparrow$ is the multiplication by $[M:K]$.
Thus $\clI$ and $\clN\circ\clI$ map elements of infinite order to such, and 
if $[\xx]\in\euCl(X)$ has finite order $o_{[\xx]}$, the other of 
$\clI([\xx])$ divides $o_{[\xx]}$, and the order of $\clN\circ\clI([\xx])$ 
is precisely $o_{[\xx]}/n$, where $n:={\rm g.c.d.}(o_{[\xx]},[M:K])$.

\begin{fact}
\label{fctcinci} 
There exists a Zariski open nonempty set $\clU\subset\Val k$ 
such that each $v\in\clU$ satisfies:
\vskip2pt
\begin{itemize}[leftmargin=30pt]
\vskip2pt
\item[{1)}] The inclusion/norm morphisms 
$\clI:\Div(\zX)\to\Div(\zY)$, $\clN:\Div(\zY)\to\Div(\zX)$ are 
compatible with the specialization maps $\spsp_{\zY,v}$ 
and $\spsp_{\zX,v}$, and with the inclusion/norm  
morphisms $\clI_v:\Div(\zclX_v)\to\Div(\zclY_v)$ and 
$\clN_v:\Div(\zclY_v)\to\Div(\zclX_v)$, i.e., one has: 
\[
\spsp_{\zY,v}\circ\clI=\spsp_{\zX,v}\circ\clI_v,\quad
 \clN_v\circ\spsp_{\zY,v}=\clN\circ\spsp_{\zX,v} 
\]
\item[{2)}] The specialization maps $\spsp_v$ are comptible
with principal divisors, hence define specialization 
morphisms $\spsp_v:\euCl(\zX)\to\euCl(\zclX_v)$, 
$\spsp_v:\euCl(\zY)\to\euCl(\zclY_v)$ which are compatible
with the inclusion/norm morphisms, thus one has
commutative diagrams:
\vskip-2pt
\[
\begin{matrix}
\hhb{25}
&\clH(M) \hhb7\hor{\spsp_v} \hhb7 \clH(M\hm1\vcr M)&&
\Div(\zY) \hhb7\hor{\spsp_v} \hhb7 \Div(\zclY_v)&&
\euCl(\zY) \hhb7\hor{\spsp_v} \hhb7 \euCl(\zclY_v)  \cr
&\upa\ \dwa \hhb{50} \upa\ \dwa\hhb{10} &&
\upa\ \dwa \hhb{60} \upa\ \dwa &&
\upa\ \dwa \hhb{45} \upa\ \dwa \cr 
&\clH(K) \hhb7\hor{\spsp_v} \hhb7 \clH(K\hm1\vcr K)&&
\Div(\zX) \hhb7\hor{\spsp_v} \hhb7 \Div(\zclX_v)&&
\euCl(\zX) \hhb7\hor{\spsp_v} \hhb7 \euCl(\zclX_v)  \cr
\end{matrix}
\]
\vskip2pt
\noindent
in which the vertical maps are the defined by the 
inclusion/norm homomorphisms, and the horizontal 
maps are the corresponding specialization homomorphism 
(as introduced in Fact~\ref{fcttrei}). 
\vskip2pt
\item[{3)}] Recalling that $\euCl^0(\bullet)\subseteq\euCl'(\bullet)$ 
are the maximal divisible, respectively $\ell$-divisible, subgroups
of $\euCl(\bullet)$, it follows that the specialization morphism
$\spsp_v:\euCl(\zX)\to\euCl(\zclX_v)$ satisfies:
$\spsp_v\big(\euCl^0(\zX)\big)\subseteq\euCl^0(\zclX_v)$, 
$\spsp_v\big(\euCl'(\zX)\big)\subseteq\euCl'(\zclX_v)$, 
and correspondingly for $\zY$. 
\vskip2pt
\item[{4)}] Finally, for $[\xx]\in\euCl(\zX)$ and 
$[\yx]:=\clI([\xx])$, one has: 
\vskip2pt
\begin{itemize}[leftmargin=25pt]
\item[{a)}] The order of $\spsp_v[\xx]$ is infinite if and only if
the order of $\spsp_v[\yx]$ is infinite.
\vskip2pt
\item[{b)}] If $\spsp_v[\yx]$ has finite order, then $\spsp_v[\xx]$
has finite order, and their orders satisfy: 
\[
o_{[M:K]\cdot\hp\spsp_v[\xx]}\,|\,o_{\spsp_v[\yx]}\,|\,o_{\spsp_v[\xx]}.
\]
\end{itemize}
\end{itemize}
\end{fact}
We conclude this preparation by mentioning the following 
fundamental fact about specialization of points of abelian 
$k$-varieties, which follows from Appendix,~Theorem~\ref{Thm:Main} 
by~\nmnm{J}{ossen}, which generalizes and earlier result 
by~\nmnm{P}{ink}~[Pk] to arbitrary finitely generated base fields.
\begin{fact}
\label{fctsase}
Let $A$ be an abelian variety over $k$, and $x\in A(k)$. 
Then there exists $\clU_A\subset\Val k$ Zariski open such 
for $v\in\clU_A$, there exists an abelian $\clO_v$-scheme
$\clA$ with generic fiber $A$ and special fiber $\clA_v$, and the 
specialization $x_v\in\clA_v(kv)$ of $x \in A(k)$ satisfying:
\vskip2pt
\begin{itemize}[leftmargin=30pt]
\item[{1)}] If $x$ has finite order $o_x$, then the order 
$o_{x_v}$ of $x_v$ is finite, and $o_x=o_{x_v}$.
\vskip2pt
\item[{2)}] If $o_x$ is infinite, then for every $\ell^e>0$
there exists $\Sigma_x\subset\clU_A$ Zariski dense
satisfying:
\vskip2pt
\begin{itemize}[leftmargin=20pt]
\item[{i)}] $kv$ is an algebraic closure of a finite field with
${\rm char}(kv)\neq\ell$.
\vskip2pt
\item[{ii)}] $o_{x_v}$ is divisible by $\ell^e$.
\end{itemize}
\end{itemize} 
\end{fact}
\begin{proof} Let $R\subset k$ 
be a finitely generated $\lvZ$-algebra, say 
$R=\lvZ[a_1,a_1^{-1},\dots,a_n,a_n^{-1}]$, $a_i\in k\tms$,
over which $A$ and $x\in A(k)$ are defined, i.e., there 
exists an $R$-abelian scheme $\clA_R$ and an $R$-point
$x_R\in \clA_R(R)$, such that $A$ and $x\in A(k)$ are the 
base changes of $\clA_R$ and $x_R\in\clA_R(R)$ under 
the canonical inclusion $R\hra k$. We denote by $\clA_s$
and $x_s\in\clA_s\big(\kappa(s)\big)$ the fibers of $\clA_R$,
respectively $x_R$, at $s\in{\rm Spec}\,R$, and notice
that if the order $o_x$ of $x$ is finite, then $o_x=o_{x_R}$. 
Thus the set of all the points $s\in{\rm Spec}\,R$ such that 
$o_x=o_{x_R}=o_{x_s}$ is a Zariski open subset in
${\rm Spec}\,R$. Thus replacing $R$ by a $R[a^{-1}]$ 
for a properly chosen $a\in R$, $a\neq0$, we can suppose
without loss of generality that $o_x=o_{x_R}=o_{x_s}$
for all $s\in{\rm Spec}\,R$.
\vskip2pt
Let $\clU_A$ be the set of all the valuations $v$ of $k$ 
with $R\subset\clO_v$. Notice that $v\in\clU_A$ if and
only if $v$ has a center $s\in{\rm Spec}\,R$ 
via the inclusion $R\hra k$ if and only if $a_i$ is a 
$v$-unit for $i=1,\dots,n$. In particular, the base change
$\clA_{\clO_v}:=\clA_R\times_R\clO_v$ is an abelian
$\clO_v$-scheme with generic fiber $A$, and the base
change $x_{\clO_v}$ of $x_R$ under $R\hra\clO_v$ is
an $\clO_v$-point of $\clA_{\clO_v}$ whose generic fiber
is $x$, and whose special fiber $x_v\in\clA_v(kv)$ is the
base change of $x_s$ under the canonical embedding
$\kappa(s)\hra kv$. Thus the assertions~1) and~2) 
of~Fact~\ref{fctsase} follow from the discussion above, 
whereas the more difficult assertion~3) follows 
from~Appendix,~Theorem~\ref{Thm:Main} by~\nmnm{J}{ossen}
applied to the abelian $R$-scheme $\clA_R$ over 
${\rm Spec}R$ (recalling that $R$ is a $\lvZ$-algebra 
of finite type). Namely, the set $\Sigma_{x_R}$ of all 
$s\in{\rm Spec}\,R$ with $o_{x_s}$ divisible by $\ell^e$ 
is Zariski dense. Hence the set $\Sigma_x$ of all the $v$ 
which have center in $\Sigma_{x_R}$ is Zariski dense 
in $\Val k$, etc. 
\end{proof}
\vskip5pt
We now define the \defi{quasi arithmetical $\whU$-lattice} 
$\lxzl K^0\subseteq\lxzl K$ mentioned at~Step~3 
in~subsection~A). Recall that a $\Zell$-submodule 
$\Delta\subset\whK$ is said to have \defi{finite co-rank}, 
if there exists some geometric set of prime divisors 
$D$ such that $\Delta\subseteq\whU\hmm_D$. Since the 
family of geometric sets for $K|k$ is closed under
intersection (and union), it follows that the set of all the finite 
corank $\Zell$-submodules $\Delta\subset\whK$ is filtered w.r.t.\ 
inclusion, and that their union $\whKfin$ is given by:
\vskip7pt
\centerline{$\whKfin={\mathop\cup\limits_{\scriptscriptstyle D}}\,
\whU\hmm_D\subset\whK$, \ \ $D$ geometric.}
\vskip4pt
\noindent
Clearly, $\whKfin$ is a birational invariant of $K|k$, and if 
$\lxzl{K}'\subset\whK$ is any divisorial $\whU\!_K$-lattice, then 
$\lxzl{K}'\subset\whKfin$ and $\whKfin=\lxzl{K}'\otimes\,\Zell$.
Further, if $L|l$ is another function field over an algebraically 
closed field $l$, then every embedding $L|l\hra K|k$ 
induces and embedding $\whLfin\hra\whKfin$. 
\vskip3pt
\begin{remarks}
\label{rmknot}
In the context/notations from subsection~B) above
and Proposition~\ref{propunu}, let $\jmtha K\!$,
$\jmath_\Kbfv:\Kbfv^\times\to\whKbfv$ be the 
$\ell$-adic completion homomorphisms. Let 
$\lxzl{K}'\subset\whK$, $\lxzl{\Kbfv}'\subset\whKbfv$
be fixed divisorial latices, and $\lxzl{K}\subset\whK$, 
$\lxzl{\Kbfv}\subset\whKbfv$ be the canonical ones.
Then by mere definitions, there exist unique 
$\epsilon, \epsilon_\bfv\in\Zell\tms/\hhb1\lvZ^\times_\pel$ 
such that $\lxzl{K}'=\epsilon\cdot\lxzl K$ and 
$\lxzl{\Kbfv}'=\epsilon_\bfv\cdot\lxzl\Kbfv$.
\begin{itemize}[leftmargin=30pt]
\item[1)] For every $\eta\in\Zell\tms$ the following conditions 
are equivalent:\footnote{\hhb2Recall that for every abelian 
group $A$ we denote $A_\pel:=A\otimes{\textstyle\lvZ_\pel}$.}
\vskip2pt
\begin{itemize}[leftmargin=20pt]
  \item[{i)$\hhb2$}] $\eta\cdot\jmthim K_\pel\subset\lxzl{K}'$.
  \vskip2pt
  \item[{i)$'$}] $\eta\cdot\jmthim K_\pel\cap\lxzl{K}'$ is non-trivial.
  \vskip2pt
  \item[{ii)$\hhb2$}] $\eta\cdot\lxzl{K}=\lxzl{K}'$ 
  \vskip2pt
  \item[{ii)$'$}] $\eta\cdot\lxzl{K}\cap\lxzl{K}'$ contains
  $\whU\!_K$ strictly.
\end{itemize}
\vskip3pt
\noindent
Actually, $\eta:=\epsilon$ is the unique $\eta\in\Zell\tms/\,\lvZ\tms_\pel$ 
satisfying the above equivalent conditions. 
\vskip2pt
\item[2)] Correspondingly, the same is true about 
$\lxzl{\Kbfv}'$, and $\eta_\bfv:=\epsilon_\bfv$ is the unique 
$\eta_\bfv\in\Zell\tms/\lvZ\tms_\pel$ such that 
$\epsilon_\bfv\cdot\jmthim\Kbfv\subseteq\lxzl{\Kbfv}'$, etc.
\vskip5pt
\item[{3)}] Since
$\jmath_\bfv\big(U_\bfv\cap\jmthim K\big)=\jmthim{\Kbfv}$,
it follows that $\epsilon/\epsilon_\bfv\in\Zell\tms/\lvZ_\pel^\times$ 
is unique such that $\jmath_\bfv(\lxzl K'\cap\whU\!_\bfv)$ and 
$(\epsilon/\epsilon_\bfv)\cdot\lxzl\Kbfv'$ are equal modulo 
$\jmath_\bfv(\whU\!_K)\cdot\whU\!_{\Kbfv}\subset\whKbfvfin$.
\vskip6pt
\item[4)] Conclude that
for every $\eta_\bfv\in\Zell\tms/\,\lvZ\tms_\pel$ 
the following conditions are equivalent:
\vskip2pt
\begin{itemize}[leftmargin=20pt]
\item[{j)}] $\eta_\bfv\cdot\jmthim\Kbfv_\pel
  \subseteq\hat\jmath_\bfv(\lxzl K'\cap\whU\!_\bfv)$.
\vskip2pt
\item[{jj)}] $\eta_\bfv\cdot\jmthim\Kbfv_\pel\cap\,
   \hat\jmath_\bfv(\lxzl K'\cap\whU\!_\bfv)$ is non-trivial.
\end{itemize}
Moreover, $\eta_\bfv:=\epsilon/\epsilon_\bfv$ 
is the unique $\eta_\bfv\in\Zell\tms/\lvZ\tms_\pel$ 
satisfying the conditions~j),~jj).
\vskip4pt
\item[5)] For every finite corank $\lvZ_\ell$-submodule
$\Delta\subset\whKfin$ with $\whU\!_K\subset\Delta$, 
one has: There exists a Zariski open non-empty 
subset $\clU_\Delta\subset\Val k$ such that for every 
$v\in\clU_\Delta$ there exist ``many'' quasi prime
divisors $\bfv$ with $\bfv|_k=v$ and the following 
hold: 
\vskip2pt
\begin{itemize}[leftmargin=25pt]
\item[a)] $\Delta\subset \whU\!_\bfv$, and $\jmath_\bfv$ 
is injective on $\Delta$ and maps $\Delta$ into $\whKbfvfin$.
\vskip2pt
\item[b)] For $\epsilon,\epsilon_\bfv$ from 1), 2)
above one has: $\epsilon\hm1\cdot\hm1(\Delta\cap\lxzl K)\subset\lxzl K'$ iff 
$\epsilon_\bfv\hm1\cdot\hm1\jmath_\bfv(\Delta\cap\lxzl K)\subset\lxzl{\Kbfv}'$.
\end{itemize}
\end{itemize}
\end{remarks}
\begin{definition}
\label{bscdef}
Let $\lxzl K$ be the canonical divisorial $\whU\!_K$-lattice 
for $K|k$. For finite corank $\Zell$-modules 
$\Delta\subset\whKfin$ with $\whU\!_K\subset\Delta$, 
let $\clD\!_\Delta$ be the quasi prime divisors $\bfv$ 
such that $k\bfv$ is algebraic over a finite field, 
$\chr(k\bfv)\neq\ell$, and $\bfv$ satisfies conditions~5a),~5b) 
above.~We~set
\vskip5pt
\centerline{
$\Delta(\bfv):=\whU\!_K\cdot\{\,\xx\in\Delta\cap\lxzl K\mid
     \hat\jmath_\bfv(\xx)\in\jmthim{\Kbfv}_\pel\}$,
\ 
$\Delta^0:=\cap_{\bfv\in\clD\!_\Delta}\Delta(\bfv)$,}  
\vskip5pt
\noindent
and call $\clL^0_K:=\cup_\Delta \Delta^0$ the \defi{quasi 
arithmetical $\whU\!_K$-lattice} of $\clG_{\trK}$. 
\end{definition}
\begin{prop}
\label{specres2}
In the above notations, suppose that $\td K k >2$. The 
following hold:
\vskip2pt
\begin{itemize}[leftmargin=30pt]
\item[{\rm1)}] $\whU\!_K\cdot\jmthim{K}\subseteq\clL^0_K$ 
and $\clL^0_K\,/\,\big(\whU_{\!K}\cdot\jmthim{K}\big)$ 
is a torsion $\lvZ_\pel$-module.
\vskip2pt
\item[{\rm2)}] Let $\Phi\in{\rm Isom}^{\rm c}(\pia K,\pia L)$ 
define an isomorphim $\clG_{\tottrK}\to\clG_{\tottrL}$,
and $\hat\phi:\whL\to\whK$ be its Kummer isomorphism.
Then for all $\epsilon\in\Zell\tms\,$ one has: 
$\hat\phi(\lxzl L)=\epsilon\hm1\cdot\hm1\lxzl K\,$ iff 
$\,\hat\phi(\clL^0_L)=\epsilon\hm1\cdot\hm1\clL^0_K$.
\end{itemize}
\end{prop}
\begin{proof} To 1): First, the inclusion 
$\whU_{\!K}\cdot\jmthim{K}\subseteq\clL^0_K$ is clear, 
because $\hat\jmath_\bfv\big(\jmthim K\big)
\subseteq\jmthim\Kbfv$ for all $\bfv$, and therefore
$\hat\jmath_\bfv\big(\jmthim K_\pel\big)
\subseteq\jmthim\Kbfv_\pel$ as well. The torsion 
assertion is much more involved. Let $D_K$ be a 
complete regular like set of prime divisors for $K|k$, 
and recalling the notations from~Proposition~\ref{propSubsecC}, 
consider the corresponding canonical exact sequence:
\[
1\to\whU_{\!K}\hor{}\lxzl{K}\hhb2\hor{\widehat\jmath^{\hhb2D_K}}
\hhb2\Div(D_K)_\pel\hhb2\hor{{\rm can}\ }
\whCl\hhb1(D_K).
\]

The fact that $\clL^0_K\,/\,\big(\whU_{\!K}\cdot\jmthim{K}\big)$ 
is a torsion $\lvZ_\pel$-module is equivalent to the 
fact that for every $\xx\in\clL^0_K\subset\clL_K$, 
there exists some positive integer $n>0$ such that 
${\widehat\jmath^{\hhb1D_X}}(\xx^n)\in
{\widehat\jmath^{\hhb1D_X}}(K\tms)$, thus principal. 
Let $\xx\in\clL^0_K$ be a fixed element throughout.
\vskip5pt
$\bullet$ \ {\it By contradiction,\/} suppose that 
${\widehat\jmath^{\hhb1D_X}}(\xx^n)$  is not a principal divisor for any $n>0$. 
\vskip5pt
\noindent
Let $D'\subseteq D_K$ be any geometric set of prime 
divisors for $K|k$ such that $\xx\in\whU\!_{D'}$. Then 
$\xx\in\whU\!_D$ for every geometric set $D\subseteq D'$. 
Hence considering a projective normal variety $X_0$ with 
$D_K\subset \DD{X_0}$, we can choose a closed subset
$S_0\subset X_0$ such that $D:=\DD{\bsl{X_0}{S_0}}\subseteq D'$, and 
$\bsl{X_0}{S_0}$ is smooth. Recalling the context 
from~Fact/Notations~\ref{factnotations}, we consider 
$X\to X_0$, and the preimage $S\subset X$
of $S_0$ under $X\to X_0$, one has: $\DD X$ is a 
complete regular like set of prime divisors for $K|k$ 
(because it contains $D_K$ which is so). Further,
since $\bsl{X_0}{S_0}$ is smooth, one has 
$\pia{1,\bsl{X_0}{S_0}}=\Pi_1(\bsl{X_0}{S_0})$,
and by Proposition~\ref{Pivspi}, it follows that all the
canonical projections below are isomorphisms:
\[
\pia{1,D}=\pia{1,\DD{\bsl{X_0}{S_0}}}\to\pia{1,\DD{\bsl X S}}
\to\Pi_1(\bsl X S)\to\Pi_1(\bsl{X_0}{S_0}).
\]
Therefore, the corresponding $\ell$-adic duals are 
isomorphic as well, hence $\whU\!_D=\whU\!_{\bsl X S}$ 
inside~$\whK$. Since $\lxzl K$ is a birational 
invariant of $K|k$, considering the canonical exact
sequence
\[
1\to\whU_{\!K}\hor{}\lxzl{K}\hhb2\hor{\widehat\jmath^{\hhb2D\!_X}}
\hhb2\Div(X)_\pel\hhb2\hor{{\rm can}\ } \whCl\hhb1(X),
\]
one has that $\xx\in\lxzl K^0\subset\lxzl K$ satisfies: 
$\xx\in\whU\!_D=\whU\!_{\bsl XS}$, and therefore 
$\xx\in\lxzl K^0\cap \whU\!_{\bsl XS}$ inside $\whK$.
Further, ${\widehat\jmath^{\hhb1D_X}}(\xx^n)$ is 
not a principal divisor for any positive integer $n>0$. 
 
\vskip2pt
Next recall that $\euCl^0(X)\subset\euCl'(X)$ are the maximal 
divisible, respectively $\ell$-divisible, subgroups of $\euCl(X)$, 
and that $\Div^0(X)\subseteq\Div'(X)$ are their preimages in 
$\Div(X)$, respectively. Since $D_X$ is complete regular 
like, be mere definitions one has that
${\widehat\jmath^{\hhb1D_X}}(\clL_K)=\Div'(X)_\pel$. 
Therefore, ${\widehat\jmath^{\hhb1D_X}}(\xx)\in\Div'(X)_\pel$,
thus its divisor class $[\xx]$ satisfies $[\xx]\in\euCl'(X)$. And
since ${\widehat\jmath^{\hhb1D_X}}(\xx^n)$ is not principal for any $n>0$, it follows
that $[\xx]\in\euCl'(X)$ has infinite order, thus the same is 
true for every multiple $m[\xx]$, $m\neq0$. On the other hand,
by the structure of $\euCl(X)$, see e.g.\ \nmnm{P}{op}~[P3],
Appendix, it follows that $\euCl'(X)/\euCl^0(X)$ is a finite 
prime-to-$\ell$ torsion group. Thus we conclude that some multiple
$m[\xx]$ lies in $\euCl^0(X)$ and has infinite order. 
{\it Mutatis mutandis,\/} we can replace $\xx\in \clL^0_K$ by its 
power $\xx^m\in\clL^0_K$, thus without loss we can suppose 
that $[\xx]\in\euCl^0(X)$ has infinite order.
\vskip2pt
Now let $\clU\subset\Val k$ be the Zariski open subset 
introduced at~Fact/Notations~\ref{factnotations}, and consider 
$\clU^{\scriptscriptstyle \rm max}:=\{v\in\clU\mid kv
          =\oli\lvF_q\ \hbox{for some prime $q\neq\ell$}\}$.
Then by general valuation theoretical non-sense, it follows
that $\clU^{\scriptscriptstyle \rm max}$ is Zariski dense
in $\Val k$.
\vskip3pt
\underbar{Step 1}.
We claim that there exists a Zariski dense subset $\Sigma_\xx$ of 
$\Val k$ with $\Sigma_\xx\subset\clU^{\scriptscriptstyle \rm max}$ 
such that $\spsp_v[\xx]\in\euCl(\clX_v)$ has order
divisible by $\ell$. 
Indeed, in order to do so, recall the finite generically 
normal alteration $ Z\to X$ of $X$ from 
Fact/Notations~\ref{factnotations}. Since 
$Z$ is a projective smooth $k$-variety, the connected 
component ${\rm Pic}^0_Z$ of ${\rm Pic}_Z$ 
is an abelian $k$-variety. Further, the image $[\yx]=\clI([\xx])$ 
of $[\xx]$ under $\clI:\euCl(X)\to\euCl(Z)$ satisfies: First,
$[\yx]$ lies in the divisible part $\euCl^0(Z)$ of $\euCl(Z)$, 
because $[\xx]$ lies in the divisible part $\euCl^0(X)$ 
of $\euCl(X)$, and second, $[\yx]$ has infinite order 
by~Fact~\ref{fctcinci}. On the other hand, the divisible 
part $\euCl^0(Z)$ is nothing but the $k$-rational 
points of the abelian variety ${\rm Pic}^0_Z$. Thus 
by~Fact~\ref{fctsase} above, it follows 
that for every given positive integer $e>0$, there exists 
$\Sigma_\xx\subset\clU^{\scriptscriptstyle\rm max}$
which is dense in $\Val k$ such
that all $v\in\Sigma_\xx$ satisfy: The specialization 
$\spsp_v([\yx])\in\Pic^0_{\clZ_v}(kv)=\euCl^0(\clZ_v)$ 
has order divisible by~$\ell^e$. On the other hand, 
by~Fact~\ref{fctcinci},~4), one has that the order of
$\spsp_v([\xx])$ is divisible by the order of $\spsp_v([\yx])$,
thus by $\ell^e$. We thus conclude that $o_{\spsp_v([\xx])}$ 
is divisible by $\ell$, as claimed.  
\vskip3pt
\underbar{Step 2}. Recalling that $\whU\!_{\DD{\bsl XS}}=\whU\!_D$
and $\xx\in\clL^0_K\cap\whU\!_D$, set $\Delta:=\whU\!_{\DD{\bsl XS}}$ 
and keep in mind the definition of $\clD\!_\Delta$. We claim that 
for every $v\in\clU^{\scriptscriptstyle \rm max}$ there 
exist some $\bfv\in\clD\!_\Delta$ such that $v=\bfv|_K$
(or in other words, $\clU^{\scriptscriptstyle \rm max}$
is contained in the restriction of $\clD\!_\Delta$ to $K$).
Indeed, for every $v\in\clU^{\scriptscriptstyle \rm max}$, 
consider the general hyperplane sections $\clX_H:=H\cap\clX_v$
and $\clS_H:=H\cap \clS_v$, thus $\clS_H\subset\clX_H$ 
is a closed subset. Then $\clX_H$ is a projective normal 
variety over $kv$, and $\clX_H\subset\clX_v$ is a Weil
prime divisor, say with valuation $v_H$. Then setting
$D_H:=\DD{\bsl{\clX_H}{\clS_H}}$, and $\bfv:=v_H\circ\vcr K$,
by Proposition~\ref{propunu},~2),~3), if follows that 
$\bfv\in\clD\!_\Delta$.
\vskip2pt
\vskip3pt
\underbar{Step 3}. Consider $\bfv=v_H\circ\vcr K$ 
with $v\in\clU^{\scriptscriptstyle \rm max}$ and 
$v_H$ as above. Then by Proposition~\ref{propunu}
it follows that $\whU\!_K\subseteq\whU\!_{\DD{\bsl XS}}
\subseteq\whU\!_{\vcr K}$, and $\hat\jmath_{\vcr K}$ maps 
$\whU\!_K\subseteq\whU\!_{\DD{\bsl XS}}$ bijectively onto 
$\whU\!_{K\hm1\vcr K}\subseteq\whU\!_{\DD{\bsl{\clX_v}{\clS_v}}}$,
and $\whU\!_K\subseteq\whU\!_{D_{\bsl X S}}
\subseteq\whU\!_\bfv$, and $\hat\jmath_\bfv$ maps 
$\whU\!_K\subseteq\whU\!_{\DD{\bsl X S}}$ bijectively onto 
$\whU\!_{\Kbfv}\subseteq\whU\!_{\DD{\bsl{\clX_H}{\clS_H}}}$.
And setting $\xx_H:=\hat\jmath_\bfv(\xx)=
   \hat\jmath_{v_H}\big(\hat\jmath_{\vcr K}(\xx)\big)$,
it follows that $\xx_H\in\whU\!_{\Kbfv}\cdot\jmath_\Kbfv(\Kbfv\tms)$
by the fact that $\xx\in\Delta^0\subseteq\Delta(\bfv)$.
Thus in particular, $\xx_H$ has a trivial divisor
class $[\xx_H]=0$.   
\vskip2pt
On the other hand, if $v\in\Sigma_\xx$, then the divisor 
class of $[\xx_v]$ in $\euCl(\clX_v)$ has finite order 
$o_{v,\xx}$ divisible by $\ell$. Let $\dvdv_v(\xx_v)=\sum_in_iP_i$
with distinct Weil prime divisors $P_i$ of $\clX_v$. Since 
$\clX_H:=H\cap\clX_v$ is a general hyperplane section, 
it follows that $Q_i:=H\cap P_i$ are distinct prime Weil 
divisors of $\clX_H$, and the divisor class
of $[\xx_H]:=[\hhb{.2}\sum_i n_i Q_i]$ in $\euCl(\clX_H)$
has order equal to the order of $[\xx]$. 
This contradicts the fact proved above that $[\xx_H]=0$.
\vskip2pt
To 2): Since $\td K k >2$, it follows that for all
quasi prime divisors $\bfv$ of $K|k$, one has 
$\td \Kbfv{k\bfv}>1$. Let $\clV_\bfv$ be the
generalized quasi prime $(d\!-\!1)$-divisors 
$\tlbfv$ with $\bfv<\tlbfv$. One has:
\vskip2pt
a) We claim that $k\bfv$ is an algebraic closure of a 
finite field iff $\,k\tlbfv$ is an algebraic closure of 
a finite field for all $\,\tlbfv\in\clV_\bfv$. Indeed, if 
$k\bfv$ is not an algebraic closure of a finite field, 
then one has: First, $\Kbfv|k\bfv$ is a function field 
with $\td{\Kbfv}{k\bfv}=\td Kk-1$.
Second, choosing any prime $(d\!-\!2)$-divisor 
$\tlbfw$ of the function field $\Kbfv|k\bfv$, it 
follows that the valuation theoretical composition
$\tlbfv:=\tlbfw\circ\bfv$ is a quasi 
prime $(d\!-\!1)$-divisor of $K|k$ with 
$\Ktlbfv=K\tlbfw$, thus $k\bfv=k\bfw$. Hence
$k\bfv$ is an algebraic closure of a finite field
iff $\,k\tlbfv$ is so, as claimed. 
\vskip2pt
b) Let $\clT^1(K)\subset\pia K$ be the topological 
closure of the minimized quasi divisorial inertia 
$\cup_{\bfv\in\clQ_{K|k}}\mT_\bfv\subset\pia K$. 
For every given $\tlbfv\in\clV_\bfv$, let 
$\clT^1(\Ktlbfv)\subset\Pi^1_{K\tlbfvm}$ be the image of 
$\clT^1\cap\mZ_\tlbfvm$ under the canonical projection 
$\mZ_\tlbfvm\to\Pi^1_\tlbfvm=\mZ_\tlbfvm/\mT_\tlbfvm$. 
Then [P5], Theorem~4.2,~a), gives a group theoretical 
recipe to decide whether $k\tlbfv$ is an algebraic 
closure of a finite field. In particular, combining this 
with the discussion at point~a) above, it follows 
that there are group theoretical recipes to characterize 
the quasi prime divisors $\bfv$ having $k\bfv$ an 
algebraic closure of a finite field in terms of the set 
$\clT^1(K)$. Moreover, if $k\bfv$ is an algebraic 
closure of a finite field, the quasi prime divisors and 
the prime divisors of $\Kbfv|k\bfv$ coincide (because 
$k\bfv$ being algebraic over a finite field, has non 
non-trivial valuations). Thus we conclude that the
minimized $(r\!-\!1)$-divisorial groups in $\Pi^1_{\Kbfv}$ 
are precisely the images $\mT_\tlbfvm/\mT_\bfv
\subset\mZ_\tlbfvm/\mT_\bfv$ of the minimized 
quasi $r$-divisorial groups $\mT_{\tlbfvm}
\subset\mZ_{\tlbfvm}$ of the quasi prime $r$-divisors 
$\tlbfv>\bfv$. On the other hand, by~[P4],~Propostion~3.5
and~\nmnm{T}{opaz},~[To1],~Theorem~8, there are 
group theoretical recipes which recover the groups 
$\mT_{\tlbfvm}\subset\mZ_{\tlbfvm}\subset\pia K$ 
from $\pic K\to\pia K$, provided $r<d:=\td Kk$. To 
recover $\mT_{\tlbfvm}\subset\mZ_{\tlbfvm}$ for 
$\tlbfv>\bfv$ for quasi prime $d$-divisors $\tlbfv$,
one uses~[P5],~Theorem~1.2, as follows: First, by
mere definitions, for any given quasi prime $r$-divisor
$\tlbfv>\bfv$, there exists a exists chain of generalized 
quasi prime divisors $\bfv=:\tlbfv_1<\dots<\tlbfv_r:=\tlbfv$. 
And since $k\bfv$ is an algebraic closure of a finite
field, all the $\tlbfv_i/\bfv$, $1\leqslant i\leqslant r$, 
are trivial on $k\bfv$, hence $\tlbfv_i/\bfv$ are 
generalized {\it prime divisors\/} of $\Kbfv|k\bfv$. 
In particular, if $r=d$, after setting $w:=\tlbfv_{d-1}$, 
in the terminology from~[P5],~Theorem~1.2, one 
has: The quasi prime $d$-divisors $\tlbfv>w$ above 
are actually c.r. $w$-divisors of $K|k$. Hence 
by~[P5],~Theorem~1.2, one can recover the 
minimized $d$-divisorial groups
$\mT_{\tlbfvm}\subset\mZ_{\tlbfvm}\subset\pia K$
from $\pia K$ endowed with $\clT^1(K)$. 
\vskip2pt
We thus finally conclude that for all generalized
quasi prime divisors $\tlbfv>\bfv$, one can 
recover the generalized quasi divisorial groups 
$\mT_{\tlbfvm}\subset\mZ_{\tlbfvm}\subset\pia K$. 
Hence via the canonical projection 
$\mZ_\bfv\to\Pi^1_\Kbfv=\mZ_\bfv/\mT_\bfv$ 
one can recover all the generalized quasi 
divisorial groups $\mT_{\tlbfvm}/\mT_\bfv
\subset\mZ_{\tlbfvm}/\mT_\bfv\subset\Pi^1_\Kbfv$
for all $\tlbfv>\bfv$.  
\vskip5pt
Thus finally, using Proposition~\ref{rescharneqell} 
as well, for every quasi prime divisor $\bfv$, given 
$\pic K$ endowed with $\mT_\bfv\subset\mZ_\bfv$, 
the following can be recovered by group theoretical recipes:
\vskip5pt
{\it 
\itm{45}{
\item[{\rm1)}] The total quasi decomposition graph
$\clG^1_{\clQ^{\rm tot}_\Kbfv}$
\vskip2pt
\item[{\rm2)}] The fact that $k\bfv$ is an algebraic 
closure of a finite field with $\chr(k\bfv)\neq\ell$.
}
}
\vskip2pt
c) Moreover, the group theoretical recipes to check 
whether $k\bfv$ is an algebraic closure of a finite 
field with $\chr\neq\ell$, and if so, to reconstruct 
$\clG_{\clQ^{\rm tot}_\Kbfv}$, are invariant under 
all isomorphisms $\Phi\in{\rm Isom}^{\rm c}(\pia K,\pia L)$. 
Indeed, every such isomorphism $\Phi$ maps the minimized 
quasi $r$-divisorial groups $\mT_\tlbfvm\subset\mZ_\tlbfvm$ 
for $K|k$ isomorphically onto the minimized quasi 
$r$-divisorial groups $\mT_\tlbfwm\subset\mZ_\tlbfwm$ 
for $L|l$, provided $r<\td Kk$. Hence $\Phi$ maps 
$\clT^1(K)$ homeomorphically onto $\clT^1(L)$. Thus
if $\bfv$ and $\bfw$ correspond to each other under 
$\Phi$, i.e., $\Phi(\mT_\bfv)=\mT_\bfw$ and
$\Phi(\mZ_\bfv)=\mZ_\bfw$, then $\Phi$
gives rise to a residual isomorphism
\[
\Phi_\bfv:\Pi^1_\Kbfv=\mZ_\bfv/\mT_\bfv
  \to\mZ_\bfw/\mT_\bfw=\Pi^1_\Lbfw\,.
\] 
Further, $\Phi_\bfv$ maps $\clT^1(\Kbfv)$ 
homeomorphically onto $\clT^1(\Lbfw)$, hence by~b)
above, one has:
\vskip5pt
\noindent
{\it 
{\bf Conclusion.} The \defi{local isomorphisms} 
$\Phi_\bfv: \Pi^1_\Kbfv\to\Pi^1_\Lbfw$ satisfy:
\vskip3pt
\itm{45}{
\item[{\rm1)}] $\Phi_\bfv$ gives rise to an isomorphism 
$\Phi_\bfv:\clG^1_{\clQ^{\rm tot}_\Kbfv}\to
\clG^1_{\clQ^{\rm tot}_\Lbfw}$.
\vskip2pt
\item[{\rm2)}] $k\bfv$ is an algebraic closure of 
a finite field with $\chr\neq\ell$ \ iff \ $l\bfw$ is so.
}
}
\vskip5pt
d) Now suppose that $k\bfv$ is an algebraic closure 
of $\lvF_p$ with $p\neq\ell$, hence $l\bfw$ is an
algebraic closure of $\lvF_p$ as well. Then by Theorem~2.2 
and Theorem~2.1 of \nmnm{P}{op}~[P4], it follows 
that there exists $\epsilon_\bfv\in\Zell\tms$ such that 
$\Phi_{\epsilon_\bfv}:=\epsilon_\bfv\cdot\hhb1\Phi_\bfv$ 
is defined by an isomorphism of the pure inseparable 
closures $\imath_\bfv:\Lbfw^{\rm i}\to\Kbfv^{\rm i}$ 
of $\Lbfw$ and $\Kbfv$ respectively. Moreover, 
$\epsilon_\bfv$ is unique up to multiplication by
powers $p_\bfv^m$, $m\in\lvZ$, of $p_\bfv:=\chr(k\bfv)$. 
In particular, the Kummer isomorphism 
$\hat\phi_{\epsilon_\bfv}:\whgngn\Lbfw\to\whgngn\Kbfv$ 
of $\Phi_{\epsilon_\bfv}$ is nothing but the 
$\ell$-adic completion $\hat\phi_{\epsilon_\bfv}=
\hat\imath_\bfv$ of $\imath_\bfv$.
\vskip5pt
$(*)$ Finally, putting everything together, we conclude 
the following: Let $\bfv$ be a quasi prime divisor 
of $K|k$ and $\bfw$ a quasi prime divisor of $L|l$ 
which correspond to each other under $\Phi$. The 
by the discussion at~a),~b),~c),~e), above, it follows 
that $k\bfv$ is an algebraic closure of a finite field  
of characteristic $\neq\ell$ if and only if $l\bfw$ is
an algebraic closure of a finite field of characteristic
$\neq\ell$. If so, then there exists and isomorphism
$\phi_{\epsilon_\bfv}:\Kbfv^{\rm i}|k\bfv\to\Lbfw^{\rm i}|l\bfw$
defining $\Phi_{\epsilon_\bfv}:=\epsilon_\bfv\cdot\Phi_\bfv$ 
for some $\epsilon_\bfv\in\Zell\tms$, which is unique up to
powers $p^m$, $m\in\lvZ$, of $p_\bfv:=\chr(k\bfv)$.   
\vskip5pt
Next recall that since $\Phi$ maps $\clG_{\tottrK}$ 
isomorphically onto $\clG_{\tottrL}$, it follows 
by~subsection~B),~Proposition~\ref{propSubsecC},~2),~b) 
that replacing $\Phi$ by some multiple 
$\epsilon^{-1}\cdot\Phi$ with $\epsilon\in\Zell\tms$
properly chosen, {\it mutatis mutandis,\/} we can  
suppose that the Kummer isomorphism 
$\hat\phi:\whL\to\whK$ of $\Phi$ maps $\clL_L$ 
isomorphically onto $\clL_K$. In particular, since $\clL_L$ 
is the unique divisorial lattice for $L|l$ which intersects
$\jmthim L_\pel$ non-trivially, it follows that $\clL_K$ 
is the unique divisorial lattice for $K|k$ which intersects 
$\hat\phi\big(\jmthim L_\pel\big)$ non-trivially. 
\vskip2pt
We claim that for $\bfv$ and $\bfw$ being as at
the conclusion~$(*)$ above, one actually has 
$\epsilon_\bfv\in\lvZ_\pel$. Indeed, 
by~Remarks~\ref{rmknot} above we have: $\clL_L$ 
is the unique divisorial lattice for $L|l$ such that 
$\jmath_\bfw(\clL_L)$ intersects $\jmthim\Lbfw$ 
non-trivially, and $\clL_K$ is the unique divisorial lattice 
for $K|k$ such that $\jmath_\bfv(\clL_K)$ intersects 
$\jmthim\Kbfv_\pel$ non-trivially. On the other hand, since 
$\hat\phi_\bfv\circ\hat\jmath_\bfw=\hat\jmath_\bfv\circ\hat\phi$, 
and $\hat\phi(\clL_L)=\clL_K$, and 
$\hat\phi_\bfv\big(\jmthim\Lbfw_\pel\big)=
\epsilon_\bfv\cdot\jmthim\Kbfv_\pel$, one has:
\vskip2pt
\begin{itemize}[leftmargin=30pt]
\item[{i)}] $\jmthim\Kbfv_\pel\subseteq\hat\jmath_\bfv(\clL_K)$
\vskip2pt
\item[{ii)}] $\epsilon_\bfv\cdot\jmthim\Kbfv_\pel=
\hat\phi_\bfv\big(\jmthim\Lbfw\big)\subseteq
\hat\phi_\bfv\big(\hat\jmath_\bfw(\clL_L)\big)=
\hat\jmath_\bfv(\clL_K)$
\end{itemize}
Hence by Remarks~\ref{rmknot},~4),~j),~jj), 
above, it follows that $\epsilon_\bfv\in\lvZ_\pel$, 
as claimed. 
\vskip5pt
Finally, let $\Xi\subset\whL$ be some $\Zell$-submodule,
and $\Delta=\hat\phi(\Xi)$ be its image under $\hat\phi$.
Since $\hat\phi(\whU\!_L)=\whU\!_K$ and $\hat\phi$
maps the finite corank $\Zell$-submodules of $\whL$
isomorphically onto the ones of $\whK$, one has: $\Xi$
has finite corank and $\whU\!_L\subset\Xi$ iff $\,\Delta$
has finite corank and $\whU\!_K\subset\Delta$. Now  
suppose that $\Delta=\hat\phi(\Xi)$ 
satisfy these conditions. Then in the notations from 
Definition~\ref{bscdef}, it follows that if $\bfv$ and 
$\bfw$ are quasi-prime divisors of $K|k$, respectively 
$L|l$ which correspond to each other under $\Phi$, 
and satisfy the conclusion~$(*)$ above, then 
$\bfv\in\clD\!_\Delta$ iff $\bfw\in\clD_\Xi$. Hence 
$\hat\phi$ maps $\Xi(\bfw)$ isomorphically onto
$\Delta(\bfv)$, thus $\hat\phi$ maps
$\Xi^0:=\cap_{\bfw\in\clD_\Xi}\Xi(\bfw)$ isomorphically
onto $\Delta^0:=\cap_{\bfv\in\clD\!_\Delta}\Delta(\bfv)$. 
Thus finally, $\hat\phi$ maps $\clL^0_L:=\cup_\Xi\, \Xi^0$ 
isomorphically onto $\clL^0_K:=\cup_\Delta \Delta^0$.
\end{proof}
%
%
%
%
\vskip2pt
\noindent
E) \ {\it Recovering the rational quotients\/}
\vskip5pt
\noindent
We recall briefly the notion of an \defi{abstract 
quotient} of the total decomposition graph, see 
\nmnm{P}{op}~[P3], Sections~4,~5 for details (where 
quotients of decomposition graphs were discussed).
\vskip2pt
First, let $\clG_\alpha$ be the pair consisting of a
pro-$\ell$ abelian group $G_\alpha$ endowed 
with a system of pro-cyclic subgroups 
$(T_{v_\alpha})^{\phantom{'}}_{v_\alpha}$ 
which makes $G_\alpha$ \defi{curve like} of 
genus $g=0$, i.e., there exists a system of generators 
$(\tau_{v_\alpha})^{\phantom{'}}_\alpha$ of 
the $(T_{v_\alpha})^{\phantom{'}}_\alpha$ such that the 
following are satisfied:
\vskip2pt
\begin{itemize}[leftmargin=30pt]
\item[i)] $\prod_\alpha\,\tau_{v_\alpha}=1$, and this 
is the only pro-relation satisfied by 
$(\tau_{v_\alpha})^{\phantom{'}}_\alpha$.
\vskip2pt
\item[ii)] $G_\alpha$ is topologically generated by
$(\tau_{v_\alpha})^{\phantom{'}}_\alpha$.
\end{itemize}
Notice that if  $\euT':=(\tau'_{v_\alpha})^{\phantom{'}}_{v_\alpha}$ 
is a further system of generators of $(T_{v_\alpha})_{v_\alpha}$ 
satisfying~i),~ii) above, then there exists a unique 
$\ell$-adic unit $\epsilon\in\Zell^\times$ such that 
$\euT'=\euT^\epsilon$. Let
$\whclL_{\clG_\alpha}:={\rm Hom}_{\rm cont}(G_\alpha,\Zell)$
the $\ell$-adic dual of $G_\alpha$, and 
$\lxzl{\clG_\alpha}\subset\whclL_{\clG_\alpha}$ be the
$\Zetell$-submodule of all $\varphi:G_\alpha\to\Zell$ 
satisfying $\varphi(\tau_{v_\alpha})\in\Zetell$ for all 
$v_\alpha$. Notice that if $\euT'=\euT^\epsilon$ as above, 
then the corresponding lattice is $\lxzl{\clG_\alpha}'=
\epsilon^{-1}\cdot\lxzl{\clG_\alpha}$. We call $\lxzl{\clG_\alpha}$
a \defi{divisorial lattice} for $\clG_\alpha$.
\begin{definition/remark}
A \defi{divisorial morphism} 
$\Phi_\alpha:\clG_{\tottrK}\to\clG_\alpha$
is an open group homomorphism 
$\Phi_\alpha:\pia K\to G_\alpha$ which, inductively 
on $d:=\td K k $, satisfies the following: 
\vskip2pt
\begin{itemize}[leftmargin=25pt]
\item[a)] Given the canonical system of inertia generators 
$(\tau_v)_v$ for the inertia groups $T_v\subset\pia K$ of
the prime divisors $v$ of $K|k$, there exists $\epsilon\in\Zell\tms$
and $\epsilon_v\in\epsilon\cdot\lvZ_\pel$ satisfying: 
\vskip2pt
\begin{itemize}[leftmargin=25pt]
\item[i)] For every $v$ there exists $v_\alpha$ such that 
$\Phi_\alpha(\tau_v)=\tau_{v_\alpha}^{\epsilon_v}$, and
$\Phi_\alpha(T_v)=\Phi_\alpha(Z_v)$ if $\epsilon_v\neq0$. 
\vskip2pt
\item[ii)] If $D$ is a complete regular like and 
sufficiently large set of prime divisors of $K|k$, then 
$D_{v_\alpha}:=\{\,v\in D\mid \Phi_\alpha(\tau_v)
=\tau_{v_\alpha}^{\epsilon_v},\, \epsilon_v\neq0\,\}$ is
finite and non-empty for every $v_\alpha$.
\end{itemize}
\vskip2pt 
\item[b)] If $\Phi_\alpha(T_v)=1$, then the induced
homomorphism $\Phi_{\alpha,v}:\pia{\Kv}\to G_\alpha$
has open image and defines a divisorial morphism 
$\Phi_{\alpha,v}:\clG_{\clD^{\rm tot}_{\!\Kv}}\to\clG_\alpha$.
\vskip2pt
\end{itemize}
Notice that taking $\ell$-adic duals, $\Phi_\alpha$ 
gives rise  to a \defi{Kummer homomorphism}
$\hat\phi_\alpha:\whclL_{\clG_\alpha}\to\whK$,
and by mere definitions, one has $\hat\phi_\alpha(\lxzl{\clG_\alpha})
\subset\epsilon\cdot\lxzl K$ and $\hat\phi(\lxzl{\clG_\alpha})
\cap\whU\!_K=1$.
\end{definition/remark}
Recall that for every valuation $v$ of $K$, we 
denote by $\jmath^v:\whK\to\Zell$ and
$\jmath_v:\whU\!_v\to\widehat{\Kv}$ the $\ell$-adic 
completions of $v:K\tms\to\lvZ$, respectively of 
the residue homomorphism $U\!_v\to\Kv\tms$.
\begin{definition}
\label{ratquot}
A divisorial morphism $\Phi_\alpha:\clG_{\tottrK}\to\clG_\alpha$ 
is called an \defi{abstract rational quotient} of $\clG_{\tottrK}$, 
if $\Phi_\alpha$ is surjective, and setting 
$\whlalp:=\hat\phi_\alpha(\widehat\clL_{\clG_\alpha})\subset\whK$,
the following hold:
\vskip3pt
\begin{itemize}[leftmargin=25pt]
\item[{a)}] Every prime divisor $v$ of $K|k$ satisfies:
If $\jmath_v(\whlalp\cap\hhb1\whU\!_v)\neq1$, then 
$\whlalp\subset\whU\!_v$, $\whlalp\cap\,\ker(\jmath_v)=1$.
\vskip2pt
\item[{b)}] All sufficiently large complete regular like 
sets $D$ of prime divisors of $K|k$ satisfy: Let
$\Delta\subset\whK$ be a $\Zell$-module of finite 
co-rank containing $\whU\!_K$. Then there exist 
$v\in D$ such that $\jmath^v(\whlalp)\neq0$ and
further, $\Delta\subset \whU\!_v$ and 
$\Delta\cap{\rm ker}(\jmath_v)=\Delta\cap\whlalp$.
\end{itemize}
\end{definition}
We recall that by \nmnm{P}{op}~[P3], Proposition~40,
one has: If $x\in K$ is a generic element, i.e., $\kpx=k(x)$
is relatively algebraically closed in $K$, then the canonical 
projection of Galois groups $\Phi_{\kpx}:\pia K\to\pia{\kpx}$ 
defines an abstract rational quotient $\clG_{\tottrK}\to\clG_{\kpx}$
in the sense above. We also notice that the canonical 
divisorial lattice for $\clG_{\kpx}$ is nothing but 
$\lxzl{\kpx}:=\jmath_{\kpx}(\kappa_x^\times)\subset\widehat{\!\kappa}_x$,
and the Kummer homomorphism $\hat\phi_{\kpx}:\widehat{\!\kappa}_x\to\whK$
maps $\lxzl{\kpx}$ isomorphically onto $\jmath_K(\kappa_x^\times)$.
\vskip3pt
Using Proposition~\ref{specres2} above, we are now in the 
position to recover the geometric rational quotients of 
$\clG_{\tottrK}$ in the case $\td K k >2$ as follows:
\begin{prop}
\label{charratquot}
Let $k$ is an arbitrary algebraically closed field,
and $\td K k >2$. Then in the notations from above
and of Proposition~\ref{specres2} the following hold:
\vskip3pt
\begin{itemize}[leftmargin=30pt]
\item[{\rm1)}] Let $\Phi_\alpha:\clG_{\tottrK}\to\clG_\alpha$ 
be an abstract rational quotient. Then the following 
are equivalent:
\begin{itemize}[leftmargin=20pt]
\item[{\rm i)}] $\Phi_\alpha$ is geometric, i.e., there exists 
a generic element $x\in K$ and an isomorphism of 
decomposition graphs 
$\Phi_{\alpha,\kpx}:\clG_\alpha\to\clG_\kpx$ such that
$\Phi_\kpx=\Phi_{\alpha,\kpx}\circ\Phi_\alpha$.
\vskip2pt
\item[{\rm ii)}] There exists $\epsilon\in\Zell^\times$ such that 
$\hat\phi_\alpha(\lxzl{\clG_\alpha})\subset\epsilon\cdot\clL^0_K$.
\end{itemize}
\vskip2pt
\item[{\rm2)}] Let $L|l$ be a function field over an algebraically
closed field $l$, and $\Phi:\clG_{\tottrK}\to\clG_{\tottrL}$ be
an abstract isomorphism of decomposition graphs. 
Then $\Phi$ is compatible with geometric rational quotients 
in the sense that if $\Phi_{\kpy}:\clG_{\tottrL}\to\clG_\kpy$ 
is a geometric rational quotient of $\clG_{\tottrL}$, then 
$\Phi_\alpha:=\Phi_{\kpy}\circ\Phi$ is a geometric rational 
quotient of $\clG_{\tottrK}$.
\end{itemize}
\end{prop}

\begin{proof} To 1):  The implication i)~$\Rightarrow$~ii)
is clear by the characterization of the geometric rational
quotients and Proposition~\ref{specres2},~2), precisely
the assertion that $\jmath_K(K\tms)\subset\lxzl K^0$. For 
the implication ii)~$\Rightarrow$~i) 
we proceed as follows: Let $\lxzl{\clG_\alpha}$ be a divisorial
lattice for $\clG_\alpha$ satisfying 
$\lxzl\alpha:=\hat\phi_\alpha(\lxzl{\clG_\alpha})
                                       \subset\epsilon\cdot\clL^0_K$.
Then replacing $\lxzl{\clG_\alpha}$ by its $\epsilon^{-1}$
multiple, without loss we can and \underbar{will} suppose
that $\lxzl\alpha=\hat\phi_\alpha(\lxzl{\clG_\alpha})\subset\clL^0_K$, 
thus $\lxzl\alpha\subset\clL^0_K\subseteq\lxzl K$. 
Recall that by \nmnm{P}{op}~[P3], Fact~32,~(1), $\lxzl\alpha$ is 
$\lvZ_\pel$-saturated in $\lxzl K$, i.e., $\lxzl K/\lxzl\alpha$ 
is $\lvZ_\pel$-torsion free. Further, if $u\in\bsl K k$, and 
$\kpu\subset K$ is the corresponding 
relatively algebraically closed subfield, then 
$\jmath_K(\kpu\tms)_\pel$ is $\lvZ_\pel$-saturated
in $\lxzl K$ as well, i.e., $\lxzl K/\jmath_K(\kpu\tms)_\pel$ is 
$\lvZ_\pel$-torsion free. Finally, by~Proposition~\ref{specres2},~2),
for every $\uux\in\clL^0_K$, there exists a multiple $\uux^n$
such that $\uux^n\in\whU\!_K\cdot\jmath_K(K\tms)$.
\vskip2pt
Let $\uux\in\lxzl\alpha\subset\lxzl K^0$ be a non-trivial
element, and $n_\uux>0$ be the minimal positive integer 
with $\uux^{n_\uux}\in\whU\!_K\cdot\jmath_K(K\tms)$, 
say $\uux^{n_\uux}=\theta\cdot\jmath_K(u)$ for some 
$u\in K\tms$ and $\theta\in\whU\!_{\trK}$, and notice
that $\jmath_K(u)$ and $\theta$ with the property 
above are unique, because $\whU\!_\clG\cap\lxzl\alpha$ 
is trivial, by~\nmnm{P}{op}~[P3], Fact~32.
Further, $n_\uux$ is relatively prime to $\ell$, because
$\lxzl K/\jmath_K(\kpu\tms)_\pel$ has no non-trivial
$\lvZ_\pel$-torsion. As in the proof of Proposition~5.3 from
\nmnm{P}{op}~[P4], one concludes that actually 
$\jmath_K(\kpu\tms)_\pel=\lxzl\alpha$.
Note that in the proofs of the Claims~1,~2,~3, of loc.cit,\ and
the conclusion of the proof of Proposition~5.3 of loc.cit.,  it 
was nowhere used that $k$ is an algebraic closure of a 
finite field.
\vskip2pt
This concludes the proof of~assertion~1).
\vskip3pt
To~2): Let $\Phi_{\kpy}:\clG_{\tottrL}\to\clG_\kpy$ be a 
geometric rational quotient of $\clG_{\tottrL}$. Then as in the 
proof of assertion~2) of Proposition~5.3 of \nmnm{P}{op}~[P4],
it follows that 
\[
\Phi_\alpha=\Phi_{\kpy}\circ\Phi:\clG_{\tottrK}
\to\clG_{\kpy}=:\clG_\alpha
\] 
is an abstract rational quotient of $\clG_{\tottrK}$ with
Kummer homomorphism $\hat\phi_\alpha=
\hat\phi\circ\hat\phi_{\kpy}:\widehat{\!\kappa}_y\to\whK$.
We show that $\Phi_\alpha$ satisfies hypothesis~ii) 
from Proposition~\ref{charratquot},~1). 
Indeed, $\clG_\alpha:=\clG_{\kpy}$ has
$\lxzl{\clG_\alpha}:=\lxzl{\kpy}=\jmath_{\kpy}(\kappa_y^\times)$ 
as canonical divisorial lattice, and since 
$\Phi_{\kpy}:\clG_{\tottrL}\to\clG_\kpy$ is defined by 
the $k$-embedding $\kpy\hra L$, on has by definitions that
$\hat\phi_{\kpy}(\lxzl{\kpy})=\jmath_L(\kappa_y^\times)_\pel
\subseteq\jmath_L(L\tms)_\pel$. Thus finally, 
$\hat\phi_{\kpy}(\lxzl{\kpy})\subseteq\jmath_L(L\tms)\subseteq\lxzl L^0$
by Propostion~\ref{specres2},~2) applied to $L|l$.
On the other hand, by~Proposition~\ref{specres2},~1),
one has that $\hat\phi(\lxzl L^0)=\epsilon\cdot\lxzl K^0$
for some $\epsilon\in\Zell\tms$, and therefore:
\[
\hat\phi_\alpha(\clL_{\clG_\alpha})=
\hat\phi\big(\hat\phi_{\kpy}(\clL_{\kpy})\big)\subset
\hat\phi(\clL^0_L)=\epsilon\cdot\clL^0_K.
\]
Hence by applying assertion~1) we have: $\Phi_\alpha$
is a geometric quotient of $\clG_{\tottrK}$, etc.
\end{proof}
\vskip2pt
\noindent
F) {\it Concluding the proof of Theorem~\ref{mainthm}\/}
\vskip5pt
First, by Proposition~\ref{charratquot} above if follows
that in the context of Theorem~\ref{mainthm} one can
recover the geometric rational quotients of $\clG_{\tottrK}$
from $\pic K$ endowed with $\Inrdiv(K)$, and
that the recipe to do so is invariant under isomorphisms
$\Phi\in{\rm Isom}^{\rm c}(\pia K,\pia L)$ which map \
$\Inrdiv(K)$ onto $\Inrdiv(L)$. Conclude by applying 
Main Result from [P3], Introduction.
\newpage
\section{Proof of Theorem~\ref{appl1}}
\noindent
A) \ {\it Recovering the nature of $k$\/}
\vskip2pt
In  this sub-section we prove the first assertion of 
Theorem~\ref{appl1}. Precisely, for each non-negative 
integer $\delta\geq0$, we will give inductively on $\delta$ a 
group theoretical recipe $\eudim(\delta)$ in terms of 
$\clG_{\qtottrK}$, thus in terms of $\pic K$, such that 
$\eudim(\delta)$ is true if and only if 
\vskip5pt
$(*)$ \ $\td K k >\dim(k)$ \ and \ $\dim(k)=\delta$.
\vskip5pt
\noindent
The simplify language, recovering the above information
about $k$ will be called ``recovering the nature of $k$.'' 
Note that ${\rm char}(k)$ is not part of ``nature of $k$'' 
in the above sense. Moreover, $\eudim(\delta)$ will be 
invariant under isomorphisms as follows: If $L|l$ is a 
further function field over an algebraically closed field 
$l$, and $\clG_{\qtottrK}\to\clG_{\qtottrL}$ is an 
isomorphism, then $\eudim(\delta)$ holds for $K|k$ if 
and only if $\eudim(\delta)$ holds for $L|l$.
\vskip3pt
Before giving the recipes $\eudim(\delta)$, let us recall 
the following few facts about generalized quasi prime 
divisors, in particular, the facts from \nmnm{P}{op}~[P4],~[P5], 
Section~4, and~\nmnm{T}{opaz}~[To2]. 
\vskip3pt
First, letting 
$\clQ(K|k)$ be the set of all the quasi prime divisors 
of $K|k$, let $\clT^{\hpp1}(K)\subset\pia K$ the 
topological closure of $\cup_{\bfv\in\clQ(K|k)}\,\mT_\bfv$
in $\pia K$. Further, for $l\subseteq k$ an algebraically
closed subfield, let $\clQ_l(K|k)$ be the set of all 
the quasi prime divisors $\bfv$ with $\bfv|_l=\bfw|_l$, 
and $\clT^{\hpp1}_l(K)\subset\pia K$ be the topological 
closure of $\cup_{\bfv\in\clQ_l(K|k)}\,\mT_\bfv$.
Then by~[P5],~Theorem~A, one has: 
$\clT^{\hpp}_l(K)\subseteq\clT^{\hpp}(K)$ consist of 
minimized inertia elements, and $\clT^{\hpp}_l(K)$ 
consists of minimized inertia elements at valuations 
$v$ with $v|_l=\bfw|_l$. 
\vskip3pt
Second, for every generalized quasi prime
$r$-divisor $\bfw$, recall the canonical exact sequence
\[
1\to\mT_\bfw\to\mZ_\bfw\to
\mZ_\bfw/\mT_\bfw=:\Pi^1_\Kbfw\to1,
\]
and recalling that for $\bfw\leq\bfv$, one has
$\mT_\bfw\subset\mT_\bfv$, $\mZ_\bfv\subset\mZ_\bfw$,
we endow $\Pi^1_\Kbfw=\mZ_\bfw/\mT_\bfw$ with
its \defi{minimized inertia/decomposition groups}
$\mT_{\bfv/\bfw}:=\mT_\bfv/\mT_\bfw
\subset\mZ_\bfv/\mT_\bfw=:\mZ_{\bfv/\bfw}
\subset\Pi^1_\Kbfw$, which in turn, give rise in the 
usual way to the \defi{total quasi decomposition graph} 
$\clG_{\clQ^{\rm tot}_\Kbfw}$ for $\Kbfw|k\bfw$.%
\footnote{\hhb2Recall that if $\chr(\Kbfw)=\ell$, the groups 
$\mT_{\bfv/\bfw}\subset\mZ_{\bfv/\bfw}\subset\Pi^1_\Kbfw$
are not Galois groups over $\Kbfw$.} 
\vskip2pt
Further, let
$\clT^{\hpp}_l(\Kr)\subseteq\clT^{\hpp}(\Kr)
\subset\Pi^1_\Kbfw$ be the images of 
$\clT^1_l(K)\cap\mZ_\bfw\subseteq\clT^1(K)\cap\mZ_\bfw$ 
under the canonical projection $\mZ_\bfw\to\Pi^1_\Kbfw$.
Then $\clT^{\hpp1}(\bullet)$ behaves 
functorially, in the sense that for $\bfw\leq\bfv$, one has: 
The image of $\clT^{\hpp1}(\Kbfw)\cap\mZ_{\bfv/\bfw}$ 
under $\mZ_{\bfv/\bfw}\to\Pi^1_\Kbfv$ equals
$\clT^{\hpp1}(\Kbfv)$. 
\vskip3pt
Finally, recall that a pro-$\ell$ abelian group $G$ 
endowed with a system of procyclic subgroups
$(T_\alpha)_\alpha$ is called \defi{complete 
curve like}, if there exists a system of generators 
$(\tau_\alpha)_\alpha$ with $\tau_\alpha\in T_\alpha$ 
such that letting $T\subseteq G$ be the closed 
subgroup of $G$ generated by $(\tau_\alpha)_\alpha$, 
the following hold:
\vskip3pt
\begin{itemize}[leftmargin=20pt]
\item[i)] $\prod_\alpha\tau_\alpha=1$ and this is the 
only profinite relation satisfied by $(\tau_\alpha)_\alpha$ in
$G$.\footnote{\hhb2This implies by definition that 
$\tau_\alpha\to1$ in $G$, thus every open subgroup 
of $G$ contains almost all $T_\alpha$.}
\vskip2pt
\item[ii)] The quotient $G/T$ is a finite $\lvZ_\ell$-module.
\end{itemize} 

\vskip3pt
A case of special interest is that of quasi prime
$r$-divisors $\bfw$, with $r=d-1$, where $d=\td Kk$,
thus $\td\Kbfw{k\bfw}=1$. Then 
$\bfw K/\bfw k\cong\lvZ^{d-1}$ and the following
hold: Let $(t_2,\dots,t_d)$ be a system of elements 
of $K$ such that $(\bfw t_2,\dots,\bfw t_d)$
define a basis of $\bfw K/\bfw k$, and 
$t_1\in\clO\tms_\bfw$ be such that its residue
$\oli t_1\in\Kbfw$ is a separable transcendence 
basis of $\Kbfw|k\bfw$. Then $(t_1,\dots,t_d)$
is a separable transcendence basis of $K|k$, and 
letting $\kone\subset K$
be the relative algebraic cloure of $k(t_2,\dots,t_d)$
in $K$, it follows that $\td K\kone=1$, and 
$\kone\bfw=k\bfw$. Hence we are in the situation
of~Section~4 from~[P5],~Theorem~4.1, thus 
we have: 
\begin{fact} 
\label{redtocurves}
In the above notations, the following hold:
\vskip3pt
\begin{itemize}[leftmargin=30pt]
\item[{\rm I)}] for every non-trivial element
$\sigma\in\clT^{\hpp1}(\Kr)$ there exists a 
unique quasi prime divisor $\bfv_\sigma>\bfw\,$ 
such that $\sigma\in\mT_\bfv$, thus the image of
$\sigma$ in $\Pi^1_\Kr$ lies in $\mT_{\bfv/\bfw}$.
\vskip2pt
\item[{\rm II)}] Let $(T_\alpha)_\alpha$ be a maximal 
system of distinct maximal cyclic subgroups of 
$\,\Pi^1_\Kr$ satisfying one of the following conditions:
\vskip3pt
\begin{itemize}[leftmargin=30pt]
\item[i)] $T_\alpha\subset\clT^1(\Kr)$ for 
each $\alpha$.
\vskip2pt
\item[ii)] $T_\alpha\subset\clT^1_l(\Kr)$ 
for each $\alpha$.
\end{itemize}
\vskip2pt
Then $\Pi^1_\Kr$ endowed with $(T_\alpha)_\alpha$ 
is complete curve like if and only if: 
\vskip3pt
\begin{itemize}[leftmargin=25pt]
\item[a)] $\kr$ is an algebraic closure of a finite field, 
provided~\hhb2{\rm i)} is satisfied.
\vskip2pt
\item[b)] $\lr=\kr$, provided~\hhb2{\rm ii)} is satisfied.
\end{itemize}
\vskip2pt
Moreover, if either~i), or~ii), is satisfied, then
$(T_\alpha)_\alpha$ is actually the set of the
minimized inertia groups $T_\alpha=\mT_{v_\alpha}$ 
at all the prime divisors 
$v_\alpha:=\bfv_\alpha/\bfw$ of $\Kbfw|k\bfw$. 
\end{itemize} 
\end{fact}
\vskip2pt
\noindent
We are now prepared to give the recipes $\eudim(\delta)$.
\vskip5pt
\noindent \ \
$\bullet$ The recipe $\eudim(0)$: 
\vskip5pt
We say that $K|k$ satisfies $\eudim(0)$ if and only if 
for every quasi prime divisor $\bfw$ of $K|k$ with
$\td\Kbfw{k\bfw}=\td K k-1$, one has: $\Pi^1_\Kbfw$ 
endowed with any maximal system of maximal pro-cyclic 
subgroups $(T_\alpha)_\alpha$ satisfying condition~i) 
above is complete curve like. 
\vskip5pt
\noindent \ \
$\bullet$ Given $\delta>0$, and the recipes 
$\eudim(0),\dots,\eudim(\delta-1)$, we define $\eudim(\delta)$ 
as follows:
\vskip5pt
First, if $\td K k \leq\delta$, we say that $\eudim(\delta)$ does
not hold for $K|k$. For the remaining discussion, 
we suppose that $\td K k >\delta$, and proceed as follows:
\vskip5pt
For $\delta'<\delta$, let $\clD_\delta(K)$ be the set of  
all the minimized quasi divisorial subgroups 
$\mT_\bfv\subset\mZ_\bfv$ of $\pia K$ such that 
$\Pi_\Kbfv^1$ endowed with the family
$(\mT_{\bfv'\hm2/\bfv})^{\phantom1}_{\bfv<\bfv'}$,
\underbar{does not} satisfy $\eudim(\delta')$ 
for any $\delta'<\delta$. We set 
$\Inr_\delta(K):=\cup_{\bfv\in\clD_\delta}\mT_\bfv\subset\pia K$, 
and notice that $\Inr_\delta(K)$ consists of minimized inertia
elements (by its mere definition), and it is closed under taking 
powers. Hence by \nmnm{P}{op}~[P5],~Introduction, Theorem~A, 
the topological closure $\oli\Inr_\delta(K)\subset\pia K$ consists 
of minimized inertia elements in $\pia K$, and it is closed under 
taking powers, because $\Inr_\delta(K)\subset\pia K$ was 
so. For $d:=\td K k $ and every quasi $(d-1)$-divisorial 
subgroup $\mT_\bfv\subset \mZ_\bfv$ of $\pia K$,
we set $\Inr_\bfv:=\oli\Inr_\delta(K)\cap \mZ_\bfv$, 
and note that $\Inr_\bfv$ consists of minimized inertia 
elements of $\pia K$ which are contained in $\mZ_\bfv$,  
it is topologically closed, and closed under taking powers. 
Thus the image $\pi_\bfv(\oli\Inr_\bfv)\subset\Pi_{\Kbfv}^1$ 
under the canonical projection 
$\pi_\bfv:\mZ_\bfv\to\mZ_\bfv/\mT_\bfv=\Pi_{\Kbfv}^1$
consists of minimized inertia elements, is topologically 
closed, and closed under taking powers. Finally, by 
Fact~\ref{redtocurves},~I), above, it follows that for every 
$\sigma\in\Inr^1_\bfv$ which has a nontrivial image
under $\mZ_\bfv\to\Pi^1_\Kbfv$, there 
exists a unique minimal valuation $\bfv_\sigma$ such that 
$\sigma$ is a minimized inertia element at $\bfv_\sigma$, 
and in particular, $\bfv_\sigma>\bfv$ by loc.cit. Hence 
$\mT_\bfv\subset\mT_{\bfv_\sigma}$, and moreover, 
since $\td \Kbfv {k\bfv}=1$, it follows that $\bfv_\sigma$ is 
a quasi prime $d$-divisor of $K|k$. Thus we have:
\begin{fact} 
\label{factbfvi}
Let $\bfv_i$, $i\in I_\bfv$, be the set of distinct
quasi prime divisors of $K|k$ which satisfy: $\bfv_i>\bfv$ 
and $\mT_{\bfv_i}$ contains some $\sigma\in\Inr^1_\bfv$ 
with $\sigma\not\in\mT_\bfv$. Then the following hold: 
\vskip2pt
\begin{itemize}
\item[{i)}] $\bfv_i$ are quasi prime $d$-divisors of $K|k$.
\vskip4pt
\item[{ii)}] Setting $\mT_{v_i}:=\mT_{\bfv_i/\bfv}\subset\Pi_\Kbfv^1$,
one has $\mT_{v_i}\cap\mT_{v'_i}=\{1\}$ for $\bfv_i\neq\bfv_{i'}$.
\vskip2pt
\item[{iii)}] $\Inr^1_\bfv\subset\cup_i \mT_{\bfv_i}$, thus the
image of $\Inr^1_\bfv$ under $\mZ_\bfv\to\Pi^1_\Kbfv$ 
equals $\cup_i \mT_{v_i}$.
\end{itemize}
\end{fact}
\begin{prop} 
\label{natureofk}
In the above notations suppose that $d:=\td K k >\delta$. 
The one has:
\vskip2pt
\begin{itemize}[leftmargin=30pt]
\item[{\rm1)}] Suppose $\clD_\delta(K)$ is empty. Then 
$\dim(k)<\delta$ and $\dim(k)$ is the maximal $\delta_k$ 
such that there exists some quasi prime divisor $\bfv$ whose 
$\Pi_\Kbfv^1$ endowed with $(\mT_{v_i})^{\phantom i}_{v_i}$ 
satisfies $\eudim(\delta_k)$. Further, a quasi prime divisor 
$\bfv$ of $K|k$ is a prime divisor if and only if $\Pi_\Kbfv^1$ 
endowed with its $(\mT_{v_i})^{\phantom I}_{v_i}$ satisfies 
$\eudim(\delta_k)$.
\vskip2pt
\item[{\rm2)}] Suppose $\clD_\delta(K)$ is non-empty. 
Then $\dim(k)\geq\delta$, and in the notations from 
Fact~\ref{factbfvi} above one has: $\,\dim(k)=\delta$ iff  
the quasi prime $(d-1)$-divisors $\bfv$ of $K|k$ satisfy:
\vskip2pt 
\begin{itemize}[leftmargin=30pt]
\item[{\rm i)}] There exist $\bfv$ such that $I^1_\bfv$ 
is non-empty.
\vskip2pt
\item[{\rm ii)}] If $I^1_\bfv$ is non-empty, then 
$\Pi_{\Kbfv}^1$ endowed with $(\mT_{v_i})^{\phantom I}_{v_i}$ 
is curve like.  
\end{itemize}
\vskip2pt
\item[] If the above conditions are satisfied, a quasi prime divisor 
$\bfv$ of $K|k$ is a prime divisor of $K|k$ if and only if 
$\Pi_\Kbfv^1$ endowed with $(\mT_{v_i})^{\phantom I}_{v_i}$
does not satisfy $\eudim(\delta')$ for any $\delta'<\delta$. 
\vskip2pt
\item[$(*)$] In particular, there exists a group theoretical recipe
to distinguish the divisorial groups $T_v\subset Z_v$ in $\pia K$
among all the quasi divisorial subgroups, thus to recover the set
of divisorial inertia $\Inrdiv(K)\subset\pia K$. Further, the recipes
to check $\,\td K k>\dim(k)$ and to recover
$\Inrdiv(K)$ from $\pic K$ are invariant under isomorphisms
$\Phi\in{\rm Isom}^{\rm c}(\pia K, \pia L)$. 
\end{itemize}
\end{prop}
\begin{proof}
To 1): First let $\dim(k)<\delta$. Then for all quasi prime 
divisors $\bfv$ of $K|k$ we have $\dim(k)\geq\dim(k\bfv)$,
hence $\td K k >\delta>\dim(k)\geq\dim(k\bfv)$. Thus 
taking into account that $\td \Kbfv{k\bfv}=\td K k -1\geq\delta$, 
we get: $\td \Kbfv{k\bfv}>\dim(k\bfv)$, and $\Pi_\Kbfv^1$ 
endowed with $(\mT_{v_i})^{\phantom I}_{v_i}$ satisfies
the recipe $\eudim(\delta')$ with $\delta'=\dim(k\bfv)<\delta$.
Thus $\bfv\not\in\clD_\delta(K)$, etc. Second, suppose that
$\clD_\delta(K)$ is empty. Equivalently, $\Pi_\Kbfv^1$ 
endowed with $(\mT_{v_i})^{\phantom I}_{v_i}$ satisfies
$\eudim(\delta')$ for some $\delta'<\delta$ for each
quasi prime divisor $\bfv$. Hence choosing $\bfv$ to be a
prime divisor of $K|k$ we get: $k\bfv=k$, and $\Kbfv|k$
satisfies $\eudim(\delta')$ for some $\delta'<\delta$.
Thus $\dim(k)=\delta'<\delta$. The description of the
divisors $\bfv$ of $K|k$ is clear, because $\bfv$ is a
prime divisor of $K|k$ iff $\bfv$ is trivial on 
$k$ iff $\dim(k\bfv)=\dim(k)$. 
\vskip3pt
To 2): First suppose that $\dim(k)=\delta$. Recall that a
quasi prime divisor $\bfv$ is a prime divisor iff $\bfv$ is
trivial on $k$ iff $\dim(k)=\dim(k\bfv)$. Therefore, if $\bfv$
is a prime divisor, then $\eudim(\delta')$ is not satisfied
by $\Pi_\Kbfv^1$ endowed with $(\mT_{v_i})^{\phantom I}_{v_i}$
for any $\delta'<\delta$. Second, if $\bfv$ is not a prime 
divisor, then $\bfv$ is not trivial on $k$, and reasoning as 
above, one gets that $\eudim(\delta')$ holds for $\Pi_\Kbfv^1$ 
endowed with $(\mT_{v_i})^{\phantom I}_{v_i}$ 
for $\delta'=\dim(k\bfv)$. Thus
finally it follows that $\clD_\delta(K)$ consists of exactly
all the prime divisors of $K|k$. Hence by \nmnm{P}{op}~[P2],
Introduction, Theorem~B, it follows that $\oli\Inr_\delta(K)$
consists of all the tame inertia at all the $k$-valuations of
$K|k$. Thus reasoning as in the proof of Proposition~3.5
of \nmnm{P}{op}~[P4], but using all the divisorial subgroups
instead of the quasi divisorial subgroup, and 
$\oli\Inr_\delta(K)$ instead of 
$\Inrtm(K)$ as in loc.cit., 
it follows that the {\it flags of divisorial subgroups\/} --as 
introduced in Definition~3.4 of loc.cit.-- can be recovered by 
a group theoretical recipe from $\pia K$ endowed with 
$\oli\Inr_\delta(K)$. Hence in the notations from 
Fact~\ref{factbfvi} above, one has the following: Let 
$\bfv$ be a quasi $(d-1)$-prime divisor of $K|k$. 
\vskip2pt
\begin{itemize}[leftmargin=20pt]
\item[{a)}] Suppose that $\bfv$ is trivial on $k$. Then by the 
discussion above, $(T_{v_i})_{v_i}$ is exactly the set of all 
the divisorial inertia subgroups in $\Pi^1_\Kbfv$. By applying 
Fact~\ref{redtocurves} above, one concludes 
that~$\pia{\Kbfv}$ endowed with $(T_{v_i})_{v_i}$ is curve like, etc.
\vskip2pt
\item[{b)}] Suppose that $\bfv$ is non-trivial on $k$. We claim that
the set $I_\bfv$ from Fact~\ref{factbfvi} above is empty. 
Indeed, since each non-trivial $\sigma\in\oli\Inr_\delta(K)$ 
is tame inertia element at some $k$-valuations of $K|k$, it 
follows that $\bfv_\sigma$ is a $k$-valuation. Thus one cannot
have $\bfv_\sigma>\bfv$, i.e., $\bfv_\sigma\not\in I_\bfv$.
\end{itemize}
Now suppose that $\dim(k)>\delta$. We show that there 
exist prime $(d-1)$-divisors $\bfv$ such that $I_\bfv$ is
not empty, but $\Pi^1_\Kbfv$ endowed with $(T_{v_i})_{v_i}$
is not curve like. Indeed, let $l\subset k$ be an algebraically
closed subfield with $\td k l =1$. Then $\dim(l)\geq\delta$,
and therefore, if $\bfv$ is a quasi divisor of $K|k$ which is
trivial on $l$, we have $\dim(k\bfv)\geq\dim(l)\geq\delta$.
Hence $\eudim(\delta')$ does not hold for $\Kbfv\md k\bfv$
for all $\delta'<\delta$. Therefore, $\bfv\in\clD_\delta$. We
then can proceed as in the proof of Proposition~4.4 of 
\nmnm{P}{op}~[P4], but
using all the quasi prime divisors of $K|k$ which are trivial
on $l$ instead of using all the (maximal) quasi prime divisors
of $K|k$. Namely let $\Inrtm_{\hp l}(K)$ be the set of the
inertia elements at valuations which are trivial on $l$. Then
by \nmnm{P}{op}~[P2], Introduction, Theorem~A, the set
$\Inrtm_{\hp l}(K)$ is closed in $\pia K$, and by loc.cit.\
Theorem~B, the set $\Inrtmqdiv_l(K)$ of tame inertia 
at the quasi divisors $\bfv$ which are trivial on $l$ is dense 
$\Inrtm_{\hp l}(K)$. On the other hand, by the discussion 
above, every quasi prime divisor $\bfv$ which is trivial 
on~$l$ lies in $\clD_\delta(K)$. Hence we conclude that 
$\Inrtmqdiv_l(K)\subseteq\Inr_\delta(K)$, and so, 
$\Inrtm_{\hp l}(K)$ is contained in the closure 
$\oli\Inr_\delta(K)$. In other words, if $\bfv$ is
any prime $(d-1)$-divisor $\bfv$ of $K|k$, and 
$\pi_\bfv: Z_\bfv\to\pia{\Kbfv}$ is the canonical
projection, it follows that $\pi_\bfv\big(\oli\Inr_\delta(K)\big)$
equals all the tame inertia at valuations which are trivial 
on $l$. Therefore, in the notations from Fact~\ref{factbfvi}
above, the set $I_\bfv$ is non-empty, and 
$(T_{v_i})^{\phantom I}_{v_i}$
consists of the inertia groups of all the quasi prime divisors 
of $\Kbfv\md k$ which are trivial on $l$. Since $l\subset k$ 
strictly, by Fact~\ref{redtocurves} it follows that 
$\pia{\Kbfv}$ endowed with all $(T_{v_i})_{v_i}$ is not 
curve like.
\vskip2pt
Finally, the last assertion~$(*)$ of the Proposition is an 
obvious consequence of the recipes described at~1) and~2).
\end{proof}
\vskip5pt
\noindent
B) {\it Concluding the proof of Theorem~\ref{appl1}\/}
\vskip7pt
By Proposition~\ref{natureofk}, there are group theoretical 
recipes to check whether $\td Kk >\dim(k)$ and if so, the
group theoretical recipes recover the divisorial inertia 
$\Inrdiv(K)$ from the group theoretical information encoded
in $\pic K$. Further, the
recipes to do so are invariant under isomorphisms 
$\Phi\in\Isom^{\rm c}(\pia K,\pia L)$. Conclude by applying 
Theorem~\ref{mainthm}.
%
%
%
%
%
%
%
%
\vskip20pt
\centerline{\Large  Appendix:} 
\vskip10pt
\centerline{\normalsize \bf ON THE ORDER OF THE REDUCTION OF POINTS}
\vskip3pt
\centerline{\normalsize \bf ON ABELIAN SCHEMES}
\vskip15pt
\centerline{\small PETER JOSSEN}
\vskip20pt
%
%
%
%
%
\newcommand{\abs}[1]{\vert#1\vert}
\newcommand{\angl}[1]{\langle #1\rangle}
\newcommand{\eps}{\varepsilon}
\newcommand{\tq}{\: | \:}
\newcommand{\pt}{\mbox{ for all }}
\newcommand{\qqet}{\qquad\mbox{and}\qquad}
\newcommand{\upet}{\textup{\'et}}

\newcommand{\Tell}{\lvT_\ell}
\newcommand{\Vell}{\textup{V}_{\!\ell}}

\newcommand{\IQ}{\lvQ}
\newcommand{\IZ}{\lvZ}

\newcommand{\fp}{\eup}
\newcommand{\fl}{\mathfrak l}

\newcommand{\HH}{{\rm H}}

\newcommand{\lto}{\xrightarrow{\quad}}
\newcommand{\lmapsto}{\longmapsto}

%
%
%
\newtheorem{theo}{Theorem}
\renewcommand{\thethm}{\arabic{thm}}
\newtheorem{lemm}[theo]{Lemma}
\newtheorem{propo}[theo]{Proposition}
\newtheorem{fapt}[theo]{Fact}

\setcounter{section}{1}
\setcounter{equation}{1}

\begin{par}
We fix an integral scheme $S$ of finite type over 
$\Spec\lvZ$, and a prime number $\ell$ different from 
the characteristic of the function field of $S$. Our goal 
is to establish the following
\end{par}

\begin{theo}\label{Thm:Main} {\it
Let $A$ be an abelian $S$-scheme, and $P\in A(S)$ 
be a point of infinite order. Then the set of closed points 
$s:\Spec(\kappa) \to S$ such that $\ell$ divides the 
order of the image $P_s\in A(\kappa)$ of $P$ has positive 
Dirichlet density in $S$, thus it is Zariski dense as well.\/}
\end{theo}

\begin{par}
In the special case where $S\subseteq\Spec\clO_k$,
with $\clO_k$ the ring of integers of a number field
$k$, the Theorem~\ref{Thm:Main} above was 
proven by \nmnm{P}{ink}, see~\cite{Pk},~Theorem~2.8
(and some arguments in our strategy are similar to his). 
The proof rests on two classical theorems: These are 
the general version of the Mordell--Weil theorem, stating 
that $A(S)$ is a finitely generated (abelian) group, 
and a generalization of Chebotarev's density theorem.
\end{par}
\vskip10pt
\noindent
{\bf Remarks/Basic Facts.} \ Since $\ell$ is different from the
characteristic, $S$ has an open dense subsets on which 
$\ell$ is invertible. Thus replacing $S$ by such an open
dense subset, without loss of generality, we can suppose
that $\ell$ is invertible on $S$. Let 
$\pi_1 := \pi_1^\upet(S,\overline \eta)$ be the \'etale
fundamental group of $S$.
We view $\IZ$ as the constant group $S$-scheme, and for
every abelian $S$-scheme $A$, we let $\Tell A$ be the 
$\ell$-adic Tate module of $A$ viewed as a $\pi_1$-module.
\begin{itemize}[leftmargin=25pt]
\item[a)]  Consider the complex $M$ of group 
$S$-schemes, viewed as \defi{1-motive} in the sense 
of Deligne:
\[ 
M = [u:\IZ\to A], \quad u(1)=P
\]
\item[b)] One associates with $M$, in a functorial way, 
a finitely generated free $\IZ_\ell$-module $\Tell M$ with 
continuous $\pi_1$-action, by setting $\Tell M =: 
\plim{\scriptscriptstyle e}M_{\ell^e}(\overline k)$, where
\[
M_{\ell^e}(\overline k) := 
  \{(Q,n)\in A(\overline k)\times \IZ \tq \ell^e Q
    =nP\}\hp\big/\hp\{(nP,\ell^e n)\in A(\overline k)\times \IZ\tq n\in\IZ\}
\]
Hence $M_{\ell^e}(\overline k)$ is the group of 
$\overline k$-points of a finite \'etale group $S$-scheme 
$M_{\ell^e}$ killed by $\ell^e$.
\vskip2pt
\item[c)] For later use, we notice that the elements of 
$\Tell M$ are represented by sequences $(P_i,n_i)_{i\in\lvN}$, 
with $P_i \in A(\overline k)$ and $n_i \in \IZ$, satisfying 
the three relations: $\ell^iP_i=n_iP$, $\ell P_i-P_{i\hbox{\hp-}1}=m_i P$, 
$n_i-n_{i\hbox{\hp-}1}=\ell^{i\hbox{\hp-}1}m_i$, 
for suitable $m_i\in\lvZ$. Two such sequences $(P_i,n_i)_i$, 
$(P'_i,n'_i)^{\phantom{'}}_i$ represent the same element of $\Tell M$ 
if there exists a sequence $(m_i)_i$ in $\IZ$ such that
\[
\ell^im_i = n_i-n'_i,\quad m_iP = P_i-P'_i \ \hbox{ for all } i\in\lvN\,.
\]
\item[d)]  Letting $\pi_1$ act trivially on $\lvZ_\ell$, 
one gets canonically an exact sequence of $\pi_1$-modules
\[
0 \to \Tell A \to \Tell M \to \IZ_\ell \to 0\,.
\]
In particular, $\Tell M$ is a free $\IZ_\ell$-module 
of rank $2g+1$, where $g=\dim_S(A)$.
\vskip2pt
\item[e)] Let $\pi_1 \to \GL(\Tell M)$ define $\Tell M$ as
a $\pi_1$-module, and consider the $\ell$-adic Lie groups
\[
L^M := \im\big(\pi_1 \to \GL(\Tell M)\big),\quad 
L^A := \im\big(\pi_1 \to \GL(\Tell A)\big).
\]
By restriction, $L^M\srjr L^A\!$, and set 
$L^M_A:=\ker(L^M\srjr L^A)= 
\{\sigma\in L^M \tq \sigma|_{\Tell A}=\id\}$. 
\vskip5pt
\item[f)] Finally, we notice that if $P \in A(S)$ is a 
torsion point, then the sequence of $\pi_1$--modules 
$0 \to \Tell A \to \Tell M \to \IZ_\ell \to 0$ splits after passing to 
$\IQ_\ell$--coefficients, and it follows that the group $L^M_A$ 
is trivial in that case. 
\noindent
The point is that the converse holds as well:

\end{itemize}
\vskip5pt
\begin{propo}\label{Pro:KummerInjection}
{\it If $P \in A(S)$ has infinite order, then the group $L^M_A$ 
is not trivial.\/}
\end{propo}
\begin{proof} This is a special case of 
\nmnm{J}{ossen}~[Jo],~Theorem~2.
\end{proof}
\vskip3pt
\begin{par}
Let $s \hra S$ be a closed point, and $G_{\kappa(s)}\to\pi_1$
be the embedding defined by some path $\oli s\to\oli\eta$.
The image $F_s\in\pi_1$ of the Frobenius element of
$G_{\kappa(s)}$ is called a \defi{Frobenius element 
over $S$}. We shall make use of the following version of 
Chebotarev's Density Theorem: 
\end{par}

\begin{theo}[Artin--Chebotarev]\label{Thm:ArtinChebo}
{\it The set of all Frobenius elements is dense in $\pi_1$.\/}
\end{theo}

\begin{par}
\noindent This follows in principle from~Theorem~7 
in~\cite{Serre63}, but see rather~\nmnm{H}{olschbach}~\cite{Ho} 
for a complete proof (of this and other related assertions). 
The deduction of our Theorem \ref{Thm:ArtinChebo} from 
\nmnm{S}{erre}'s~Theorem~7 in~\cite{Serre63} goes as 
follows: Let $\pi_1 \to G$ be a finite 
quotient of $\pi_1$, corresponding via the defining property 
of the fundamental group to a finite \'etale Galois cover 
$X$ of $S$. Let $R \subseteq G$ be any subset stable under 
conjugation. Then the Artin--Chebotarev Density Theorem 
for the scheme $S$ states that the set of closed points in $S$ 
whose Frobenius conjugacy class in $G$ is contained in $R$ 
has Dirichlet density $|R|/|G|$, and in particular it is Zariski
dense in $S$. Moreover, every element of $G$ lies in a 
Frobenius conjugacy class. This being true for all finite 
quotients of $\pi_1$, the statement of~Theorem~\ref{Thm:ArtinChebo} 
follows.
\end{par}

%
\begin{lemm}\label{Lem:FrobUndOrdnung}
{\it Let $s = \Spec\kappa\to S$ be a closed point and let 
$F_s \in \pi_1$ be a Frobenius element over $s$. The 
order of the image of $P$ in $A(\kappa)$ is prime to 
$\ell$ if and only if the homomorphism 
$(\Tell M)^{\angl{F_s}} \to \IZ_\ell$ is surjective.\/}
\end{lemm}

\begin{proof}
The order of $P$ in the finite group $A(\kappa)$ is 
prime to $\ell$ if and only if $P$ is $\ell$--divisible 
in $A(\kappa)$. From the description of elements 
of $\Tell M$ by sequences as indicated at~Remarks 
and~Basic~Facts,~c), above, it follows that this is the case 
if and only if $\Tell M$ contains an element fixed by 
$F_s$ which is mapped to $1\in \IZ_\ell$ by the 
canonical projection $\Tell M \to \IZ_\ell$.
\end{proof}

\begin{par}
We are done once we have shown that 
$\Sigma_P:=\{\sigma\in\pi_1\mid (\Tell M)^{\angl\sigma} 
\to \IZ_\ell \ \hbox{is \emph{not} surjective}\}$ 
contains a nonempty open subset of $\pi_1$, provided 
$P \in A(S)$ has infinite order. We can also work 
with the image of $\pi_1$ in $\GL(\Tell M)$ in place 
of $\pi_1$, which we denoted~by~$L^M$.
\end{par}
%
\begin{propo}\label{Pro:OffenNichtLeer}
{\it If $P \in A(S)$ has infinite order, $\Sigma_P$
contains a nonempty open subset of $L^M$.\/}
\end{propo}

\begin{proof}
Let $G$ be the subgroup of $\GL(\Tell M)$ consisting 
of those elements which leave invariant the subspace 
$\Tell A$ of $\Tell M$ and act trivially on the quotient 
$(\Tell M/\Tell A)=\IZ_\ell$. Relative to an appropriate 
$\IZ_\ell$-basis of $\Tell M$, the group $G$ consists 
of matrices of the form
$$
\left(\begin{array}{c|c} U & v\\ \hline 0 & 1 \end{array} 
 \right), \qquad U\in\GL(2g,\IZ_\ell),\quad v\in \IZ^{2g}
$$
where $g$ is the relative dimension of $A$ over $S$. 
The group $L^M$ is a closed subgroup of $G$. By 
Proposition \ref{Pro:KummerInjection}, there exist
$\sigma\in L^M_A$, $\sigma\neq1$. As matrices, 
elements $\sigma\neq1$ are of the form:
\[
\sigma = \left(\begin{array}{c|c} \id_{2g} & w\\ \hline 0 & 1 
\end{array} \right), \qquad w\in \IZ^{2g},\quad w\neq 0
\]
Let $N\geq 0$ be an integer such that 
$\ell^{\,\hbox{-}N} = \abs{\ell^N}_\ell < \max_i\abs{w_i}_\ell$, 
where $w_i$ are the coefficients of~$w$. We consider 
the subset $X$ of $G$ consisting of the matrices of the form
\[
\left(\begin{array}{c|c} \id_{2g}+\ell^NM & w+\ell^Nv\\ 
\hline 0 & 1 \end{array} \right), \qquad M\in
\mathrm{M}(2g\times 2g,\IZ_\ell),\quad v\in \IZ^{2g}
\]
This set is open and closed in $G$, hence $X\cap L^M$ 
is open and closed in $L^M$. Moreover, the intersection 
$X\cap L^M$ is not empty because it contains $\sigma$. 
We are done if we show that $X\cap L^M$ is contained 
in $\Sigma$, so let us show that for all $x\in X$, the map 
$(\Tell M)^{\angl x}\to\IZ_\ell$ is not surjective. Let 
$t\in (\Tell M)^{\angl x}$. We~claim that the image of $t$ 
in $\IZ_\ell$ lies in $\ell\,\IZ_\ell$. Indeed, set
$$
x = \left(\begin{array}{c|c} \id_{2g}+\ell^NM & w+\ell^Nv\\ 
\hline 0 & 1 \end{array} \right), \qquad t = 
\left(\begin{array}{c} t' \\ \lambda \end{array} \right),\quad 
t'\in \IZ_\ell^{2g},\quad \lambda\in\IZ_\ell
$$
We assume that $xt=x$, hence
$$0 = (x-\id_{2g+1})t = \left(\begin{array}{c} 
\ell^N(Mt'+ \lambda v) + \lambda w \\ 0 \end{array} \right)
$$
However, the equality $\ell^N(Mt'+ \lambda v) + \lambda w = 0$ 
can only hold if we have $\abs{\lambda}_\ell<1$, 
because of our choice of $N$. The image of $t$ in 
$\IZ_\ell$ is $\lambda$, and the last inequality shows 
$\lambda\in\ell\,\IZ_\ell$.
\end{proof}

\begin{proof}[Proof of Theorem \ref{Thm:Main}]
We just have to put the pieces together: If 
$P \in A(S)$ has infinite order, there exists 
by~Proposition~\ref{Pro:OffenNichtLeer} 
and~Theorem~\ref{Thm:ArtinChebo} a closed 
point $s:\Spec(\kappa) \to S$ and a Frobenius element 
$F_s \in \pi_1$ over $s$, such that 
$(\Tell M)^{\angl{F_s}} \to \IZ_\ell$ is not surjective. 
By~Lemma~\ref{Lem:FrobUndOrdnung} that means 
that the order of the image of $P$ in $A(\kappa)$ is 
divisible by $\ell$, so we are done. In fact, the above 
arguments show that the set $\Sigma_P$ is precisely 
the set of all the $s\in S$ such that the specialization 
$P_s$ of $P$ at $s$ has order divisible by $\ell$. 
On the other hand, with the notion of Dirichlet 
density (of closed points) as introduced
in~\nmnm{S}{erre}~\cite{Serre63}, it follows that
$\Sigma_P$ has positive Dirichlet density (which 
in principle can be explicitly given). 
\end{proof}
\vskip10pt

\vskip10pt
{\footnotesize
\sc
{\small Florian Pop}
\vskip0pt\noindent
Department of Mathematics
\vskip0pt\noindent
University of Pennsylvania
\vskip0pt\noindent
DRL, 209 S 33rd Street
\vskip0pt\noindent
Phila\-delphia, PA 19104,  USA
\vskip2pt\noindent
{\it E-mail address\/}: \ {\tt pop@math.upenn.edu}
\vskip2pt\noindent
{\it URL\/}: \ {\tt http://math.penn.edu/{\auau}pop}
\vskip10pt\noindent
{\small Peter Jossen}
\vskip0pt\noindent
Department of Mathematics
\vskip0pt\noindent
ETH Zurich 
\vskip0pt\noindent
8092 Zurich, 
Switzerland
\vskip5pt
\noindent
{\it E-mail address\/}: \ {\tt peter.jossen@gmail.com} 
\vskip2pt\noindent
{\it URL\/}: \ {\tt https://www.math.ethz.ch/the-department/people.html?u=jossenpe}
}

\end{document}